\documentclass[10pt]{amsart}

\usepackage{amssymb,amsfonts,amsmath,nccmath,amsthm}
\usepackage[all,arc]{xy}
\usepackage{tikz-cd}
\usepackage{tikz}
\usetikzlibrary{arrows,decorations}

\usepackage{enumitem}
\usepackage{mathrsfs}
\usepackage{multicol}
\usepackage{xcolor,colortbl}
\usepackage{graphics}
\usepackage{color}

\setlist[enumerate]{leftmargin=2.5em}
\setlist[itemize]{leftmargin=2.5em}

\usepackage
[pdfauthor={Catherine Hsu, Preston Wake, and Carl Wang-Erickson},
 pdftitle={Explicit non-Gorenstein R=T via rank bounds I: Deformation theory},
 bookmarks=false,hidelinks]
{hyperref}

\newcounter{qcounter}

\newcommand\define{\newcommand}

\define\rk{\mathrm{rk}}

\define\isoto{\xrightarrow{\sim}}
\define\onto{\twoheadrightarrow}

\DeclareMathOperator{\Spec}{Spec}

\newcommand{\ttmat}[4]{\left( \begin{array}{cc}
#1 & #2 \\
#3 & #4
\end{array}
\right)}

\newcommand{\Z}{\mathbb{Z}}
\newcommand{\Q}{\mathbb{Q}}

\newcommand{\F}{\mathbb{F}}

\newcommand{\fH}{\mathfrak{H}}

\newcommand{\m}{\mathfrak{m}}
\newcommand{\frt}{\mathfrak{t}}

\newcommand{\varep}{\varepsilon}

\newcommand{\Hom}{\mathrm{Hom}}
\newcommand{\Gal}{\mathrm{Gal}}

\newcommand{\Ext}{\mathrm{Ext}}

\newcommand{\End}{\mathrm{End}}

\newcommand{\Fr}{\mathrm{Fr}}

\newcommand{\lb}{{[\![}}
\newcommand{\rb}{{]\!]}}
\newcommand{\lp}{{(\!(}}
\newcommand{\rp}{{)\!)}}

\newcommand{\red}{\mathrm{red}}
\newcommand{\irr}{\mathrm{irr}}

\define\GL{{\mathrm{GL}}}

\define\kcyc{\kappa_{\mathrm{cyc}}}

\define{\Fitt}{\mathrm{Fitt}}
\define{\Ann}{\mathrm{Ann}}

\define\RGamma{{\mathrm{R}\Gamma}}

\newtheorem{thm}{Theorem}[subsection] 

\newtheorem{cor}[thm]{Corollary}
\newtheorem{prop}[thm]{Proposition}
\newtheorem{lem}[thm]{Lemma}
\newtheorem{sublem}[thm]{Sublemma}
\newtheorem{conj}[thm]{Conjecture}

\theoremstyle{definition}
\newtheorem{defn}[thm]{Definition}

\newtheorem{eg}[thm]{Example}

\newtheorem{assump}[thm]{Assumption}

\theoremstyle{remark}
\newtheorem{rem}[thm]{Remark}

\newcommand{\lra}{\longrightarrow}

\newcommand{\rinj}{\hookrightarrow}

\newcommand{\rsurj}{\twoheadrightarrow}

\DeclareMathOperator{\ad}{ad}

\newcommand{\bF}{\mathbb{F}}

\newcommand{\bT}{\mathbb{T}}

\newcommand{\Db}{{\bar D}}

\newcommand{\Ad}{{\mathrm{Ad}}}

\newcommand{\cG}{{\mathcal{G}}}

\newcommand{\barQ}{{\overline{\Q}}}
\newcommand{\Jm}{{J^{\min{}}}}

\newcommand{\fl}{{\mathrm{flat}}}

\newcommand{\Tr}{\mathrm{Tr}}
\newcommand{\bro}{{\bar\rho}}

\newcommand{\sm}[4]{\ensuremath{\big(\begin{smallmatrix}#1 & #2 \\ #3 & #4\end{smallmatrix}\big)}}

\newcommand{\univ}{\mathrm{univ}}
\newcommand{\up}[1]{^{(#1)}}
\newcommand{\upp}[1]{^{(#1)'}}

\usepackage{bbm}

\newcommand{\oQ}{\overline{\Q}}

\newcommand{\Eis}{\mathrm{Eis}}
\newcommand{\ep}{\epsilon}
\newcommand{\US}{\mathrm{US}}

\makeatletter
\let\c@equation\c@thm
\makeatother

\numberwithin{equation}{subsection}

\title[$R=\bT$ via rank bounds I]{Explicit non-Gorenstein $R=\bT$ via rank bounds I: Deformation theory}

\author{Catherine Hsu}
\address{Department of Mathematics \& Statistics, Swarthmore College,  Swarthmore, PA 19081, USA}
\email{chsu2@swarthmore.edu}

\author{Preston Wake}
\address{Department of Mathematics, Michigan State University, East Lansing, MI 48824, USA}
\email{wakepres@msu.edu}

\author{Carl Wang-Erickson}
\address{Department of Mathematics, University of Pittsburgh, Pittsburgh, PA 15260, USA}
\email{carl.wang-erickson@pitt.edu}

\begin{document}


\begin{abstract}
Ribet has proven remarkable results about non-optimal levels of residually reducible Galois representations. We focus on a non-optimal level $N$ that is the product of two distinct primes and where the Galois deformation ring is not expected to be Gorenstein. We prove a Galois-theoretic criterion for the deformation ring to be as small as possible---that is, for there to be a \emph{unique} newform of level $N$ with reducible residual representation. When this criterion is satisfied, we deduce an $R=\bT$ theorem.
\end{abstract}

\maketitle

\tableofcontents

\section{Introduction}

\subsection{Summary}
\label{subsec:intro summary}
We prove, under some hypotheses, an integral $R=\bT$ theorem for the mod-$p$ Galois representation $\bro = 1 \oplus \omega$. Here $\bT$ is the Hecke algebra acting on modular forms of weight $2$ and level $N=\ell_0\ell_1$, $p \geq 5$ is a prime number, $\omega$ is the mod-$p$ cyclotomic character, and $R$ is a level $N$ universal Galois pseudodeformation ring for $\bro$. We adopt the following conditions on $N$: 
\begin{enumerate}
    \item $\ell_0$ is a prime number with $\ell_0 \equiv 1 \pmod{p}$, and
    \item $\ell_1$ is a prime number with $\ell_1 \not \equiv 0, \pm 1 \pmod{p}$, such that $\ell_1$ is a $p$th power modulo $\ell_0$.
    \item there is a unique cuspform $f$ of level $\ell_0$ that is congruent to the Eisenstein series modulo $p$.
\end{enumerate}
By a theorem of Ribet \cite{ribet2010,ribet2015,yoo2019}, restated as Theorem \ref{thm: ribet} below, conditions (1) and (2)  imply that there is a newform of level $N$ with reducible residual Galois representation $\bro$, and condition (3) ensures that the space of oldforms is as small as possible.\footnote{Mazur's theorem \cite{mazur1978} implies that there is at least one such cuspform.}  
Moreover, under these conditions on $N$, the algebra $\bT$ is expected to be non-Gorenstein (and this is borne out computationally), and so we focus on this case because it is the simplest situation we can find that exhibits this non-Gorenstein behaviour. We intend that the methods developed here might serve as a prototype for more general residually reducible contexts.

Now, condition (3) is equivalent to the non-vanishing of an easily-computed numerical invariant called \emph{Merel's number}, due to a deep theorem of Merel \cite{merel1996} (see Remark \ref{P1: rem: Merel}). Our main result can be thought of as an analog of Merel's theorem at level $N$. Indeed, the standard techniques that are used to prove that the surjection $R \onto \bT$ is an isomorphism do not apply in our setting because $\bT$ is not a local complete intersection. Instead, we prove $R=\bT$ using rank bounds. The $\Z_p$-rank of $\bT$ is at least $3$: there is the Eisenstein series, the unique cuspform of level $\ell_0$, and at least one newform of level $N$. We define invariants $a\up1(\Fr_{\ell_1})$ and $\alpha^2 + \beta$ in $\F_p$, discussed more in Section \ref{subsub:dim le 3} below, which play the role of Merel's number in that they control whether or not the newform of level $N$ is unique.

\begin{thm}[{Theorem \ref{P1: thm: main with a1}}]
    \label{P1: thm: main initial}
    Let $p \geq 5$ and assume that the level $N$ satisfies conditions  \emph{(1)-(3)}. The $\F_p$-dimension of $R/pR$ is greater than $3$ if and only if both
    \begin{enumerate}[label=(\roman*)]
        \item $a\up1(\Fr_{\ell_1}) = 0$ and
        \item $\alpha^2+\beta=0$.
    \end{enumerate}
    Moreover if one of (i) or (ii) fails, then the map $R \onto \bT$ is an isomorphism, the $\Z_p$-rank of $\bT$ is $3$, and there is a unique newform of level $N$ that is congruent to the Eisenstein series modulo $p$.
\end{thm}

This method for proving $R=\bT$ is novel. It differs significantly from the method of the paper \cite{WWE5} in which $R$ is formulated, the surjection $R \rsurj \bT$ is established, and similar $R=\bT$ results for the representation $\bro$ and certain squarefree levels $N$ are proven. In \cite{WWE5}, the theorems rely on conditions designed to force the rings $R$ and $\bT$ to be local complete intersections. Then, the crux of the method of \cite{WWE5} is to verify Wiles's numerical criterion \cite[Appendix]{wiles1995}, which relies on the complete intersection property and only uses information about $R$ that corresponds to \emph{first-order} deformations of pseudorepresentations. First-order calculations are also used to give a Galois-deformation-theoretic proof of Merel's theorem in \cite{WWE3}.

To prove Theorem \ref{P1: thm: main initial}, first order deformations are no longer sufficient: they can be used to prove that $\dim_{\F_p} R/pR \geq 3$, but cannot give an upper bound. Instead, we show that $\dim_{\F_p} R/pR > 3$ if and only if certain \emph{second-order} deformations exist. We set up technology that links the existence of {second-order} deformations of pseudorepresentations to the vanishing of cup products and triple Massey products in Galois cohomology, deploying the framework of tangent and obstruction theory for pseudorepresentations developed by the third-named author in \cite{A-inf}. We extract from these products the numerical invariants appearing in Theorem \ref{P1: thm: main initial}. 
To establish the theorem, we prove that these invariants are the only obstructions to constructing the required second-order deformations.
In this sense, we carry out a fine-grained computation of $R/pR$ modulo the cube of the maximal ideal.

In the second paper \cite{part2} in this series, we interpret the vanishing of $a\up1(\Fr_{\ell_1})$ and $\alpha^2+\beta$ in terms of algebraic number theory and use this description to develop algorithms that determine whether or not these invariants vanish.
Specifically, we show that the condition $\alpha^2+\beta=0$ can be detected by the splitting behaviour of primes in an explicit three-step solvable extension of $\Q(\zeta_p)$ that has degree $p^4$.
We give an implementation of this algorithm in Sage \cite{SAGE} and compile data from computer experiments showing that the rank of $R$ is exactly $3$ whenever the rank of $\bT$ is. We regard the second paper as an important proof of feasibility and applicability of the framework for computing with $R$ that is developed in this paper.

\subsubsection{Toward $R=\bT$ beyond rank $3$}
One potential drawback of our main theorem is that it only establishes $R=\bT$ when the rank of $\bT$ is $3$. We believe that $R=\bT$ regardless of the rank of $\bT$, but proving this will require new results on both the Hecke and Galois sides. To see why, it is instructive to revisit the prime level case.

In the prime level case, Mazur \cite{mazur1978} originally raised the question about the arithmetic significance of the rank of the analogous Hecke algebra. Merel used modular symbols---that is, Hecke-theoretic techniques---to prove his criterion for when the rank is one. More recently, Lecouturier \cite{lecouturier2021} gave a Hecke-theoretic interpretation of the rank in general. Calegari and Emerton \cite{CE2005} used deformation theory---that is, Galois-theoretic techniques---to give a criterion for when the rank is one, and the work of the second and third-named authors \cite{WWE3} gave a Galois-theoretic interpretation of the rank in general. The fact that the two methods arrive at the same answer is closely related to a case of the equivariant main conjecture of Iwasawa theory (see \cite{wake2020eisenstein}).

In this paper, the starting point is a Hecke-side result: Ribet's proof that the rank of $\bT$ is at least $3$. We expect that there is an Hecke-side formula for the rank of $\bT$ in general, along the lines of \cite{lecouturier2021}. The techniques of this paper could be used to give a Galois-side formula for the dimension of $R/pR$ more generally. The fact that these formulas should give the same answer is an Iwasawa-theory-type phenomenon, but it is not part of any conjectural framework (as far as we are aware). We hope and expect that is part of a rich theory that has yet to be discovered.

\subsection{Setup}
Let $p \ge 5$ be a prime and let $\bro$ be the 2-dimensional pseudorepresentation induced by $\omega \oplus 1$, where $\omega:G_\Q \to \F_p^\times$ is the mod-$p$ cyclotomic character. For an integer $M$, we say that $\bro$ is \emph{modular of level $M$} if there is a (cuspidal) newform $f$ of weight $2$ and level $\Gamma_0(M)$ such that the residual pseudorepresentation of $f$ is $\bro$. For an irreducible residual representation, the question of which levels it is modular for (if any) is the subject of Serre's conjecture \cite{serre1987}, proven by Khare--Wintenberger \cite{KW2009-inv1} and of level-raising and level-lowering results of Ribet \cite{ribet1984, ribet1990}. For reducible residual representations, like $\bro$, the situation is much different. For example, in Mazur's landmark paper on the Eisenstein ideal \cite{mazur1978}, he proves that, for any prime $\ell$, $\bro$ is modular of level $\ell$ if and only if $\ell \equiv 1 \pmod{p}$. In particular, since $\bro$ is not modular of level $1$, there is no ``optimal level" for $\bro$ that is an absolute minimum with respect to divisibility.

Ribet \cite{ribet2010} (see also \cite{yoo2019}) initiated the study of level raising for $\bro$. He observed that, here too, the results are qualitatively very different from the residually irreducible case, as witnessed by the following result (which is a special case of what Ribet proved).
\begin{thm}[Ribet]
\label{thm: ribet}
    If $\ell_0$ is a prime such that $\ell_0 \equiv 1 \pmod{p}$ and $\ell_1 \not \equiv \pm 1$ is another prime, then $\bro$ is modular of level $\ell_0\ell_1$ if and only of $\ell_1$ is a $p$th power modulo $\ell_0$. 
\end{thm}
The key thing to note about this result is that, unlike in the residually irreducible case \cite{ribet1984}, the level-raising condition on the prime $\ell_1$ depends not just on $\ell_1$ and $\bro$, but also on $\ell_0$.

\begin{assump}
\label{P1: assump: main} 
Now, and for the rest of the paper, we specialize to level $N=\ell_0\ell_1$, where $\ell_0$ and $\ell_1$ are primes such that 
\begin{enumerate}
    \item $\ell_0 \equiv 1 \pmod{p}$,
    \item  $\ell_1 \not \equiv 0, \pm 1 \pmod{p}$ and $\ell_1$ is a $p$th power modulo $\ell_0$, and
    \item \label{item: ell_0 rank 1}there is a unique cusp form $f$ of level $\ell_0$ that is congruent to the Eisenstein series modulo $p$.
\end{enumerate}
By Ribet's Theorem \ref{thm: ribet}, (1) and (2) imply that $\bro$ is modular of level $N$.
\end{assump}

\begin{rem}
    \label{P1: rem: Merel}
The number of cusp forms of level $\ell_0$ that are congruent to the Eisenstein series is well understood \cite{merel1996,CE2005,lecouturier2021, WWE3} and there is a numerical equivalent to assumption \eqref{item: ell_0 rank 1} as follows. 
Let $\log_{\ell_0}\colon \F_{\ell_0}^\times \to \F_p$ be a surjective homomorphism (that is, a discrete logarithm). Then \emph{Merel's number} is the quantity
\[
\sum_{i=1}^{\frac{\ell_0-1}{2}} i\log_{\ell_0}(i) \in \F_p.
\]
By Merel's Theorem \cite[Th\'eor\`eme 2]{merel1996}, the assumption \eqref{item: ell_0 rank 1} is equivalent to Merel's number being non-zero.
\end{rem}

\subsubsection{The Hecke algebra and congruence with Eisenstein series}
\label{sssec: intro HA and ES}
Note that the trace of $\bro(\Fr_\ell)$, where $\Fr_\ell$ is an arithmetic Frobenius element $\Fr_\ell \in G_\Q$ at a prime $\ell \neq p$, equals $\ell+1 \in \F_p$, which is the reduction modulo $p$ of the eigenvalue $\ell+1$ of the $\ell$th Hecke operator $T_\ell$ on the Eisenstein series $E_2$ of weight 2 and level 1. (Although the form $E_2$ is non-holomorphic, it has a holomorphic stabilization to any level $M$ with $M>1$.) Hence, for an integer $M>1$, $\bro$ is modular of level $M$ if there is a newform $f$ of level $M$ such that for all $n$ prime to $M$, $a_n(f)$ is congruent to $a_n(E_2)$ modulo a prime above $p$. In particular, if $M$ is squarefree with $t\geq1$ prime divisors, there are $2^t-1$ Eisenstein series of level $M$, all of which are stabilizations of $E_2$. As such, when $M$ is not prime, we need to specify the eigenvalues of Hecke operators at primes dividing $M$, thereby selecting a single Eisenstein series of level $M$, before setting up a bijection between eigenforms and pseudorepresentations.

Now let $M=N=\ell_0\ell_1$. There is a 3-dimensional space of Eisenstein series of weight $2$ and level $\Gamma_0(N)$, all having $T_\ell$-eigenvalue $\ell+1$ for $\ell \nmid N$. As in the paper \cite{WWE5}, we choose a basis of eigenforms for the Atkin--Lehner involutions $w_{\ell_0}$ and $w_{\ell_1}$. The possible pairs of eigenvalues of the Eisenstein series under the Atkin--Lehner operators $(w_{\ell_0}, w_{\ell_1})$ are $(-1,-1), (-1,1), (1,-1)$. However, it is known that a level $\Gamma_0(N)$ newform that is congruent to an Eisenstein series must have $(w_{\ell_0}, w_{\ell_1})$-eigenvalues $(-1,-1)$; we therefore select that particular Eisenstein series, calling it $E_{2,N}$. 

Let $\bT'$ be the $\Z_p$-algebra acting on modular forms of weight 2 and level $\Gamma_0(N)$ with coefficients in $\Z_p$ that is generated by the operators $T_\ell$ for $\ell \nmid N$ along with the Atkin--Lehner involutions $w_{\ell_0}$ and $w_{\ell_1}$. Let $\bT$ be the completion of $\bT'$ at the maximal ideal generated by $p$ and the annihilator of $E_{2,N}$, and let $\bT^0$ be the largest quotient of $\bT$ that acts faithfully on cusp forms. By Ribet's Theorem \ref{thm: ribet}, the $\Z_p$-rank of $\bT$ is at least 3, accounting for the contributions of the eigensystems of 
\begin{itemize}
    \item the Eisenstein series $E_{2,N}$, valued in $\Z_p$,
    \item the unique stabilization to level $N$ of the $\Z_p$-valued cusp form of level $\ell_0$, specified in \eqref{item: ell_0 rank 1} above, that has Atkin--Lehner eigenvalues $(-1,-1)$,
    \item the newform of level $N$ arising from Ribet's theorem, which has $\Z_p$-rank at least $1$. 
\end{itemize}

\subsubsection{Residually reducible modularity lifting and imposing conditions at $N$} 
\label{sssec: intro level and US discussion}

By Ribet's Theorem \ref{thm: ribet}, we know that $\bro$ is modular of level $N$, so we can ask about modularity lifting. Let $R^\univ$ denote the universal pseudodeformation ring of $\bro$ ramified only at $Np$. Considering the Galois representations associated to modular forms, it is not too difficult to show that there is a surjective homomorphism
$R^\mathrm{univ} \onto \bT$ (see \cite[\S4.1]{WWE5}).

To formulate a modularity lifting theorem, we must then define a \emph{level $N$ quotient} $R_N$ of $R^\mathrm{univ}$ that parameterizes pseudodeformations that ``look modular of level $N$.'' We also write $R$ for $R_N$ because the level $N$ is fixed throughout the paper. The putative theorem is that the induced map
\[
R \onto \bT
\]
is an isomorphism.

For a deformation $\rho$ to ``look modular of level $N$,'' we want it to satisfy the following conditions. Such $\rho$ are exactly those parameterized by $R$. 
\begin{enumerate}
    \item $\det(\rho)=\kcyc$, the $p$-adic cyclotomic character \hfill (weight $2$)
    \item $\rho$ is finite-flat at $p$ \hfill (geometricity)
    \item $\rho$ is unramified or Steinberg at $\ell_0$ and $\ell_1$ \hfill (level $\Gamma_0(N)$)
\end{enumerate}
Condition (1) is easy to formulate for pseudorepresentations, but (2) and (3) are more involved. For condition (2), which is cohomological in nature, a robust theory was developed in \cite{WWE4}. Condition (3) is even more complex. Roughly, this is for two reasons: because the Steinberg representation is reducible but indecomposible, and because it involves $p$-integrally interpolating between two conditions, unramified and Steinberg, that do not overlap in characteristic 0. 

In \cite{WWE5}, a candidate definition of (3), called \emph{unramified-or-Steinberg}, is made. The rough idea of this definition is as follows. A two-dimensional representation $\rho$ is Steinberg at $\ell$ if there is an isomorphism 
\[
\rho|_\ell \sim \ttmat{\kcyc}{*}{0}{1}
\]
on the restriction $\rho|_\ell$ to a decomposition group at $\ell$.
This implies that, for all $\sigma$ and $\tau$ in the decomposition group, the expression
\begin{equation}
\label{eq:steinberg def}
    (\rho(\sigma)-\kcyc(\sigma))(\rho(\tau)-1)
\end{equation}
is zero. Indeed, the form of the Steinberg representation implies that \eqref{eq:steinberg def} is conjugate to a matrix product of the form 
\[
\ttmat{0}{*}{0}{*} \cdot \ttmat{*}{*}{0}{0} = \ttmat{0}{0}{0}{0}.
\]
On the other hand, if $\rho$ is unramified at $\ell$, then the expression \eqref{eq:steinberg def} may not be zero because the Frobenius eigenvalues of $\rho$ need not be $\ell$ and $1$. However, if $\sigma$ is in the inertia group at $\ell$, then \eqref{eq:steinberg def} is zero, simply because the term $\rho(\sigma)-\kcyc(\sigma)$ is zero. Similarly, if $\tau$ is in the inertia group, then $\rho(\tau)-1$ is zero, so $\eqref{eq:steinberg def}$ is zero. Hence, if $\rho$ is either unramified or Steinberg at $\ell$, then the expression $\eqref{eq:steinberg def}$ is zero for all pairs $(\sigma,\tau)$ in the decomposition group with at least one of $\sigma$ and $\tau$ in the inertia group. A pseudorepresentation is defined to be \emph{unramified-or-Steinberg at $\ell$} if its determinant character is unramified and it comes from a Cayley-Hamilton representation $\rho$ satisfying \eqref{eq:steinberg def} for all such pairs $(\sigma,\tau)$; this corrects an error in the definition of unramified-or-Steinberg in the second- and third-named authors' previous paper \cite[\S3.4]{WWE5}. See \S\ref{sssec: US} for more details of this correction.

As an initial check that the definition of unramified-or-Steinberg is reasonable, it is shown in \cite{WWE3} that there is a surjective homomorphism $R_M \rsurj \bT_M$; that is, Galois representations arising from modular forms of level $M$ are unramified-or-Steinberg at the primes dividing $M$. Moreover, several theorems in \cite{WWE5} establish that, in many cases at many squarefree levels $M$, this is the right definition of (3), in that $R_M \cong \bT_M$. However, in all of the cases of $R_M \cong \bT_M$ proved in \cite{WWE5}, the rings $R_M$ and $\bT_M$ are local complete intersection. One of the motivations for this paper is to provide evidence that the definition of unramified-or-Steinberg given in \cite{WWE5} and clarified in \S\ref{sssec: US} is the right one, even in more pathological cases.

\subsection{Main results: bounding the rank of $R$} 

Our main result shows that $R \cong \bT$ under certain, numerically verifiable conditions, thereby supplying evidence that $R \cong \bT$ in general.

Since $\bT$ is not a local complete intersection ring in general (in fact, we expect it never is, outside the cases discussed in \cite{WWE5}), we cannot use Wiles's numerical criterion \cite{wiles1995} to prove that $R \onto \bT$ is an isomorphism. Instead, we use a new strategy: we prove that
\[
\dim_{\F_p} R /p R \le \mathrm{rank}_{\Z_p} \bT.
\]
Because $R$ is $p$-adically separated, a separated version of Nakayama's lemma then implies that $R \onto \bT$ is an isomorphism.  As discussed above, we have made assumptions that ensure that $\rk_{\Z_p} \bT \geq 3$. Hence our goal is to find conditions under which $\dim_{\F_p} R /p R \le 3$, for this will imply that $R \cong \bT$.  

\subsubsection{Conditions for $\dim_{\mathbb{F}_p} R/pR \le 3$}
\label{subsub:dim le 3}
The papers \cite{CE2005,WWE3} also bound the dimension of a (pseudo)deformation ring in terms of number-theoretic data. However, the situation there is greatly simplified by the fact that the tangent space of the deformation ring is one-dimensional, so computing the dimension amounts to determining the degree to which the tangent vector deforms. 

To bound the dimension of $R/pR$, we follow the same basic strategy of \cite{CE2005,WWE3}, but we have to deal with the fact that the tangent space of $R/pR$ is two-dimensional. Roughly speaking, we find a basis of the tangent space consisting of an ``old reducible vector'' (coming from level $\ell_0$) and a ``new irreducible vector.'' Under condition (3) in Assumption \ref{P1: assump: main}, we show that the dimension of $R /p R$ is greater than $3$ if and only if the new vector deforms to second order.

To determine when the new vector deforms to second order, we start by explicitly describing it: as a pseudorepresentation with values in $\F_p[\epsilon]/(\epsilon^2)$, it is given by 
\[
D_1=\omega+1+\epsilon( b\up1 c\up1 + (\omega-1)a\up1),
\]
where 
\begin{itemize}
    \item $b\up1 \in Z^1(G_{\Q,Np},\F_p(1))$ is the Kummer cocycle associated to $\ell_1$
    \item the cocycle $c\up1 \in Z^1(G_{\Q,Np},\F_p(-1))$ is ramified only at $\ell_0$
    \item the cochain $a\up1 : G_{\Q,Np} \to \F_p$ satisfies $-d a\up1 = b\up1 \smile c\up1$. 
\end{itemize}
To make sense of this (and to explain the notation), we think of $D_1$ as the trace of a \emph{generalized matrix algebra} representation
\begin{equation}
\label{eq: rho1 intro}
\rho_1 = \ttmat{\omega(1+a\up1 \epsilon)}{b\up1}{\omega c\up1}{1+d\up1 \epsilon},
\end{equation}
where $d\up1=b\up1 c\up1 - a\up1$, and where the generalized matrix multiplication is given by usual matrix multiplication but where the product of the off-diagonal co-ordinates is multiplied by $\epsilon$ (see \S\ref{subsec: 1-reducible} below for a formal discussion of these generalized matrix algebras). To determine if $\rho_1$ deforms to second order, we write down a putative deformation
\begin{equation}
\label{P1:eq:rho2}
    \rho_2 = \ttmat{\omega(1+a\up1 \epsilon+a\up2 \epsilon^2)}{b\up1+b\up2\epsilon}{\omega (c\up1+c\up2 \epsilon)}{1+d\up1 \epsilon+d\up2 \epsilon^2}
\end{equation}
with $\epsilon^3=0$ and write down the conditions that the new cochains $a\up2, b\up2,c\up2, d\up2$ must satisfy for $\rho_2$ to define a map $R \to \F_p[\epsilon]/(\epsilon^3)$. We find that, in order for $\rho_2$ to exist as a generalized matrix algebra representation, we must have
\begin{itemize}
    \item $a\up1 |_{\ell_1} =0$
\end{itemize}
and that if any deformation $\rho_2$ exists, it can be arranged to satisfy
\begin{itemize}
    \item $a\up1|_{\ell_0} = \alpha c\up1|_{\ell_0}$ for some $\alpha \in \F_p$,
    \item $b\up2|_{\ell_0} = \beta c\up1|_{\ell_0}$ for some $\beta \in \F_p$,
\end{itemize}
where ``$(-)|_{\ell}$'' indicates restriction to the decomposition group at $\ell$. In addition, for $\rho_2$ to be unramified-or-Steinberg at $\ell_0$, we must also have
\begin{itemize}
    \item $\alpha^2+\beta=0$.
\end{itemize}
Although this construction depends on many choices, we show that the conditions $a\up1|_{\ell_1}=0$ and $\alpha^2+\beta=0$ are independent of the choices. Actually, in \S\ref{sec: invariant is canonical}, we show more: $\alpha^2+\beta$ arises from a canonical element of the 1-dimensional $\F_p$-vector space $\mu_p \otimes \mu_p$, where $\mu_p \subset \oQ^\times$ are the $p$th roots of unity.

\subsubsection{Main results} 
The proof of this paper's main theorem relies on showing that there exists $\rho_2$ as in \eqref{P1:eq:rho2} if and only if the square of the maximal ideal of $R/pR$ is non-zero. Since the maximal ideal can be generated by two elements, if it is square-zero, then we have $R/pR \simeq \F_p[x,y]/(x^2,xy,y^2)$, with $\F_p$-dimension $3$. The main result is Theorem \ref{P1: thm: main initial}, which we restate here for convenience.

\begin{thm}[{Theorem \ref{P1: thm: main with a1}}]
    \label{P1: thm: main intro}
    Let $p \geq 5$. The $\F_p$-dimension of $R/pR$ is greater than $3$ if and only if both
    \begin{enumerate}[label=(\roman*)]
        \item $a\up1\vert_{\ell_1}(\Fr_{\ell_1})= 0$ and
        \item $\alpha^2+\beta=0$,
    \end{enumerate}
    where $a\up1$ and $\alpha^2+\beta$ are as defined in Section \ref{subsub:dim le 3}. Moreover, if $\dim_{\F_p} R/pR = 3$, then $R$ is a free $\Z_p$-module of rank $3$ and the natural map
    \[
    R \onto \bT
    \]
    is an isomorphism.
\end{thm}

The conditions $(i)$ and $(ii)$ may at first appear to be unusual enough that this theorem is of no use whatsoever. However, in the sequel to this paper \cite{part2}, we translate the conditions $(i)$ and $(ii)$ into explicit statements about splitting behaviors of primes in certain nilpotent extensions of $\Q$. Moreover, we develop algorithms to effectively compute $(i)$ and $(ii)$ using algebraic number theory. We have executed these algorithms for small values of $p$, establishing the following 

\begin{thm}  
\label{thm:main2}
Let $p=5$ and $\ell_0=11$. Then for 
\[
\ell_1 =23,67,263,307,373,397,593,857,967,1013,
\] 
condition $(i)$ of Theorem \ref{P1: thm: main intro} holds, but condition $(ii)$ does not. In particular, for these values of $\ell_1$, the $\F_p$-dimension of $R/pR$ equals $3$ and $R\cong\bT$. 

For $\ell_1 =43,197,683,727$, conditions $(i)$ and $(ii)$ of Theorem \ref{P1: thm: main intro} both hold. Consequently, the $\F_p$-dimension of $R/pR$ exceeds $3$ for these values of $\ell_1$. 
\end{thm}

\begin{rem}
For the values of $p$ and $N$ where we found $\dim_{\F_p}R/pR > 3$, we also computed $\dim_{\F_p}\bT/p\bT>3$. This is consistent with the expectation that $R\cong\bT$. 
\end{rem}

To summarize Theorem \ref{thm:main2}, in all of the examples we computed, we find that one of the following cases occurs, witnessing the main Theorem \ref{P1: thm: main intro}. 
\begin{itemize}
    \item We compute in number field extensions and determine that both $(i)$ and $(ii)$ of Theorem \ref{P1: thm: main intro} are true. In addition, we \emph{independently} compute with modular symbols and determine that $\mathrm{rank}_{\Z_p}(\bT)\geq 4$. 
    \item We compute in number field extensions and determine that, of the conditions of Theorem \ref{P1: thm: main intro}, $(i)$ is true but $(ii)$ is false. In addition, we \emph{independently} compute that $\rk_{\Z_p} \bT = 3$. Theorem \ref{P1: thm: main intro} tells us that $R \cong \bT$ in this case. 
\end{itemize}

Both of these cases are consistent with the hypothesis that $R \cong \bT$ in general, even when $\dim_{\F_p} R/pR > 3$. This leads us to a broader 

\begin{conj}
For any prime $p$ and squarefree level $M$ as in Section \ref{sssec: intro level and US discussion}, we have $R_M \cong \bT_M$. 
\end{conj}

In other words, we conjecture that the unramified-or-Steinberg condition developed in \cite[\S3]{WWE5} fully captures the ``modular of level $M$" condition on Galois pseudorepresentations. More precisely, the conjecture decomposes into ``$R_M^\varep = \bT_M^\varep$'' as $\varep$ varies over sets of Atkin--Lehner eigenvalues, as in \cite[\S1.9.1]{WWE5}.  

\subsection{Organization of the paper}
In order to organize non-canonical choices in one place, the notion of \emph{pinning data} is set up in Definition \ref{P1: defn: pinning}. Section \ref{P1: sec: recollection of WWE5} consists of recollections from the antecedent paper \cite{WWE5} regarding the fundamental concepts described in the introduction above. All notation and definitions are present in this section in order to make it reasonably self-contained, while details and proofs are left to \cite{WWE5}. Section \ref{sec: additional arithmetic} continues with several lemmas and definitions in arithmetic and Galois representations that extend the content of Section \ref{P1: sec: recollection of WWE5}, going beyond what appears in \cite{WWE5}. Section \ref{sec: explicit first-order} sets up the first-order deformation $\rho_1$ of \eqref{eq: rho1 intro} above. Section \ref{sec: computation of R} produces an explicit formula for $R/pR$ up to second order, and Section \ref{sec: R and Galois} applies this in order to prove the ``only if'' direction of the main Theorem \ref{P1: thm: main intro}. Section \ref{sec: proof of main theorem} proves the other logical direction by constructing by hand a level $N$ deformation $\rho_2$ of $\rho_1$ as in \eqref{P1:eq:rho2}. Section \ref{sec: invariant is canonical} proves that the invariant $\alpha^2+\beta$ is canonical by showing that the pinning data does not affect it. 

\subsection{Acknowledgements}

The first-named author would like to thank the University of Bristol and the Heilbronn Institute for Mathematical Research for its partial support of this project. The second-named author was supported in part by NSF grant DMS-1901867 and NSF CAREER grant DMS-2337830. The third-named author was supported in part by Simons Foundation award 846912 and NSF grant DMS-2401384, and would like to thank the Department of Mathematics of Imperial College London for its partial support of this project from its Mathematics Platform Grant. We also thank John Cremona for several helpful conversations about the computational aspects of this project. This research was supported in part by the University of Pittsburgh Center for Research Computing and Swarthmore College through the computing resources provided. Specifically, this work used the H2P cluster at the University of Pittsburgh, which is supported by NSF award number OAC-2117681.

\subsection{Notation and conventions}
\label{subsec: notation conventions}
For a group $G$, write $C^\bullet(G,-)$ for the complex of continuous, inhomogeneous $G$-cochains, and $H^i(G,-)$, $Z^i(G,-)$ and $B^i(G,-)$ for its cohomology, cocycles and coboundaries. Let $\RGamma(G,-)$ denote the class of $C^\bullet(G,-)$ in the derived category. Let $x \mapsto [x]$ denote the quotient map $Z^i(G,-) \to H^i(G,-)$. Let $\smile$ denote the cup product on $C^\bullet(G,-)$ and $\cup$ for the induced map on $H^*(G,-)$.

When $R=\Z[1/Np]$ or $R=\Q_q$ for a prime $q$, we use $C^\bullet(R,-)$ as an abbreviation for $C^\bullet(G,-)$ where $G$ is the \'etale fundamental group of $\Spec(R)$, and similarly for
$H^i(R,-)$, $Z^i(R,-)$, $B^i(R,-)$, and $\RGamma(R,-)$.

We fix an algebraic closure $\oQ$ of $\Q$. We work with the maximal subextension $\oQ \supset \Q_S \supset \Q$ that is ramified only at the places dividing $S=Np\infty$, and let $G_{\Q,Np} := \Gal(\Q_S/\Q)$.

For each prime number $q$, let $\oQ_q/\Q_q$ be an algebraic closure and let $G_q := \Gal(\oQ_q/\Q_q)$. Let $I_q \subset G_q$ be the inertia group and let $I_q^\mathrm{tame}$ be the tame quotient. When $q \neq p$, let $\gamma_q \in I_q$ denote a lift along $I_q \rsurj I_q^\mathrm{tame}$ of a topological generator of $I_q^\mathrm{tame}$. 

Let $\mu_p \subset \oQ^\times$ denote the subgroup of $p$th roots of unity, and let $\omega: G_{\Q,Np} \to \F_p^\times$ denote the mod-$p$ cyclotomic character. For $n \in \Z$, let $\F_p(n)$ denote the $\F_p[G_{\Q,Np}]$-module $\F_p$ with $G_{\Q,Np}$ acting by $\omega^n$.

Several of our constructions will depend in subtle ways on additional choices we call \emph{pinning data}. In the end (\S\ref{sec: invariant is canonical}), we will show that the invariant $\alpha^2 + \beta$ of Theorem \ref{P1: thm: main intro} is independent of the pinning data.

\begin{defn}
    \label{P1: defn: pinning} 
    The following choices constitute \emph{pinning data}:
    \begin{itemize}
        \item for each $q \in \{\ell_0, \ell_1, p\}$, an embedding $\oQ \rinj \oQ_q$,
        \item a primitive $p$th root of unity $\zeta_p \in \oQ$,
        \item for $i=0,1$, a $p$th root $\ell_i^{1/p} \in \oQ$ of $\ell_i$, such that, if possible, the image of $\ell_1^{1/p}$ in $\oQ_p$, under the fixed embedding, is in $\Q_p$. (See Lemma \ref{P1: lem: ram of ell1 at p} for a discussion of when this is possible.)
    \end{itemize}
\end{defn}

Notice that the choice of pinning data naturally induce the following further choices of 
\begin{itemize}
    \item for each prime $q$ dividing $Np$, a decomposition subgroup of $q$ in $G_{\Q,Np}$ and an isomorphism between this subgroup and $G_q$, and
    \item for each $n \in \Z$, isomorphisms $\F_p(n) \isoto \mu_p^{\otimes n}$.
\end{itemize}
We use these data to identify $\F_p(n)$ with $\mu_p^{\otimes n}$ and, for each prime $q$ dividing $Np$, $G_q$ as a subgroup of $G_{\Q,Np}$ and we let 
\[
C^i(\Z[1/Np],-) \xrightarrow{|_q} C^i(\Q_q,-), \quad f \mapsto f|_q
\]
denote the restriction map. We use the same notation $|_q$ for the induced map on cohomology, cocycles, and coboundaries.

\section{Recollection of pseudodeformation theory}
\label{P1: sec: recollection of WWE5} 

Throughout this manuscript, we retain the conventions and terminology of the preceding work \cite{WWE5} of the second-named and third-named authors. In this section, we summarize these items for the convenience of the reader, specializing them to the particular level $N=\ell_0\ell_1$ and Atkin--Lehner eigenvalues $\varep = (-1,-1)$ specified in \S\ref{sssec: intro HA and ES}. Note that since we fix this choice of Atkin--Lehner signs throughout this paper, we omit the superscript $\varep$ found in the notation throughout \cite{WWE5}.  

Nothing new is proven in this section. Those readers who have some familiarity with the ideas of \cite{WWE5} can safely skip this section on first reading, and refer back when necessary.

\subsection{Modular forms}
\label{P1: subsec: modular forms}

As in \cite[\S2.1]{WWE5}, we recall the following Hecke algebras and modular forms of weight 2. 

Let $\fH_N$ denote the Hecke algebra generated (over $\Z$) by the action of the Hecke operators
\begin{align}
\label{eq: Hecke operators}
\begin{split}
& T_q, \text{ for } q \nmid N \text{ prime, and}\\
& w_\ell, \text{ for } \ell \mid N \text{ prime},
\end{split}
\end{align}
on modular forms of weight 2 and level $\Gamma_0(N)$. Here $T_q$ is the standard unramified Hecke operator, while $w_\ell$ is the Atkin--Lehner involution at $\ell$. It is well known that $\fH_N$ is commutative, reduced, and free of finite rank as a $\Z$-module. 

As remarked in \S\ref{sssec: intro HA and ES}, the space $\Eis_2(\Gamma_0(N))$ of Eisenstein series of weight 2 and level $\Gamma_0(N)$ is 3-dimensional, and our choice of $(w_{\ell_0}, w_{\ell_1})$-eigenvalues $\varep = (-1,-1)$ specifies a unique normalized Hecke eigenform $E_{2,N}$. It has $T_q$-eigenvalue $q+1$ for all primes $q \nmid N$, and the constant term of its $q$-expansion at infinity is 
\[
a_0(E_{2,N}) =\frac{1}{2}\zeta(-1) \prod_{\ell \mid N} (\ell -1)= -\frac{1}{24}\prod_{\ell \mid N} (\ell -1).
\]

Now we define the Hecke algebras and Eisenstein ideals that are our primary object of study, measuring congruences of Hecke eigenvalues between $E_{2,N}$ and cusp forms. 

\begin{itemize}
    \item Let $\bT$ denote the completion of $\fH_N$ at its maximal ideal $(p, \Ann_{\fH_N}(E_{2,N}))$. Its residue field is $\F_p$, because $\fH_N/\Ann_{\fH_N}(E_{2,N}) \cong \Z$. 
    \item Let $\bT^0$ be the cuspidal quotient of $\bT$. 
    \item Let $I:= \Ann_{\fH_N}(E_{2,N}) \cdot \bT$, which we call the \emph{Eisenstein ideal}. We have $\bT/I \cong \Z_p$. 
    \item Let $I^0$ denote the image of $I$ in $\bT^0$. 
    \item Ohta \cite[Theorem 3.1.3]{ohta2014} has proved that
\[
\bT^0/I^0 \cong \Z_p/a_0(E_{2,N})\Z_p. 
\]
\end{itemize}
We call (the $p$-part of) $a_0(E_{2,N})$ \emph{the congruence number} for congruences (modulo $p$) of Hecke eigenvalues between $E_{2,N}$ and cusp forms. Our assumption that $\ell_0 \equiv 1 \pmod{p}$ implies that $\bT^0/I^0 \neq 0$, which is equivalent to $\bT^0 \neq 0$. 
    
    Let $M_2(N; \Z_p)_\Eis$ denote the module of modular forms of weight 2 and level $\Gamma_0(N)$ with coefficients in $\Z_p$, subject to the condition that their Hecke eigensystem under the Hecke operators of \eqref{eq: Hecke operators} are congruent modulo $p$ to that of the Eisenstein series $E_{2,N}$. 
    Let $S_2(N; \Z_p)_\Eis$ denote the submodule of $M_2(N; \Z_p)_\Eis$ consisting of cusp forms.
    We have perfect pairings 
    \begin{equation}
    \label{eq: Hecke pairings} 
    M_2(N; \Z_p)_\Eis \times \bT \to \Z_p, \qquad S_2(N; \Z_p)_\Eis \times \bT_N^{0} \to \Z_p.
    \end{equation}
    Under the usual Fourier expansion of a modular form $f(z) = \sum_{n \geq 0} a_n(f)q^n \in M_2(N; \Z_p)_\Eis$, the pairing is given by $(f,T) \mapsto a_1(Tf)$. 

    In particular, these pairings specialize to a bijection between normalized Hecke eigenforms in $M_2(N;\Z_p)_\Eis$ (resp.\ $S_2(N;\Z_p)_\Eis$) and homomorphisms $\bT\to\oQ_p$ (resp.\ $\bT^0 \to \oQ_p$) that encode their eigensystems. 

We will also require the Eisenstein-congruent Hecke algebras of weight 2 and level $\ell_0$ with Atkin--Lehner sign $-1$, denoted $\bT_{\ell_0}$, along with its cuspidal quotient $\bT_{\ell_0}^0$. This $\bT_{\ell_0}^0$ is the Hecke algebra studied by Mazur in \cite{mazur1978}. There are natural surjections $\bT \rsurj \bT_{\ell_0}$ and $\bT^0 \rsurj \bT_{\ell_0}^0$, because a choice of Atkin--Lehner signature at level $N$ designates a stabilization of level $\ell_0$ forms to level $N$. 

In light of \eqref{eq: Hecke pairings} and the fact that each of the spaces of modular forms has a basis of Hecke eigenvectors, we have the well known 
\begin{lem}
\label{lem: T are flat and reduced}
The Hecke algebras $\bT$, $\bT^0$, $\bT_{\ell_0}$, and $\bT_{\ell_0}^0$ are reduced and, as $\Z_p$-modules, finitely generated and flat. 
\end{lem}

\subsection{Galois deformation theory}
\label{subsec: Galois def theory}

The main technical feature of \cite{WWE5} was the development of theory of Galois representations adequate to characterize the Galois representations associated to $M_2(N;\Z_p)_\Eis$. In particular, while $\bT$ interpolates the Hecke eigensystems, interpolating the associated Galois representations presents technical issues addressed in \cite[\S3]{WWE5}. 

The key new notion presented there is the  \emph{unramified-or-Steinberg} condition on 2-dimensional  pseudorepresentations of $G_\ell$, which combine over all $\ell \mid N$ to a global unramified-or-Steinberg condition. Because we view this paper as a test of these notions in a more difficult setting (where $\bT$ is not Gorenstein), we carefully recall this notion. Also, since the global unramified-or-Steinberg condition involves the \emph{finite-flat} geometricity condition on representations of $G_p$, we recall that theory as well.

\subsubsection{Background on pseudodeformations}

We will presume that the reader is familiar with the theory of pseudorepresentations, as developed by Chenevier \cite{chen2014}. This is summarized in \cite[\S3.1]{WWE5}, and we recall fundamental notions here. All of our pseudorepresentations are 2-dimensional. 

Let $A$ be a commutative ring. We write $D : E \to A$ for a pseudorepresentation, which includes the implication that $E$ is an $A$-algebra (not necessarily commutative). The data represented by this notation consists of functions
\[
D_B : E \otimes_A B \to B
\]
associated functorially to commutative $A$-algebras $B$. 

When $H$ is a group, we write $D : H \to A$ as shorthand for a pseudorepresentation $D: A[H] \to A$. A pseudorepresentation $D : E \to A$ is characterized by its induced characteristic polynomial functions, which in the present 2-dimensional case are the two functions
\[
\Tr_D : E \to A \quad \text{ and } \quad {\det}_D : E \to A.
\]

When the source and target of a pseudorepresentation $D$ have a topology, $D$ is considered continuous when $\Tr_D$ and ${\det}_D$ are continuous. When $H$ is a profinite group and $A$ is a profinite ring, we will presume that a pseudorepresentation $D : H \to A$ is continuous from $A[H]$ to $A$ without further comment. 

\subsubsection{Cayley--Hamilton representations and GMA representations}
\label{sssec: CH reps and GMA reps}

While a pseudorepresentation $D : G \to A$ may not arise from a 2-dimensional representation of $G$ over $A$, it is well-understood how to broaden the category of representations to remedy this. This broader category consists of \emph{Cayley--Hamilton representations} of $G$. It is fibered over the category of pseudorepresentations and has universal objects. In this section, we overview the theory of Cayley--Hamilton representations, referring to \cite[\S3]{WWE5} for details. We also point out that the Cayley--Hamilton representations we work with in this paper admit the structure of generalized matrix algebras (``GMAs''). 

Let $A$ denote a commutative ring.
\begin{itemize}
    \item We say that a pseudorepresentation $D : E \to A$ is \emph{Cayley--Hamilton} if, for every commutative $A$-algebra $B$ and every element $\gamma \in E \otimes_A B$, $\gamma$ satisfies its $B$-valued characteristic polynomial $X^2 - \Tr_D(\gamma) X + \det_D(\gamma) \in B[X]$. 
    \item A \emph{Cayley--Hamilton algebra over $A$} is a pair $(E, D : E \to A)$, where $D$ is a Cayley--Hamilton pseudorepresentation. 
    \item An \emph{$A$-valued Cayley--Hamilton representation of $G$} is a tuple $(\rho: G \to E^\times, E, D: E \to A)$, where $(E, D)$ is a Cayley--Hamilton algebra over $A$ and $\rho$ is a group homomorphism. 
    \item The \emph{induced pseudorepresentation} of a Cayley--Hamilton representation 
    \[
    (\rho, E, D : E\to A)
    \]
    of $G$, written $\psi(\rho)$, is the $A$-valued pseudorepresentation of $G$ determined by the composition $D \circ \rho$. 
\end{itemize}

A \emph{generalized matrix algebra over $A$}, or ``$A$-GMA'' for short, is a Cayley--Hamilton algebra over $A$ with extra data. We confine our discussion to 2-by-2 GMAs. 

\begin{itemize}
    \item The data for a ($2\times 2$)-GMA over $A$ consists of two $A$-modules $B$ and $C$ together with an $A$-module map $m: B \otimes_A C \to A$ such that the two maps 
\[
B \otimes_A C \otimes_A B \to B \otimes_A A \to B \quad \text{and} \quad B \otimes_A C \otimes_A B \to A \otimes_A B \to B
\]
coincide, and similarly the two maps $C \otimes_A B \otimes_A C \to C$ coincide. We make an $A$-algebra $\ttmat{A}{B}{C}{A}$ using the rule for $2\times 2$-matrix multiplication. 
\item We think of a \emph{GMA structure} on a Cayley--Hamilton algebra as the idempotents $\sm{1}{0}{0}{0}$ and $\sm{0}{0}{0}{1}$ in the above decomposition. 
\item When $A$ is a Henselian local ring and a Cayley--Hamilton algebra $E$ over $A$ is finitely generated as an $A$-module (which will always be true in our applications, and is actually equivalent to being finitely generated as an $A$-algebra), its $A$-GMA structures are inner-isomorphic \cite[Lem.\ 5.6.8]{WWE1}. 
\item When a Cayley--Hamilton representation $(\rho, E, D : E\to A)$ of $G$ has its Cayley--Hamilton algebra $E$ equipped with the structure of an $A$-GMA, it is known as a \emph{GMA representation}. 
\end{itemize}

\subsubsection{Deformation theory of pseudorepresentations} 
\label{sssec: dt of ps} 

The functorial basis for deformation theory of pseudorepresentations is rather straightforward in \cite{chen2014}.  What is less straightforward is the approach to applying representation-theoretic conditions on pseudorepresentations that are most naturally formatted for representations. The main idea for this, developed systematically in \cite{WWE4}, is to say that a pseudorepresentation satisfies a condition $\mathcal{C}$ when some Cayley--Hamilton representation inducing it satisfies $\mathcal{C}$. In this section, we overview these deformation-theoretic concepts, first specializing to the particular pseudorepresentation that we will deform. 

\begin{itemize}
    \item Let $\omega : G_\Q \to \F_p^\times$ denote the mod-$p$ cyclotomic character, which factors through $G_{\Q,Np}$. It is the reduction modulo $p$ of the $p$-adic cyclotomic character that we denote by $\kappa : G_\Q \to \Z_p^\times$. 
    \item Let $\Db : G_{\Q,Np} \to \F_p$ denote the pseudorepresentation $\psi(\omega \oplus 1)$ of $G_{\Q,Np}$. 
    \item When $A$ is a commutative local ring with residue field $\F_p$, we say that $D : G_{\Q,Np} \to A$ \emph{deforms} $\Db$ if the composite pseudorepresentation $G_{\Q,Np} \to A \rsurj \F_p$ equals $\Db$. 
    \item Let $R_\Db$ denote the universal pseudodeformation ring of $\Db$. By \cite[Proposition E]{chen2014}, $R_\Db$ is Noetherian, which means that there is a universal pseudodeformation $D^u_\Db: G_{\Q,Np} \to R_\Db$. 
\end{itemize}

Now we bring Cayley--Hamilton representations into the deformation theory of pseudorepresentations. 

\begin{itemize}
    \item When $A$ is local with residue field $\F$ and $\Db : G \to \F$ is a pseudorepresentation, we say that an $A$-valued Cayley--Hamilton representation $(\rho, E, D)$ of $G$ is \emph{over $\Db$} when the pseudorepresentation $D \circ \rho: G \to A$ deforms $\Db$. 
    \item There exists a universal Cayley--Hamilton representation of $G_{\Q,Np}$ over $\Db$, valued in the universal pseudodeformation ring $R_\Db$. It is written
    \[
    (\rho^u : G_{\Q,Np} \to (E_\Db^u)^\times, E_\Db^u, D_{E_\Db}^u : E_\Db^u \to R_\Db^u).
    \]
    \item Because $\Db$ is multiplicity-free---that is, its associated semi-simple representation $\omega \oplus 1$ over $\F_p$ has non-isomorphic simple summands---it is known that any Cayley--Hamilton representation of $G_{\Q,Np}$ over $\Db$ admits the structure of a GMA representation. (See \cite[Theorem 3.2.2]{WWE4} for more details.)
\end{itemize}

\subsubsection{The unramified-or-Steinberg condition, correcting an error in \cite{WWE5}}
\label{sssec: US}

We now review and correct the \emph{unramified-or-Steinberg} condition that was introduced in \cite[\S3]{WWE5}. In \cite[\S3]{WWE5}, this condition is called the ``unramified-or-$\varep$-Steinberg condition'' or ``$\US_N^\varep$ condition'', to allow for arbitrary choice of Atkin--Lehner signs $\varep = (\varep_\ell)_{\ell \mid N}$ indexed by the prime divisors of $N$. In this paper, we only consider negative Atkin--Lehner signs, so we suppress the $\varep$ from our notation outside of \S\ref{sssec: US}. 

However, the second- and third-named authors appreciate the opportunity to correct an error in \cite{WWE5} in the formulation of $\US_N^\varep$ for general $\varep$, which we do in Definition \ref{defn: US local}. See Remark \ref{rem: WWE5 results same} for justification that the results of \cite{WWE5} still hold with precisely this correction to the definition. 

The definition of $\US_N^\varep$ is motivated by the forms of Galois representations of modular forms at decomposition groups, as we now recall. When $\ell \neq p$, it is known that Galois representations $\rho_f : G_{\Q,Np} \to \GL_2(\oQ_p)$ arising from a Hecke eigenform (for the Hecke operators of \eqref{eq: Hecke operators}) in $M_2(\Gamma_0(N))$ have the following form after restriction to a decomposition group:
\begin{itemize}
    \item $\rho_f \vert_{I_\ell}$ is non-trivial if and only if $f$ is new at $\ell$. In other words, $\rho_f\vert_{\ell}$ is unramified if and only if either $f$ is old at $\ell$ or $\ell\,\nmid\,N$. 
    \item If $f$ is new at $\ell$ and its $w_\ell$-eigenvalue is $\varep_\ell$, then $\rho_f \vert_{\ell} : G_\ell \to \GL_2(\oQ_p)$ has the form
    \begin{equation}
    \label{eq: Galois Steinberg form}
    \rho_f \vert_{\ell} \simeq \lambda(-\varep_\ell) \otimes 
     \ttmat{\kappa}{\tilde b_\ell}{0}{1},
    \end{equation}
    where $\lambda(\nu)$ is the unramified character of $G_\ell$ sending $\Fr_\ell \mapsto \nu$ and $\tilde b_\ell : G_\ell \to \oQ_p(1)$ is an element of $Z^1(\Q_\ell, \oQ_p(1))$ inducing a non-trivial cohomology class in $H^1(\Q_\ell, \oQ_p(1))$. By Kummer theory, this cohomology class is unique up to scalar, and consequently the $\rho_f\vert_\ell$ is uniquely prescribed up to isomorphism. 
\end{itemize}

In either case, if $f$ has $w_\ell$-eigenvalue $\varepsilon_\ell$, then the expression
\[
(\rho_f(\sigma) - \lambda(-\varep_\ell)\kappa(\sigma)) \cdot 
    (\rho_f(\tau) - \lambda(-\varep_\ell)(\tau))
\]
is zero for all $\sigma, \tau \in G_\ell$ with at least one of $\sigma$ and $\tau \in I_\ell$. Indeed, if $f$ is old at $\ell$, then $\rho_f$ is unramified and one of the two factors in the expression is zero. Otherwise, $\rho_f|_\ell$ has the form \eqref{eq: Galois Steinberg form}, and the expression is of the form
\[
\ttmat{0}{\ast}{0}{\ast} \cdot \ttmat{\ast}{\ast}{0}{0},
\]
and any such product is zero. This motivates the following definition.
\begin{defn}[Correction of {\cite[Defn.\ 3.4.1]{WWE5}}]
    \label{defn: US local}
    A Cayley--Hamilton representation $(\rho : G_{\ell} \to E, E, D_E : E \to A)$ over $\Db\vert_{\ell}$ is \emph{unramified-or-$\varep_\ell$-Steinberg at $\ell$} (or $\US_{\ell}^{\varep_\ell}$) provided that
    \begin{enumerate}
        \item the determinant of $\rho$, $\det \rho := {\det}_{D_E} \circ \rho : G_\ell \to A^\times$, is unramified
        \item the following identity 
        \begin{equation}
    \label{eq: US local ep}
        (\rho(\sigma) - \lambda(-\varep_\ell)\kappa(\sigma)) \cdot 
    (\rho(\tau) - \lambda(-\varep_\ell)(\tau)) = 0
    \end{equation}
    holds for all $(\sigma, \tau) \in G_{\ell} \times I_{\ell} \cup I_{\ell} \times G_{\ell}$. 
    \end{enumerate}
\end{defn}

\begin{rem}
    For the rest of this paper, we will only use the case that $\varep_\ell = -1$. So we will treat the identity \eqref{eq: US local ep} as 
    \begin{equation}
    \label{eq: US local}
        (\rho(\sigma) - \kappa(\sigma)) \cdot 
    (\rho(\tau) - 1) = 0
    \end{equation}
    for all $(\sigma, \tau) \in G_{\ell} \times I_{\ell} \cup I_{\ell} \times G_{\ell}$.
\end{rem}

\begin{rem}
    \label{rem: WWE5 results same}
    The definition of unramified-or-Steinberg in [WWE21, Defn. 3.4.1] assumes only condition (2) and there is an incorrect lemma [WWE21, Lem. 3.4.4] that claims that (2) implies (1). This difference does not affect the results of [WWE21] because the definition is only applied there to pseudorepresentations that satisfy (1) anyway. It is not true in general that (2) implies (1), as the following example shows. The mistake in the proof of \cite[Lem.\ 3.4.4]{WWE5} is the appeal to \cite[Lem.\ 2.7(iv)]{chen2014} when $A$ is a general ring, when in fact the reference requires $A$ to be a field. 
\end{rem}

\begin{eg}
    \label{eg: counterexample}
    Let $A = \F_p[\ep_1] := \F_p[\ep]/\ep^2$, assume $p \mid (\ell-1)$, and let $\rho = \lambda(-\varep_\ell) \otimes (\chi \oplus \chi)$ be the 2-dimensional diagonal representation, where $\chi : G_\ell \to A^\times$ satisfies $\chi(\Fr_\ell) = 1$ and $\chi(\gamma_\ell) = 1+\ep$ (here $\gamma_\ell \in I_\ell$ projects to a pro-generator of the tame inertia quotient). Then, because $\rho(\gamma_\ell)$ is unipotent---that is, $\rho(\gamma_\ell)$ satisfies $(\rho(\gamma_\ell)-1)^2 = 0$---it follows that $\rho$ satisfies $\US_\ell^{\varep_\ell}$. But its pseudorepresentation $\psi(\rho)$ is non-trivial on $I_\ell$; moreover, each of characteristic polynomial coefficients $\Tr \rho$ and $\det \rho$ comprising $\psi(\rho)$ are non-trivial on $I_\ell$. 
\end{eg}

The following lemma is a correction of  \cite[Lem.\ 3.4.4]{WWE5}.

\begin{lem}
    Let $\ell \neq p$ and let $(\rho : G_{\ell} \to E, E, D_E : E \to A)$ be a Cayley--Hamilton representation over $\Db$ satisfying $\mathrm{US}_\ell^{\varep_\ell}$ as in Definition \ref{defn: US local}. Assume that $2 \in A^\times$. Then $\psi(\rho)\vert_{I_\ell} = \psi(1 \oplus 1)$. 
\end{lem}

\begin{proof}
    Let $\tau \in I_\ell$. Definition \ref{defn: US local} implies that $\det(\rho(\tau))=1$ and $(\rho(\tau)-1)^2=0$, so $\rho(\tau)^2=2\rho(\tau)-1$. It remains to show that $\Tr(\rho(\tau))=2$. By the pseudorepresentation identity in \cite[Lem.\ 1.9(b)]{chen2014},
    \[
    2\det(\rho(\tau)) = \Tr(\rho(\tau))^2-\Tr(\rho(\tau)^2).
    \]
    Since $\det(\rho(\tau))=1$ and $\rho(\tau)^2=2\rho(\tau)-1$, this implies
    \[
    2=\Tr(\rho(\tau))^2-2\Tr(\rho(\tau))+\Tr(1).
    \]
    But $\Tr(1)=2$, so
    \[
    \Tr(\rho(\tau))(\Tr(\rho(\tau))-2) = 0.
    \]
    The reduction of $\Tr(\rho(\tau))$ modulo the maximal ideal of $A$ is $\Tr_\Db(\tau)=2$, so $\Tr(\rho(\tau))$ is a unit in $A$, and the previous equation implies $\Tr(\rho(\tau))=2$, as desired. 
\end{proof}

\subsubsection{The finite-flat condition}\label{P1: subsubsec: finite-flat}

Since the modular forms we work with have weight 2 and no level at $p$, the corresponding $p$-adic representations of $G_p$ should satisfy the finite-flat condition. 
\begin{defn}
    We say that an action of $G_p$ on a finite cardinality $\Z_p$-module $M$ is \emph{finite-flat} provided that there exists a finite-flat group scheme $\cG/\Z_p$ and an isomorphism of $\Z_p[G_p]$-modules $M \simeq \cG(\oQ_p)$. 
\end{defn}

Ramakrishna \cite{ramakrishna1993} determined how to apply the finite-flat condition to deformations of Galois representations. The crucial formal property that the finite-flat condition satisfies is that it is \emph{stable}, meaning that when $M$ is a finite-flat $\Z_p[G_p]$-module, then all of its subquotients are also finite-flat; and that if a finite number of $\Z_p[G_p]$-modules $M_i$ are finite-flat, then so is the direct sum $\bigoplus_i M_i$. 

Because not all pseudorepresentations arise from Galois representations as characteristic polynomials, it is non-trivial to impose the finite-flat condition on pseudorepresentations. This problem has been addressed in \cite{WWE4}, using a formalism that works for any stable condition. It relies on the fact that every pseudorepresentation arises from a Cayley--Hamilton representation. 
\begin{defn}
\label{defn: finite-flat}
We call a Cayley--Hamilton representation $\rho : G_p \to E$ \emph{finite-flat} if the $\Z_p[G]$-module $E$, where the action of $G_p$ on $E$ is given by $\rho$ composed with the left regular action of $E$ on $E$, is an inverse limit of finite-flat $\Z_p[G]$-modules. We call a pseudorepresentation $D : G_p \to A$ \emph{finite-flat} if it arises as the induced pseudorepresentation $\psi(\rho)$ of a Cayley--Hamilton representation $\rho$ that is finite-flat. 
\end{defn}

In \cite{WWE4}, it is proved that any stable condition cuts out a universal Cayley--Hamilton representation over any residual pseudorepresentation $\Db : G_p \to \F$, and that the coefficient ring of this Cayley--Hamilton representation is the universal finite-flat pseudodeformation ring of $\Db$. In particular, this result includes the implication that the finite-flat condition on pseudorepresentation cuts out a quotient $R_\Db \rsurj R_\Db^\mathrm{flat}$ of the universal pseudodeformation ring; in other words, the finite-flat condition is a Zariski-closed condition on pseudorepresentations. 

We have the following result about finite-flat representations over the residual pseudorepresentation $\Db\vert_p$. 
\begin{prop}
\label{prop: finite-flat implies upper tri}
For any finite-flat Cayley--Hamilton representation $\rho$ of $G_p$ over $\Db\vert_p : G_p \to \F_p$, with coefficient ring $A$, there exist unique characters $\theta_i : G_p \to A^\times$, $i=1,2$, and a GMA structure with respect to which it has the form 
\begin{equation}
\label{eq: generic finite-flat extension}
\rho \simeq \ttmat{\kappa \theta_1}{*}{0}{\theta_2}. 
\end{equation}
The characters $\theta_i$ are residually trivial and unramified. 
\end{prop}

\begin{proof}
See \cite[\S3.5]{WWE5}. 
\end{proof}

However, the finite-flat condition is more strict than merely having this form: in addition to the unramified condition on $\theta_i$, there is a restriction on the extension denoted ``$\ast$'', cutting out an $A$-submodule 
\[
\Ext_{A[G_p]}^1(\theta_2, \theta_1(1))^\mathrm{flat} \subset\Ext_{A[G_p]}^1(\theta_2, \theta_1(1))
\]
consisting of finite-flat extensions of $\theta_2$ by $\kappa \theta_1$. We will especially be interested in the case where $A = \F_p$ and the $\theta_i$ are trivial. In that case, since $\omega = (\kappa \mod p)$ lifts to $G_{\Q,Np}$, we construct 
\[
\Ext_{\F_p[G_{\Q,Np}]}^1(\F_p, \F_p(1))^\mathrm{flat} \subset\Ext_{\F_p[G_{\Q,Np}]}^1(\F_p, \F_p(1))
\]
consisting of those $G_{\Q,Np}$-extensions of $\F_p$ by $\F_p(1)$ that are finite-flat when restricted to $G_p$. 

Later we will have use for the determination of this finite-flat subspace more generally, over $\Q_{p^i}$, which denotes the unique degree $i$ unramified extension of $\Q_p$. Let $H_i := \Gal(\oQ_p/\Q_{p^i})$, so $H_1 = G_p$. Let $\Z_{p^i}$ denote the ring of integers of $\Q_{p^i}$. 

\begin{lem}[Local Kummer theory]
    \label{P1: lem: local Kummer finite-flat}
    Under the canonical isomorphism 
    \[
    \Ext_{\F_p[H_i]}^1(\F_p, \mu_p) \cong H^1(\Q_{p^i}, \mu_p) \cong \Q_{p^i}^\times/(\Q_{p^i}^\times)^p
    \]
    and the decomposition 
    \[
    \Q_{p^i}^\times/(\Q_{p^i}^\times)^p \cong \langle p \rangle \oplus \Z_{p^i}^\times/(\Z_{p^i}^\times)^p,
    \]
    the flat subspace $\Ext_{\F_p[H_i]}^1(\F_p, \mu_p)^\mathrm{flat}$ maps to $\Z_{p^i}^\times/(\Z_{p^i}^\times)^p$. In particular, when $i=1$, we have the $\F_p$-basis $\{p, 1+p\}$ of $\Q_p^\times/(\Q_p^\times)^p$, and the subspace $\Ext_{\F_p[G_p]}^1(\F_p, \mu_p)^\mathrm{flat}$ corresponds with the subspace $\langle 1+p\rangle$. 
\end{lem}

\begin{proof}
This is well known; see, for example, \cite[Prop.\ 2.2]{schoof2012}. 
\end{proof}

\begin{lem}[Global Kummer theory]
\label{lem: global Kummer theory}
\leavevmode
\begin{enumerate}
    \item The subspace 
\[
\Ext^1_{\F_p\lb G_{\Q,Np}\rb}(\F_p, \mu_p)^\mathrm{flat} \subset \Ext^1_{\F_p\lb G_{\Q,Np}\rb}(\F_p, \mu_p) 
\]
has basis $\{\ell_0,\ell_1\}$ under the canonical isomorphisms
\[
\Ext^1_{\F_p\lb G_{\Q,Np}\rb}(\F_p, \mu_p) \cong H^1(\Z[1/Np], \mu_p) \cong \Z[1/Np]^\times/(\Z[1/Np]^\times)^p.
\]
\item The natural map 
\[
\Ext^1_{\F_p\lb G_{\Q,Np}\rb}(\F_p, \F_p(1)) \lra \Ext^1_{\F_p\lb G_p\rb}(\F_p, \F_p(1))
\]
has image containing a complement of $\Ext^1_{\F_p\lb G_p\rb}(\F_p, \F_p(1))^\mathrm{flat}$. The image of the element $p \in \Z[1/Np]^\times/(\Z[1/Np]^\times)^p$ spans this complement. 
\end{enumerate}
\end{lem}

\begin{proof}
Parts (1) and (2) follow directly from Lemma \ref{P1: lem: local Kummer finite-flat} and the fact that $\{p, \ell_0, \ell_1\}$ is a basis for $\Z[1/Np]^\times/(\Z[1/Np]^\times)^p$. 
\end{proof}

Here is a method to verify finite-flatness of GMA-representations in practice.

\begin{lem}
\label{lem: GMA property} 
Let $\rho: G_p \to E$ be a Cayley--Hamilton representation with coefficient ring $A$. Suppose that $S \subset E$ be a subalgebra containing $\rho(G_p)$, and let $V$ be a faithful $S$-module. If the $G_p$-action on $V$ induced by $\rho$ is finite-flat, then $\rho$ is finite-flat.
\end{lem}
\begin{proof}
This is a slight generalization of the argument of the second paragraph of the proof of \cite[Lem.\ 7.1.9]{WWE5}. 
\end{proof}

We will also require the delicate use of a few standard and fundamental facts about lifts of group representations and the unobstructedness of finite-flat lifts, which we collect in the following two statements. We state these in less than their maximal generality, fitting our purposes. 

\begin{lem}
\label{lem: basic def theory} 
Let $G$ be a profinite group, let $\eta: G \to \GL_2(\F_p)$ be a representation, and let $s : (A',\m_{A'}) \rsurj (A,\m_A)$ be a surjection of local Artinian $\F_p$-algebras such that $\m_{A'} \cdot \ker s = 0$. Let $\eta_A$ be a lift of $\eta$ over $A \rsurj A/\m_A = \F_p$. 
\begin{enumerate}
    \item If the set of lifts of $\eta_A$ over $A' \rsurj A$ is non-empty, then it is a torsor over the group 
\[
Z^1(G, \Ad(\eta)) \otimes_{\F_p} \ker s 
\]
under addition of coordinates.
\item If $A = \F_p$, then this torsor is canonically isomorphic to $Z^1(G, \Ad(\eta)) \otimes_{\F_p} \ker s$ due to the base point given by the trivial lift $\rho \otimes_{\F_p} A$ of $\rho$ to $A$. 
\item Conjugation of $\rho_{A'}$ by $x \in \ker (\GL_2(A') \rsurj \GL_2(A))$, which is canonically isomorphic to $C^0(G,\Ad(\rho)) \otimes_{\F_p} \ker s$, amounts to acting by coboundary $dx \in B^1(G, \Ad(\rho)) \otimes_{\F_p} \ker s$ on $\rho_{A'}$ (via the torsor structure of (1)). 
\item If $\eta_A$ has constant determinant (that is, $\det \eta_A = \det \eta$ under $\F_p^\times \rinj A^\times)$), then the set of constant determinant lifts of $\eta_A$ over $s$ is non-empty if and only if the set of (unrestricted) lifts is non-empty; and if it is non-empty, it is a torsor over the group
\[
Z^1(G, \Ad^0(\eta)) \otimes_{\F_p} \ker s
\]
under addition of coordinates. 
\end{enumerate}
\end{lem}
Here ``addition of coordinates'' on $\rho_{A'}$ means that we add to the function $\rho_{A'} : G \to \GL_2(A')$ the function $G \to M_{d}(\ker s) \subset \GL_2(A')$ given by an element of $Z^1(G, \Ad(\eta)) \otimes_{\F_p} \ker s$.

\begin{prop}
\label{prop: smoothness}
Let $\eta : G \to \GL_2(\F_p)$ be a finite-flat representation. Let $s : (A',\m_{A'}) \rsurj (A,\m_A)$ be a surjection of local Artinian $\F_p$-algebras such that $\m_{A'} \cdot \ker s = 0$. Let $\eta_A$ be a finite-flat lift of $\eta$ over $A \rsurj A/\m_A = \F_p$. 
\begin{enumerate}
    \item The set of finite-flat lifts of $\eta_A$ over $s$ is \underline{non-empty}, and admits the structure of a torsor over the group
    \[
    Z^1(G, \Ad(\eta))^\mathrm{flat} \otimes_{\F_p} \ker s,
    \]
    where $Z^1(G,\Ad(\eta))^\mathrm{flat} \subset Z^1(G,\Ad(\eta))$ is a sub-vector space that contains $B^1(G,\Ad(\eta))$. 
    \item In particular, if $A = \F_p$, then this torsor is non-empty and canonically isomorphic to $Z^1(\Q_p, \Ad(\eta)) \otimes_{\F_p} \ker s$.
    \item The analogue of Lemma \ref{lem: basic def theory}(3) holds for finite-flat representations.
    \item The analogue of Lemma \ref{lem: basic def theory}(4) holds for finite-flat representations, with the addition that the set of constant determinant finite-flat lifts is non-empty. 
\end{enumerate}
\end{prop}

\begin{proof}
The non-emptiness of the set of finite-flat lifts can be found in \cite[Lem.\ 2.4.1]{CHT2008}, for example. The remaining claims can be deduced from Lemma \ref{lem: basic def theory} using \cite[Prop.\ C.4.1]{WWE3}. 
\end{proof}

\subsubsection{The global unramified-or-Steinberg condition}
\label{sssec: global US}

By combining the local conditions, we arrive at the global condition $\US_N$. 

\begin{defn}
    \label{defn: US global}
    Let $\rho$ be a Cayley--Hamilton representation over $\Db : G_{\Q,Np} \to \F_p$. We say that $\rho$ is \emph{unramified-or-Steinberg of level $N$}, or $\US_N$, when 
    \begin{enumerate}
        \item for all $\ell \mid N$, $\rho\vert_\ell$ is $\US_{\ell}$, and 
        \item $\rho\vert_p$ is \emph{finite-flat} in the sense of Definition \ref{defn: finite-flat}. 
    \end{enumerate}
    When $D : G_{\Q,Np} \to A$ is a deformation of $\Db : G_{\Q,Np} \to \F_p$, we say that $D$ is $\US_N$ if there exists some $A$-valued Cayley--Hamilton representation $\rho$ over $\Db$ such that $\rho$ is $\US_N$ and $D = \psi(\rho)$. 
\end{defn}

We fix notation for the universal objects satisfying $\US_N$, which were produced in \cite[\S3]{WWE5}. 

\begin{defn}
\label{defn: universal objects}
\hfill
\begin{itemize}
    \item Let $R$ denote the universal pseudodeformation ring of $\Db$ satisfying the $\US_N$ condition. It admits a natural surjection $R_\Db \rsurj R$. 
    
    \item Likewise, there exists a universal $\US_N$ Cayley--Hamilton representation of $G_{\Q,Np}$ over $\Db$, denoted
    \[
   (\rho_N : G_{\Q,Np} \to E^\times, E, D_E : E \to R)
    \]
    and inducing $D_N : G_{\Q,Np} \to R$, the universal $\US_N$ deformation of $\Db$. 
    \item We fix a $R^u_\Db$-GMA structure on the universal Cayley--Hamilton algebra $E^u_\Db$ over $\Db$, which induces a GMA structure on all of the Cayley--Hamilton algebras receiving a map from $E_\Db^u$ due to its universal property. In particular, we get a $R$-GMA structure on the universal $\US_N$ Cayley--Hamilton representation $(\rho_N, E, D_E)$ of $G_{\Q,Np}$ over $\Db$, and write its matrix coordinates as 
    \[
    E \cong \ttmat{R}{B}{C}{R}. 
    \]
    For $\gamma \in G_{\Q,Np}$, we write 
    \[
    \ttmat{a_\gamma}{b_\gamma}{c_\gamma}{d_\gamma} 
    \]
    for its image in $E$ under $\rho_N$. Letting $\m \subset R$ denote the maximal ideal, we may and do assume that the GMA structure on $E_\Db$ has been chosen such that 
    \[
    (a \mod{\m}) = \omega \quad \text{ and } \quad (d \mod{\m}) = 1
    \]
    as homomorphisms $G_{\Q,Np} \to \F_p^\times$. 
    \item We will also occasionally refer to $R_{\ell_0}$ as the universal pseudodeformation of $\Db$ satisfying the (global) $\US_{\ell_0}$ condition (with Atkin--Lehner sign $-1$ at $\ell_0$). There is a natural surjection $R \rsurj R_{\ell_0}$. 
\end{itemize}
\end{defn}

Having completed these constructions, the crucial application is that we can interpolate over $\bT$ the Galois pseudorepresentations induced by the representations $\rho_f : G_{\Q,Np} \to \GL_2(\oQ_p)$ associated to normalized Hecke eigenforms $f \in M_2(N; \Z_p)_\Eis$.  
\begin{prop}[{\cite[Prop.\ 4.1.1]{WWE5}}]
    \label{prop: R to T}
    We have a surjection $R \rsurj \bT$ characterized by sending traces of Frobenius elements $\Tr_{D_N}(\Fr_q) \in R$ for primes $q \nmid Np$ to the Hecke operator $T_q$. Similarly, we have $R_{\ell_0} \rsurj \bT_{\ell_0}$. 
\end{prop}
Note that since $\bT$ is generated as a $\Z_p$-algebra by the $T_q$, the characterizing property of the map makes its surjectivity visible. The level $\ell_0$ map is known to be an isomorphism $R_{\ell_0} \cong \bT_{\ell_0}$ \cite{WWE3}. 

\begin{rem}
Our hypothesis is that the local $\US_\ell$ conditions furnish a robust interpolation of the Steinberg shape of Galois representations of \eqref{eq: Galois Steinberg form} into Cayley--Hamilton algebras. Since the global $\US_N$ condition simply puts together these local conditions, we view the putative isomorphism $R \isoto \bT$ as bearing out this hypothesis. 
\end{rem}

\subsection{Reducibility of pseudorepresentations}
\label{subsec: reducibility}

A 2-dimensional pseudorepresentation $D : G \to A$ is called \emph{reducible} when it has the form $\psi(\chi_1 \oplus \chi_2)$ for some characters $\chi_1, \chi_2 : G \to A^\times$. It is well understood that reducibility is a Zariski-closed condition, meaning that there is a \emph{reducibility ideal} $J^\red_\Db \subset R_\Db$ such that a pseudodeformation $D_A : G \to A$ of $\Db$ is reducible if and only if $J^\red_\Db$ vanishes under the corresponding homomorphism $R_\Db \to A$. And any $D_A$ becomes reducible modulo the image of $J^\red_\Db$ in $A$. 

When $D_A$ arises from a GMA-representation of $G$, there is an important expression for the reducibility ideal in terms of the GMA structure. We record the universal $\US_N$ case.
\begin{prop}
    \label{prop: red ideal generators}
    The reducibility ideal $J^\red \subset R$ is equal to the image of the multiplication map $m : B \otimes_{R} C \to R$. 
\end{prop}

Another canonical ideal of $R$ is the kernel $\Jm$ of the composition 
\[
\Jm := \ker(R \rsurj \bT \rsurj \Z_p), 
\] 
that arises from the Eisenstein series $E_{2,N}$. This is characterized by sending $\Tr_{D_N}(\Fr_q) \in R$ for primes $q \nmid Np$ to $q+1$, which is the eigenvalue of $T_q$ on $E_{2,N}$. There is an inclusion of ideals $J^\red \subset \Jm$ because the $\Z_p$-valued pseudorepresentation $\psi(\kappa \oplus 1)$ associated to $E_{2,N}$ is reducible. 

In the following lemma, we compute the quotient of $R$ by the reducibility ideal, which we write as $R^\red := R/J^\red$. Here we write $\gamma_0 \in I_{\ell_0}$ for then chosen lift of the topological generator of the tame quotient of $I_{\ell_0}$, denoted $\gamma_{\ell_0}$ in \S\ref{subsec: notation conventions}. 

\begin{lem}
\label{lem: reducible quotient}
There is an isomorphism
\[
R^\red \cong  \frac{\Z_p[Y]}{(Y^2, (\ell_0-1)Y)},
\]
where $Y$ may be taken to be $a_{\gamma_0}-1$, and $Y$ generates $\Jm/J^\red$. The corresponding pseudorepresentation induced by reduction modulo $p$, $R^\red \rsurj \F_p[Y]/(Y^2)$, equals 
\[
D^\red := \psi(\omega(1 + Y a_0) \oplus (1 - Y a_0)) : G_{\Q,Np} \to \F_p[Y]/(Y^2). 
\]
\end{lem}

In Proposition \ref{prop: R-cotangent}, we use the local homomorphism $\varphi_{D^\red} : R \to \F_p[\epsilon]/(\epsilon^2)$ induced by $D^\red$. 

\begin{proof}
The first statement is a direct application of the presentation for $R^\red$ provided in \cite[Lem.\ 4.2.3]{WWE5}, and the calculations needed for the second claim are included in its proof. 
\end{proof}

\subsection{Designated generators of the universal GMA}

We recall the definitions of some useful cohomology classes and their duals from \cite[\S3.10]{WWE5}. First we need notation for generators of the tame quotients of inertia groups as in \S \ref{subsec: notation conventions}. 
\begin{defn}
    \label{defn: gamma_i}
    For a prime $\ell_i$ not equal to $p$, let $\gamma_i \in I_{\ell_i}$ stand for the element $\gamma_{\ell_i} \in I_{\ell_i}$ chosen in \S\ref{subsec: notation conventions}, which is a lift over the tame quotient $I_q \rsurj I_q^\mathrm{tame}$ of a topological generator. 
\end{defn}

\begin{prop}
    \label{prop: BC bases}
    The elements $b_{\gamma_0}$ and $b_{\gamma_1}$ of $B$ generate it as a $R$-module, and $B$ is not cyclic as a $R$-module. Similarly, $c_{\gamma_0} \in C$ is a generator as a $R$-module. 
\end{prop}

\begin{proof}
    See \cite[Lem.\ 3.9.4 and 3.9.8]{WWE5} for the claims about generators, and see \cite[Lem.\ 6.2.1]{WWE5} for the claim that $B$ is not cyclic. 
\end{proof}

\begin{lem}
    \label{lem: Jm2 and Jred}
    We have an inclusion of ideals $\Jm^2 \subset J^\red \subset \Jm$. The element $b_{\gamma_0} \cdot c_{\gamma_0}$ of $J^\red$ lies within its submodule $\Jm^2$. 
\end{lem}

\begin{proof}
    The inclusion of ideals follows from Lemma \ref{lem: reducible quotient} because the kernel of $R/J^\red = R^\red \rsurj \Z_p = R/\Jm$ is square-nilpotent (similar to the proof of \cite[Thm.\ 6.4.1]{WWE5}). The final claim is \cite[Lem.\ 5.2.5]{WWE5}. 
\end{proof}

\section{Additional arithmetic preliminaries}
\label{sec: additional arithmetic}

In this section, we continue assembling background much as in the previous section, with the distinction that the content of this section is not found in \cite{WWE5}. Our primary focus is a discussion of various implications from our choice of pinning data in Definition \ref{P1: defn: pinning} as well as the conditions in Assumption \ref{P1: assump: main}.

\subsection{Cocycles determined by the pinning data}
\label{subsec: choices}
We fix some notation for Galois cocycles determined by the pinning data of Definition \ref{P1: defn: pinning}.

Recall the canonical isomorphism 
\[
\Z[1/Np]^\times \otimes_\Z \F_p \isoto H^1(\Z[1/Np],\mu_p)
\]
of Kummer theory. It sends an element $n \in \Z[1/Np]^\times \otimes_\Z \F_p$ to the class of the cocycle 
$\sigma \mapsto \frac{\sigma n^{1/p}}{n^{1/p}}$ for a choice $n^{1/p} \in \oQ$ of $p$th root of $n$.
We call this element of $H^1(\Z[1/Np],\mu_p)$ the \emph{Kummer class of $n$} and call any cocycle in this class a \emph{Kummer cocycle of $n$}. Because $\mu_p \not\subset \Q^\times$, each Kummer cocycle of $n$ is given by $\sigma \mapsto \frac{\sigma n^{1/p}}{n^{1/p}}$ for a unique choice of $n^{1/p} \in \oQ$ of $p$th root of $n$. 
We use the isomorphism $\F_p(1) \cong \mu_p$ chosen in \S\ref{subsec: notation conventions} to value Kummer classes and cocycles in $\F_p(1)$.

\begin{defn} \hfill
    \label{P1: defn: pinned cocycles}
    \begin{itemize}

        \item Let
        \[
        b_0\up1, b_1\up1 \in Z^1(\Z[1/Np],\F_p(1))
        \]
        be the Kummer cocycles associated to $p$th roots $\ell_0^{1/p}$ and $\ell_1^{1/p}$ of $\ell_0$ and $\ell_1$, respectively, chosen in \S\ref{subsec: notation conventions}. Let $b\up1=b_1\up1$. 
        \item Denote the Kummer classes of $\ell_0$, $\ell_1$, and $p$, respectively, by
        \[
         b_0 , b_1, b_p \in H^1(\Z[1/Np], \F_p(1)). 
        \]
        Note that $b_0=[b_0\up{1}]$ and $b_1=[b_1\up1].$
        \item Let
        \[
        \gamma_0 \in I_{\ell_0}, \ \gamma_1 \in I_{\ell_1}
        \]
        be as in Definition \ref{defn: gamma_i} and fixed such that $b_i\up1(\gamma_i)~=~1$.
        \item The cohomology group $H^1_{(p)}(\Z[1/Np], \F_p(-1))$ has $\F_p$-dimension 1 by \cite[Lem.\ 3.10.2]{WWE5}. Let
        \[
        c_0 \in H^1(\Z[1/Np], \F_p(-1))
        \]
        denote the unique class in the image of $H^1_{(p)}(\Z[1/Np], \F_p(-1))$ such that $\tilde{c}_0(\gamma_0)=1$ for any cycle $\tilde{c}_0 \in Z^1(\Z[1/Np], \F_p(-1))$ representing $c_0$. 
        \item If $\tilde{c}_0 \in Z^1(\Z[1/Np], \F_p(-1))$ is a cocyle representing $c_0$, then every other such cocyle is of the form $\tilde{c}_0+dx$ for some $x \in \F_p(-1)$. Then
        \[
        (\tilde{c}_0+dx)(\Fr_{\ell_1})=\tilde{c}_0(\Fr_{\ell_1})+(\ell_1^{-1}-1)x.
        \]
        Then $x=-(\ell_1^{-1}-1)^{-1}\tilde{c}_0(\Fr_{\ell_1})$ is the unique choice such that $(\tilde{c}_0+dx)(\Fr_{\ell_1})=~0$. 
        Let 
        \[
        c\up1 \in Z^1(\Z[1/Np], \F_p(-1))
        \]
        be $\tilde{c}_0+dx$ for this choice of $x$. Then $c\up1$ is the the unique cocycle with cohomology class $c_0=[c\up1]$ such that $c\up1\vert_{\ell_1} = 0$. 
    \item Let 
    \[
    x_{c\up1} \in C^0(G_p,\F_p(-1))=\F_p(-1)
    \]
    be such that $c\up1|_p = dx_{c\up1}$. Concretely, for any $\tau \in G_p$ such that $\omega(\tau) \ne 1$, we can define $x_{c\up1}$ as $x_{c\up1}=(\omega(\tau)^{-1}-1)^{-1}c\up1(\tau).$
    \item Let 
    \[
    a_0, a_p \in Z^1(\Z[1/Np],\F_p)
    \]
    be non-zero homomorphisms ramified exactly at $\ell_0$ and at $p$, respectively, and such that $a_0(\gamma_0)=1$. This determines $a_0$ uniquely and determines $a_p$ up to $\F_p^\times$-scaling (which is sufficient for our purposes). 
    \end{itemize}
\end{defn}

\begin{rem}
    \label{lem: initial pinning dependence}
The following choices made in Definition \ref{P1: defn: pinned cocycles} depend only on the pinning data of Definition \ref{P1: defn: pinning}:
\begin{itemize}
    \item The cocycles $b\up1$, $c\up1$, and $a_0$.
    \item The images of $\gamma_i$ in $I_{\ell_i}^\mathrm{tame} \otimes_\Z \F_p$.
\end{itemize}
\end{rem}

\subsection{Cup products and congruence conditions}

The conditions in this paper's running assumption, Assumption \ref{P1: assump: main}, are presented in what we think is the most readable language. However, our methods require various implications of these conditions that are related to the the vanishing of certain cup products among the cohomology classes that we have just defined and/or the local vanishing of the cohomology classes themselves. The point of this section is to record those implications. 

We emphasize that we assume $p \geq 5$ throughout.

\begin{lem}[Conditions equivalent to (2) in Assumption \ref{P1: assump: main}]
\label{lem: log ell1 is zero}
Let $\ell_0, \ell_1$ be distinct primes such that $\ell_0 \equiv 1 \pmod{p}$ and $\ell_1 \not\equiv 0, \pm 1 \pmod{p}$. The following conditions (1)-(4) are equivalent. 
\begin{enumerate}
    \item $\ell_1$ is a $p$th power modulo $\ell_0$.
    \item $a_0 \vert_{\ell_1} = 0$ in $H^1(\Q_{\ell_1}, \F_p)$.
    \item $b_1\vert_{\ell_0} \in H^1(\Q_{\ell_0}, \F_p(1))$ vanishes.
    \item $b_1 \cup c_0 = 0$ in $H^2(\Z[1/Np], \F_p)$.
\end{enumerate}
\end{lem}

\begin{proof}
(1) $\iff$ (2). 
Let $F/\Q$ be the unique degree-$p$ subextension of $\Q(\zeta_{\ell_0})/\Q$. 
We see that (2) is true if and only if $\ell_1$ splits completely in $F/\Q$, which, in turn, is equivalent to a Frobenius element $\Fr_{\ell_1}$ for $\ell_1$ becoming trivial in $\Gal(F/\Q)$. Then the equivalence of (1) and (2) follows from the standard fact that $\Fr_{\ell_1} \mapsto \ell_1$ under the canonical isomorphism $\Gal(\Q(\zeta_{\ell_0})/\Q) \isoto \F_{\ell_0}^\times$. 

(1) $\iff$ (3). The Kummer theory isomorphism $H^1(\Q_{\ell_0}, \F_p(1)) \cong \Q_{\ell_0}^\times/(\Q_{\ell_0}^\times)^p$ sends $b_1$ to $\ell_1$. 

(3) $\iff$ (4). We will apply the injection $H^2(\Z[1/Np], \F_p) \rinj H^2(\Q_{\ell_0}, \F_p)$ of  \cite[Lem.\ 12.1.1]{WWE3} (recorded also in  Lemma \ref{lem: Hasse}, below), reducing the condition (4) to $b_1 \vert_{\ell_0} \cup c_0\vert_{\ell_0} = 0$ in $H^2(\Q_{\ell_0}, \F_p)$. Then (3) $\Rightarrow$ (4) is clear. The converse follows from the characterization of the $\ell_0$-local cup product of Lemma \ref{lem: tate duality at ell0}: because $c_0\vert_{\ell_0}$ is ramified, while $b_1\vert_{\ell_0}$ is non-trivial and unramified, their cup product is non-zero. 
\end{proof}

Next, the following lemma generalizes, to odd primes $p$, the pattern of ramification of the prime 2 in quadratic number fields. In particular, it establishes when $\ell_i^{1/p}\in\overline{\Q}$ in the pinning data of Definition \ref{P1: defn: pinning} can be chosen to have image in $\Q_p$ under the fixed embedding $\overline{\Q}\hookrightarrow\overline{\Q}_p$.
\begin{lem}
    \label{P1: lem: ram of ell1 at p}
    Let $\ell$ be a prime, $\ell \neq p$, and let $b_\ell \in H^1(\Q,\F_p(1))$ be the Kummer class of $\ell$. The following conditions are equivalent. 
    \begin{enumerate}
        \item $\ell^{p-1}-1$ is divisible by $p^2$
        \item $a_p\vert_{\ell} \in H^1(\Q_\ell, \F_p)$ is trivial
        \item $b_\ell\vert_p \in H^1(\Q_p, \F_p(1))$ is trivial 
        \item $b_\ell\vert_p \in H^1(\Q_p^\mathrm{ur}, \F_p(1))$ is trivial 
        \item $\Q(\ell^{1/p})/\Q$ is not totally ramified at $p$; or, what is the same, tamely ramified at $p$
        \item $p$ splits into two primes in $\Q(\ell^{1/p})/\Q$, one with ramification degree $p-1$ and one with ramification degree 1. 
    \end{enumerate}
    
\end{lem}

\begin{proof}
(1) $\iff$ (2). Because this proof is very similar to the proof of (1) $\iff$ (2) in Lemma \ref{lem: log ell1 is zero}, we omit it. 

(1) $\iff$ (3). Likewise, see the proof of (1) $\iff$ (3) in Lemma \ref{lem: log ell1 is zero}. 

(3) $\iff$ (4). This is \cite[Lem.\ B.1.1]{WWE5}. 

(4) $\iff$ (5). Consider $b_\ell\vert_{\Q(\zeta_p)}$, which is a surjective homomorphism $\Gal(\oQ/\Q(\zeta_p)) \rsurj \F_p(1)$. We observe that both (4) and (5) are equivalent to $b_\ell\vert_{\Q(\zeta_p)}$ being unramified at the unique prime $(1-\zeta_p)$ of $\Q(\zeta_p)$ over $p$. 

(5) $\iff$ (6). The implication (6) $\implies$ (5) is clear. For the converse, note that the Galois closure of $\Q(\ell^{1/p})/\Q$ is $\Q(\ell^{1/p}, \zeta_p)/\Q$, and carry out a prime decomposition exercise. 
\end{proof}

We now shift to a discussion of local cup products related to item (1) in Assumption \ref{P1: assump: main}. Indeed, since we have assumed $\ell_0 \equiv 1 \pmod{p}$, our chosen primitive $p$th root of unity $\zeta_p \in \oQ$, along with the chosen embedding $\oQ \rinj \oQ_{\ell_0}$, induces an isomorphism 
\begin{equation}
\label{eq: Fp(i) isoms}
\F_p(i)\vert_{\ell_0} \isoto \F_p(j)\vert_{\ell_0}, \qquad x \mapsto x \otimes \zeta_p^{j-i}
\end{equation}
of representations of $G_{\ell_0}$ for any $i,j \in \Z$. We can also view this as a cup product in cohomology, because $\F_p(i) = H^0(\Q_{\ell_0}, \F_p(i))$. One may readily check that cup products with these cohomology classes result in isomorphisms
\[
H^0(\Q_{\ell_0}, \F_p(i)) \otimes_{\F_p} H^j(\Q_{\ell_0}, M) \isoto H^j(\Q_{\ell_0}, M(i))
\]
for any $\F_p[G_{\ell_0}]$-module $M$ and any $i,j \in \Z$. We will also use, in what follows, that the cup product is ``bilinear under multiplication (via the cup product) by elements of $H^0(\Q_{\ell_0}, \F_p(s))$, $s \in \Z$.'' A concise way to precisely state this fact is that the sum of the cup products on $H^1(\Q_{\ell_0}, -)$ applied to all of the $\F_p(i)$, namely, 
\begin{equation}
\label{eq: cup bilinear}
\bigoplus_{i \in \Z} H^1(\Q_{\ell_0}, \F_p(i)) \times \bigoplus_{j \in \Z} H^1(\Q_{\ell_0}, \F_p(j)) \to \bigoplus_{k \in \Z} H^2(\Q_{\ell_0}, \F_p(k)),
\end{equation}
is graded bilinear over the graded ring $\bigoplus_{s \in \Z} H^0(\Q_{\ell_0}, \F_p(s))$. 

We will be particularly interested in the cup product action of $\F_p(i)\vert_{\ell_0}$ on the local Tate duality pairing: for $i \in \Z$,
\begin{equation}
\label{eq: tate duality} 
H^1(\Q_{\ell_0}, \F_p(i)) \times H^1(\Q_{\ell_0}, \F_p(1-i)) \lra H^2(\Q_{\ell_0}, \F_p(1)) \cong \F_p. 
\end{equation}
We express all of the possible twists of this pairing in the following lemma. 
\begin{lem}
    \label{lem: tate duality at ell0}
    For any $i,j \in \Z$, we have a perfect pairing
    \begin{equation}
        \label{eq: twisted tate duality}
        H^1(\Q_{\ell_0}, \F_p(i)) \times H^1(\Q_{\ell_0}, \F_p(j)) \to H^2(\Q_{\ell_0}, \F_p(i+j)) \cong \F_p(1-i-j)
    \end{equation}
    under which 
    \begin{enumerate}
        \item the cup product of a ramified class with a non-trivial unramified class is non-zero and
        \item the cup product of any two unramified classes is zero. 
    \end{enumerate} 
    When $i=j$, the self-pairing \eqref{eq: twisted tate duality} is alternating. 
\end{lem}

\begin{proof}
The claims (1) and (2) are straightforward for $i=0, j=1$ using class field theory. This holds true for all $i,j$ using graded bilinearity of \eqref{eq: cup bilinear}. The alternating property follows from (1), (2), and an extra application of duality. 
\end{proof}

We turn from local cup products to implications for global cup products, which we will frequently use. 
\begin{lem}[Hasse principle]
\label{lem: Hasse} 
For $i=-1,0,1$, the map 
\[
H^2(\Z[1/Np],\F_p(i)) \to H^2(\Q_{\ell_0},\F_p(i)) \oplus H^2(\Q_{\ell_1},\F_p(i)), \quad x \mapsto (x|_{\ell_0}, x|_{\ell_1})
\]
is injective. 
\end{lem}
Our proof for $i=-1$ also uses (2) in Assumption \ref{P1: assump: main}, i.e., $\ell_1 \not\equiv 0, \pm 1 \pmod{p}$. 
\begin{proof}
    For $i=0,-1$, the map is an isomorphism. The case $i=0$ follows directly from \cite[Lem.\ B.1.2]{WWE5}. 
    The case $i=-1$ more-or-less follows from the argument for \cite[Lem.\ 12.1.1]{WWE3}, but that argument is written in the setting where ``$N$'' is a prime that is $1 \pmod{p}$. The same argument applies in our setting, where $N = \ell_0\ell_1$ with $\ell_0 \equiv 1 \pmod{p}$ and $\ell_1 \not\equiv \pm 1, 0 \pmod{p}$, because $H^j(\Q_{\ell_1}, \F_p(-1)) =0$ for all $j \in \Z_{\geq i}$, making the exact triangle
    \[
    R\Gamma(\Z[1/\ell_0 p], \F_p(-1)) \to  R\Gamma(\Z[1/N p], \F_p(-1)) \to R\Gamma(\Q_{\ell_1}, \F_p(-1))
    \]
    degenerate. 
    
    The case $i=1$ remains. Here the localization map 
    \[
    H^2(\Z[1/Np],\F_p(i)) \to H^2(\Q_{\ell_0},\F_p(i)) \oplus H^2(\Q_{\ell_1},\F_p(i)) \oplus H^2(\Q_p,\F_p(i))
    \]
    has cokernel of dimension 1, since $H^3_{(c)}(\Z[1/Np], \F_p(1)) \cong \F_p$ (the target of global duality pairings). By the theory of the Brauer group (see e.g.\ \cite[Thm.\ 1.5.36]{poonen2017}), we know that the map is injective with image consisting of the subspace summing to zero under the isomorphisms $H^2(\Q_q, \F_p(1)) \cong \F_p$ for $q = \ell_0, \ell_1, p$. Therefore its projection to any two summands, such as those in the lemma, is injective. 
\end{proof}

We conclude this section with several conditions that are equivalent to (3) in Assumption \ref{P1: assump: main}. This assumption states that the Hecke algebra $\bT_{\ell_0}$, which captures the Hecke eigensystems all of the weight 2 level $\Gamma_0(\ell_0)$ modular forms congruent to an Eisenstein series (see \S\ref{P1: subsec: modular forms}), is as small as possible given Mazur's result that there exists some cusp form congruent to an Eisenstein series. Note that this proposition is proven in {\cite[Thm.\ 1.2.1]{WWE3}}.

\begin{prop}[Conditions equivalent to item (3) in Assumption \ref{P1: assump: main}]
\label{prop: rk2 assumption}
Assume that $\ell_0 \equiv 1 \pmod{p}$ and $\ell_1 \not\equiv \pm 1 \pmod{p}$. The following are equivalent. 
\begin{enumerate}
    \item $\rk_{\Z_p} \bT_{\ell_0} = 2$
    \item $b_0 \cup c_0 \neq 0$ in $H^2(\Z[1/\ell_0p], \F_p)$
    \item $b_0 \cup c_0 \neq 0$ in $H^2(\Z[1/Np], \F_p)$
    \item $b_0\vert_{\ell_0} \cup c_0\vert_{\ell_0} \neq 0$ in $H^2(\Q_{\ell_0}, \F_p)$
    \item $a_0 \cup c_0 \neq 0$ in $H^2(\Z[1/\ell_0p], \F_p(-1))$
    \item $a_0 \cup c_0 \neq 0$ in $H^2(\Z[1/Np], \F_p(-1))$
    \item $a_0\vert_{\ell_0} \cup c_0\vert_{\ell_0} \neq 0$ in $H^2(\Q_{\ell_0}, \F_p(-1))$
    \item $\{a_0\vert_{\ell_0}, \zeta \cup c_0\vert_{\ell_0}\}$ is a basis for $H^1(\Q_{\ell_0}, \F_p)$, for any non-zero $\zeta \in H^0(\Q_{\ell_0}, \F_p(1))$. 
    \item $\{b_0\vert_{\ell_0}, \zeta' \cup c_0 \vert_{\ell_0}\}$ is a basis for $H^1(\Q_{\ell_0}, \F_p(1))$, for any non-zero $\zeta' \in H^0(\Q_{\ell_0}, \F_p(2))$. 
\end{enumerate}
\end{prop}

\begin{proof}
The equivalence of (1), (2), and (5) is the content of \cite[Thm.\ 1.2.1]{WWE3}. Because $\ell_1 \not\equiv \pm 1 \pmod{p}$, for $i=-1,0$, we have $H^2(\Q_{\ell_1},\F_p(i))=0$. Then Lemma \ref{lem: Hasse} supplies the equivalences of (2) with (3) and (4), and (5) with (6) and (7). The equivalence of (7), (8), and (9) follows from Lemma \ref{lem: tate duality at ell0} and the fact, visible in \cite[Lem.\ 12.1.3]{WWE3}, that $a_0 \cup \zeta = b_0$ for some non-zero $\zeta \in H^0(\Q_{\ell_0}, \F_p(1))$. 
\end{proof}

\section{An explicit first-order deformation}
\label{sec: explicit first-order}

We construct an irreducible first-order pseudodeformation $D_1 : G_{\Q,Np} \to \F_p[\ep]/(\ep^2)$ of $\Db : G_{\Q,Np} \to \F_p$ that satisfies the unramified-or-Steinberg ($\US_N$) property. This is a precursor to the constructions at second order that will be needed to prove the main technical result (Proposition \ref{prop: key alpha beta equation}). 

\subsection{1-reducible GMAs and $n$-th order pseudodeformations}
\label{subsec: 1-reducible}

When $F$ is a field, write $F[\ep_n]$ for the $F$-algebra $F[\ep]/(\ep^{n+1})$. For $m < n$, we think of $F[\ep_m]$ as an $F[\ep_n]$-algebra via the natural quotient map $F[\ep_n] \onto F[\ep_m$]. Given some algebraic object $X$ over $F$, we call a deformation of $X$ to $F[\ep_n]$ an \emph{$n$-th order deformation of $X$}. 

\subsubsection{1-reducible GMAs}
\label{sssec: 1-reducible}
We introduce $1$-reducible GMAs as a way to model truncations of a DVR-valued representations in a way that is ``lattice-independent". To justify this, consider the following example.

\begin{eg}
\label{eg:why 1-reducible}
Let $F$ be a field, $G$ be a group, and $\rho: G \to \GL_2(F \lb x \rb)$ be a function that can be written as
\[
\rho(g) = \ttmat{a(g)}{xb(g)}{c(g)}{d(g)}
\]
for some functions $a,b,c,d: G \to F \lb x \rb$. Suppose we want to check that $\rho$ is a homomorphism. Equivalently, we can check this in stages labeled by natural numbers $n$: at each stage $n$, check that $\rho \pmod{x^n}$ is a homomorphism. This amounts to checking some equations involving the functions $a$, $b$, $c$, and $d$, for instance
\[
a(gg') \equiv a(g)a(g') + x b(g)c(g') \pmod{x^n}.
\]
Note that this equation and the related equations for $b(gg')$ and $d(gg')$ only involve $b$ and $c$ modulo $x^{n-1}$. At stage $n$, only the equation for $c(gg')$ involves $c$ modulo $x^n$. 

On the other hand, another way to check that $\rho$ is a homomorphism is to consider the conjugate $\rho' = \sm{x^{-1}}{0}{0}{1} \rho \sm{x}{0}{0}{1}$---that is,
\[
\rho'(g) = \ttmat{a(g)}{b(g)}{xc(g)}{d(g)}
\]
---and check that $\rho'$ is a homomorphism. Again we can check this in stages, and this will involve the very same set of equations as for $\rho$, but in a different order. For instance, at stage $n$ for $\rho'$, the equation for $b(gg')$ will involve $b$ modulo $x^n$.
\end{eg}

In the example, if $\rho$ is a homomorphism, then $\rho$ and $\rho'$ can be thought of as two different $F\lb x \rb$-lattices in the same $F\lp x \rp$-representation. One can think of $1$-reducible GMAs as a tool for studying this kind of problem in a way that does not favor one lattice over the other, and where one considers the minimal set of equations at each stage. This tool is especially well-suited to studying pseudorepresentations (note that $\rho$ and $\rho'$ have the same trace and determinant, and that the trace and determinant of $\rho$ modulo $x^n$ only involve $b$ and $c$ modulo $x^{n-1}$).

\begin{defn}
\label{defn: 1-reducible GMA}
The $1$-reducible GMA over $\F_p[\epsilon_n]$ is the GMA $E_n$ given by
\[
E_n = \ttmat{\F_p[\epsilon_n]}{\F_p[\epsilon_{n-1}]}{\F_p[\epsilon_{n-1}]}{\F_p[\epsilon_n]}
\]
with the multiplication map 
\[
m: \F_p[\epsilon_{n-1}] \otimes_{\F_p[\epsilon_n]} \F_p[\epsilon_{n-1}] \to \F_p[\epsilon_n]
\]
given by the composition
\[
\F_p[\epsilon_{n-1}] \otimes_{\F_p[\epsilon_n]} \F_p[\epsilon_{n-1}]  \xrightarrow{b \otimes c \mapsto bc} \F_p[\epsilon_{n-1}] \xrightarrow{x \mapsto \epsilon x} \F_p[\epsilon_n].
\]
\end{defn}

The image of the multiplication map $m$ is $\ep \F_p[\epsilon_n]$. In particular, if $\rho: G \to E_n^\times$ is a Cayley--Hamilton representation such that the induced map $\F_p\lb G \rb \to E_n$ is surjective, then the reducibility ideal of $\rho$ is $\ep \F_p[\epsilon_n]$. 

\begin{rem}
The following relationship between 1-reducible GMAs and their induced pseudorepresentations plays an especially important role in this paper: a representation of a group $G$ valued in $E_n$ induces a $\F_p[\ep_n]$-valued pseudorepresentation of $G$, and we need not concern ourselves over whether this pseudorepresentation comes from a representation of $G$ valued in $\GL_2(\F_p[\ep_n])$. Indeed, in some cases, it may not. 
\end{rem}

\begin{rem}
There is also a natural notion of \emph{$k$-reducible GMA} $E_{k,n}$ for $k=2,\dots,n$, where $\F_p[\epsilon_{n-1}]$ is replaced by $\F_p[\epsilon_{n-k}]$ and that map $x \mapsto \epsilon x$ is replaced by $x \mapsto \epsilon^k x$. In this case, the reducibility ideal of a surjective Cayley--Hamilton representation $\F_p\lb G \rb \to E_{k,n}$ is $\ep^k\F_p[\epsilon_n]$. This explains the naming convention---the `$k$' in $k$-reducible refers to the exponent of the uniformizer in the reducibility ideal. We will not need this notion in this paper.
\end{rem}

\begin{eg}
For example, when $n=1$, an element of the 1-reducible GMA $E_1$ over $\F_p[\ep_1]$ can be written uniquely as $\sm{\alpha_0+\epsilon\alpha_1}{\beta}{\gamma}{\delta_0+\epsilon\delta_1}$ for $\alpha_i, \beta, \gamma, \delta_i \in \F_p$, and the multiplication is
\begin{align*}
\ttmat{\alpha_0+\epsilon\alpha_1}{\beta}{\gamma}{\delta_0+\epsilon\delta_1} \ttmat{\alpha'_0+\epsilon\alpha'_1}{\beta'}{\gamma'}{\delta'_0+\epsilon\delta'_1} = \\ \ttmat{\alpha_0\alpha_0'+\epsilon(\alpha_1\alpha_0'+\alpha_0\alpha_1'+\beta \gamma')
}{\alpha_0\beta'+\beta\delta_0'}{\gamma\alpha_0' + \delta_0\gamma'}{\delta_0\delta_0'+\epsilon(\delta_1\delta_0' + \delta_0\delta_1' + \beta' \gamma)}
\end{align*}
\end{eg}

\begin{eg}
Let $\rho$ and $a$, $b$, $c$, $d$, be as in Example \ref{eg:why 1-reducible} with $F=\bF_p$ and with the variable $x$ replaced by $\ep$. Suppose that $\rho$ is a homomorphism. Then, for every $n>0$, the map
\[
g \mapsto \ttmat{a(g) \pmod{\ep^n}}{b(g) \pmod{\ep^{n-1}}}{c(g) \pmod{\ep^{n-1}}}{d(g) \pmod{\ep^{n}}}
\]
gives a homomorphism $G \to E_n^\times$.
\end{eg}

\subsubsection{Reduction of 1-reducible GMAs}

When $n \geq m$, the standard surjection $\F_p[\ep_n] \rsurj \F_p[\ep_m]$, $\ep\mapsto\ep$, extends naturally to 1-reducible GMAs. We have a \emph{reduction map}
\[
r_{n,m} : E_n \rsurj E_m,
\]
simply reducing each of the coordinates under the usual surjections $\F_p[\ep_n] \to \F_p[\ep_m]$ and $\F_p[\ep_{n-1}]\rsurj \F_p[\ep_{m-1}]$, which is a $\F_p[\ep_n]$-algebra homomorphism. We will especially apply the case
\begin{equation}
    \label{eq: r21}
    r_{2,1} : E_2 \rsurj E_1. 
\end{equation}

The reduction map $r_{n,m}$ is distinct from the tensor reduction map
\[
E_n \rsurj E_n \otimes_{\F_p[\ep_n]} \F_p[\ep_m], \quad x \mapsto x \otimes 1,
\]
which is also a ring homomorphism. As long as $n > m$, the latter has the form
\[
\ttmat{\F_p[\ep_n]}{\F_p[\ep_{n-1}]}{\F_p[\ep_{n-1}]}{\F_p[\ep_n]} \rsurj 
\ttmat{\F_p[\ep_m]}{\F_p[\ep_m]}{\F_p[\ep_m]}{\F_p[\ep_m]}, 
\]
where the target is a GMA with cross-diagonal multiplication $b \otimes c \mapsto \ep b c$. Later, we will apply the factorization of the reduction map $r_{2,1} : E_2 \rsurj E_1$ into
\begin{equation}
    \label{eq: r21 factorization}
    E_2 \rsurj \ttmat{\F_p[\ep_m]}{\F_p[\ep_m]}{\F_p[\ep_m]}{\F_p[\ep_m]} \rsurj E_1,
\end{equation}
where the leftmost map is the tensor reduction map for $(n,m) = (2,1)$, and the rightmost map is reduction modulo $\ep$ of the off-diagonal coordinates. 

\subsubsection{Convenient mappings from 1-reducible GMAs} 

We will have to work explicitly with the finite-flat property of Cayley--Hamilton representations of $G_p$ over $\Db$. We know from Proposition \ref{prop: finite-flat implies upper tri} that they must be upper-triangular, which makes it possible to apply the test of finite-flatness in Lemma \ref{lem: GMA property} in a straightforward way. Now we contextualize it to the 1-reducible GMA, $E_n$ over $\F_p[\ep_n]$, for $n \in \Z_{\geq 1}$. 

\begin{lem}
\label{lem: UT embedding}
There is an $\F_p[\ep_n]$-algebra embedding of the upper-triangular sub-$\F_p[\ep_n]$-GMA 
\[
U_n := \sm{\F_p[\ep_n]}{\F_p[\ep_{n-1}]}{0}{\F_p[\ep_n]} \subset E_n
\]
into $M_2(\F_p[\epsilon_n])$ given by
\[
\begin{pmatrix}
a & b \\
0 & d 
\end{pmatrix} 
\mapsto 
\begin{pmatrix}
a & \epsilon b \\
0 & d 
\end{pmatrix},
\]
where the map on the upper right coordinate denotes the natural multiplication-by-$\epsilon$ map, written $\cdot \epsilon: \F_p[\epsilon_{n-1}] \to \F_p[\epsilon_{n}]$. 
\end{lem}
\begin{proof}
The map is clearly a morphism of $\F_p[\ep_n]$-modules, so it suffices to show that it respects the multiplication. This is checked easily.
\end{proof}

On the other hand, we can realize some (but not all) of $E_n$ within a matrix algebra by reducing modulo $\ep^n$ to $\F_p[\ep_{n-1}]$. 

\begin{lem}
\label{lem: matrix reduction}
There is an $\F_p[\ep_n]$-algebra homomorphism from $E_n$ to $M_2(\F_p[\ep_{n-1}])$ given by 
\[
\ttmat{a}{b}{c}{d} \mapsto \ttmat{\bar a}{b}{\ep c}{\bar d},
\]
where $\bar a, \bar d \in \F_p[\ep_{n-1}]$ indicates reduction modulo $\ep^n$.
\end{lem}

\subsection{The cochain $a^{(1)}$}

\label{subsec: cochains a and d}
Our goal is to produce a first-order 1-reducible GMA representation $\rho_1 : G_{\Q,Np} \to E_1^\times$ deforming $\omega \oplus 1$. We start by defining a cochain $a\up1: G_{\Q,Np} \to \F_p$ that will be used in the definition of $\rho_1$.

Recall the cocycles $b\up1$, $c\up1$, $a_0$ and $a_p$, and the cochain $x_{c\up1}$ specified in Definition \ref{P1: defn: pinned cocycles}. The cohomology classes of $b\up1$ and $c\up1$ are $b_1$ and $c_0$, respectively.

\begin{lem}
\label{P1: lem: produce a1}
There is a unique cochain $a\up1 \in C^1(\Z[1/Np],\F_p)$ satisfying the following three properties:
\begin{enumerate}
\item $-da\up1 = b\up1 \smile c\up1$, 
\item $(a\up1 -b\up1 \smile x_{c\up1})|_{I_p} = 0$ in $H^1(\Q_p^\mathrm{nr}, \F_p)$, and
\item the class of $a\up1|_{\ell_0}$ in $H^1(\Q_{\ell_0}, \F_p)$ is on the line spanned by $\zeta \cup c_0|_{\ell_0}$ for any (equivalently, all) non-trivial $\zeta \in \mu_p(\Q_{\ell_0}) \cong H^0(\Q_{\ell_0}, \F_p(1))$. 
\end{enumerate}
Moreover, $a\up1\vert_{\ell_0}$ is a cocycle, $a\up1|_{\ell_1}$ is an unramified cocycle, and the definition of $a\up1$ depends only on the pinning data of Definition \ref{P1: defn: pinning}.
\end{lem}

\begin{proof}
Since $b_1 \cup c_0=0$ by Lemma \ref{lem: log ell1 is zero}, we know there is a cochain $g \in C^1(\Z[1/Np],\F_p)$ such that $-dg=b\up1 \smile c\up1$. The set of such $g$ is a torsor for $Z^1(\Z[1/Np],\F_p)$, which is generated by $a_0$ and $a_p$.

For any $-g$ whose coboundary is $b\up1 \smile c\up1$, we have
\[
-d g|_p = b\up1|_p \smile c\up1|_p = b\up1|_p \smile dx_{c\up1} = - d(b\up1|_p \smile x_{c\up1})
\]
Hence $(g -b\up1|_p \smile x_{c\up1})|_p$ is a cocycle. Since $H^1(\Q_p,\F_p)$ is generated by its unramified subgroup $H^1_\mathrm{un}(\Q_p,\F_p)$ together with $a_p|_p$, we have 
\[
(g -b\up1|_p \smile x_{c\up1})|_p \equiv ya_p|_p \pmod{H^1_\mathrm{un}(\Q_p,\F_p)}
\]
for a unique $y \in \F_p$. Replacing $g$ by $g-ya_p$, we see that the set of $g$ satisfying (1) and (2) is a non-empty torsor for $H^1(\Z[1/N],\F_p)$ (which is spanned by $a_0$).

By Lemma \ref{lem: log ell1 is zero}, the homomorphism $b\up1|_{\ell_0} : G_{\ell_0} \to \F_p(1)$ vanishes. Hence for any $g$ satisfying (1) and (2), we have
\[
-dg|_{\ell_0} = b\up1|_{\ell_0} \smile c\up1|_{\ell_0} = 0,
\]
so $g|_{\ell_0}$ is a cocycle. Since we assume that the equivalent conditions of Proposition \ref{prop: rk2 assumption} are true, the set $\{a_0|_{\ell_0},c_0|_{\ell_0}\}$ is a basis for $H^1(\Q_{\ell_0},\F_p)$. Hence there is a unique $\gamma \in \F_p$ such that $(g-\gamma a_0)|_{\ell_0}$ is in the line spanned by $c_0|_{\ell_0}$, and we define $a\up1=g-\gamma a_0$ for this $\gamma$. 

Finally, since $c\up1|_{\ell_1}=0$, condition (1) implies that $a\up1|_{\ell_1}$ is a cocycle. In particular, $a\up1|_{\ell_1}$ is unramified: because $p \nmid \ell_1(\ell_1-1)$, by local class field theory, any homomorphism from $G_{\ell_1}$ to $\F_p$ is unramified. 
\end{proof}

Condition (3) in Lemma \ref{P1: lem: produce a1} provides the invariant $\alpha$, which we now define.

\begin{defn}
\label{P1: defn: alpha}
Let $\alpha \in \F_p(1)$ be the unique element such that
\[
[a\up1|_{\ell_0}] = \alpha \cup c_0|_{\ell_0}
\]
Observe that $\alpha$ depends only on the pinning data of Definition \ref{P1: defn: pinning}. 
\end{defn}

\subsection{An irreducible first-order deformation}
\label{subsec: construct rho1} 
We now produce a first-order 1-reducible GMA representation $\rho_1 : G_{\Q,Np} \to E_1^\times$ deforming $\omega \oplus 1$ and satisfying the unramified-or-Steinberg condition $\US_N$ of Definition \ref{defn: US global}. The construction uses the cocycles $b\up1$ and $c\up1$ fixed in Definition \ref{P1: defn: pinned cocycles} and the cochain $a\up1$ defined in Lemma \ref{P1: lem: produce a1}, together with the cochain $d\up1$ defined by
\[
d\up1 = b\up1c\up1-a\up1.
\]
Note that, since $-d(b\up1c\up1) = b\up1 \smile c\up1 + c\up1 \smile b\up1$, we have
\begin{equation}
    \label{eq: ddup1}
    -dd\up1 = c\up1 \smile b\up1.
\end{equation}

\begin{lem}
\label{lem: D1 construction}
Let $E_1$ be the 1-reducible GMA over $\F_p[\epsilon_1]$. Let $\rho_1:G_{\Q,S} \to E_1^\times$ be the function given in coordinates by 
\begin{equation}
\label{eq: define d1}
\rho_1 = \ttmat{\omega(1+a\up1 \ep)}{b\up1}{\omega c\up1}{1+d\up1\epsilon}.
\end{equation}
Then $\rho_1$ is a homomorphism that is $\mathrm{US}_N$. In particular, the associated pseudorepresentation 
\[
D_1:=\psi(\rho_1), \ \Tr_{D_1} = \omega+1 + \epsilon(b\up1 c\up1+(\omega-1)a\up1): G_{\Q,S} \to \F_p[\epsilon_1]
\]
is $\US_N$, and it induces a surjective homomorphism $\varphi_{D_1}:R \to \F_p[\epsilon_1]$.
\end{lem}

\begin{proof} We check the conditions one by one, recalling that the $\US_N$ condition entails a condition upon restriction to the decomposition group at every prime dividing $Np$. 
\begin{description}
\item[Homomorphism] The homomorphism condition on $\rho_1$ can readily be checked to be equivalent to the following equalities of $2$-coboundaries: $db\up1=0$, $dc\up1=0$, $-da\up1=b\up1 \smile c\up1$ and $-dd\up1=c\up1 \smile b\up1$. The first three equations hold by definition, and the last by \eqref{eq: ddup1}.

\item[Finite-flat at $p$] Recall the element $x_{c\up1} \in \F_p(-1)$ of Definition \ref{P1: defn: pinned cocycles} that satisfies $dx_{c\up1}=c\up1|_p$. Conjugating $\rho_1$ by $\sm{1}{0}{-x_{c\up1}}{1}$ we find that 
\[
\ad(\sm{1}{0}{-x_{c\up1}}{1})\rho_1|_p = \ttmat{\omega(1+(a\up1|_p- b\up1|_p \smile x_{c\up1})\epsilon )}{b\up1|_p}{0}{1-(a\up1|_p- b_1|_p \smile x_{c\up1})\epsilon}
\]
Since $a\up1|_p- b\up1|_p \smile x_{c\up1}|_p$ is an unramified element of $Z^1(\Q_{p},\F_p)$ and $b\up1$ induces a finite-flat extension of $\F_p$ by $\F_p(1)$ by Lemma \ref{P1: lem: local Kummer finite-flat}, $\rho_1\vert_p$ is finite-flat by Lemmas \ref{lem: GMA property} and \ref{lem: UT embedding}. 

\item[Unramified-or-Steinberg at $\ell_0$] Let $\sigma,\tau \in G_{\ell_0}$. Using the facts that $\omega|_{\ell_0}=1$ and $b\up1|_{\ell_0} = 0$, it follows that
\[
(\rho_1(\sigma)-\omega(\sigma))(\rho_1(\tau)-1) = \ttmat{\epsilon a\up1(\sigma)}{0}{c\up1(\sigma)}{\epsilon d\up1(\sigma)} \cdot \ttmat{\epsilon a\up1(\tau)}{0}{c\up1(\tau)}{\epsilon d\up1(\tau)} =0
\]

\item[Unramified-or-Steinberg at $\ell_1$] Let $\sigma,\tau \in G_{\ell_1}$. Using the fact that $c\up1|_{\ell_1} = 0$, we find that $(\rho_1(\sigma)-\omega(\sigma))(\rho_1(\tau)-1)$ is equal to
\begin{align*}
&\ttmat{\epsilon \omega(\sigma)a\up1(\sigma)}{b\up1(\sigma)}{0}{1-\omega(\sigma)+\epsilon d\up1(\sigma)} \cdot \ttmat{\omega(\tau)-1+\epsilon \omega(\tau)a\up1(\tau)}{b\up1(\tau)}{0}{\epsilon d\up1(\tau)} \\
& =  \ttmat{\epsilon \omega(\sigma)a\up1(\sigma)(\omega(\tau)-1)}{0}{0}{\epsilon(1-\omega(\sigma)) d\up1(\tau)}.
\end{align*}
If $\sigma \in I_{\ell_1}$, then $a\up1(\sigma)=0$ and $\omega(\sigma)=1$, so this is zero. If, on the other hand, $\tau \in I_{\ell_1}$, then $d\up1(\tau)=0$ and $\omega(\tau)=1$, so this is zero.
\item[$\varphi_{D_1}$ is surjective] We have homomorphisms $b\up1, c\up1 : G_{\Q(\zeta_p)} \to \F_p$ that are not scalar multiples of each other. Therefore there exists $\sigma \in G_{\Q(\zeta_p)}$ such that $b\up1(\sigma) \neq 0$ and $c\up1(\sigma) \neq 0$. Then we observe that $\Tr_{D_1}(\sigma) - 2 = \ep b\up1(\sigma)c\up1(\sigma)$, so $\ep$ is in the image of $\varphi_{D_1}$. 
\qedhere
\end{description}
\end{proof}

Note that $\rho_1$, $D_1$ and the homomorphism $R \to \F_p[\ep_1]$ depend only on the pinning data of Definition \ref{P1: defn: pinning}. This is clear since $a\up1$, $b\up1$, $c\up1$, and $d\up1$ only depend on this data.

\subsection{Relation to the universal case}
Recall the universal $\US_N$ Cayley--Hamilton representation $(\rho_N, E, D_{E})$ from Definition \ref{defn: universal objects}. By the universal property, the representation $\rho_1$ of Lemma \ref{lem: D1 construction} induces a homomorphism 
\[
E \otimes_{R} \F_p[\ep_1] \to E_1,
\]
of Cayley--Hamilton $\F_p[\ep_1]$-algebras. We can assume the GMA structure on $E$ to be compatible with this homomorphism, in the following sense.

\begin{prop}
    \label{prop: GMA structure for E1} 
    There exists a choice of $R$-GMA structure on $E$ such that
    \begin{enumerate}
        \item $E \to E_1$ is a map of GMAs
        \item 
        The elements
        \[
        \ttmat{0}{b_{\gamma_0}}{0}{0}, \ttmat{0}{b_{\gamma_1}}{0}{0}, \ttmat{0}{0}{c_{\gamma_0}}{0}
        \]
        of $E$ with respect to this GMA structure (as in Definition \ref{defn: universal objects}) map to the elements
        \[
        \ttmat{0}{0}{0}{0}, \ttmat{0}{1}{0}{0}, \ttmat{0}{0}{1}{0}
        \]
        of $E_1$, respectively.
    \end{enumerate}
\end{prop}

\begin{proof}
    Apply the idempotent lifting lemma of \cite[Lem.\ 5.6.8]{WWE1}. 
\end{proof}

We choose the GMA structure on $E$ such that the conditions (1) and (2) are satisfied. Although there may be many such choices, any of them will suffice for our purposes. Note that the conditions (1) and (2) are determined by the pinning data (Definition \ref{P1: defn: pinning}). 

\section{The pseudodeformation ring $R/pR$ up to second order}
\label{sec: computation of R} 

Recall from Definition \ref{defn: universal objects} that $R$ denotes the pseudodeformation ring of $\omega \oplus 1$ with the $\US_N$ condition. Let $\bar \m \subset R/pR$ denote the maximal ideal. In this section, we prove that $\dim_{\F_p} R/(p,\bar\m^2) = 3$ while $\dim_{\F_p} R/(p,\bar\m)^3 \leq 4$, also identifying a generator of $\bar \m^2$. These results are summed up in Corollary \ref{cor: constraint on S}. In the sequel, we will use these results to establish a presentation of $R/(p,\bar \m^3)$ and to distinguish between the cases $\rk_{\Z_p} \bT = 3$ and $\rk_{\Z_p} \bT > 3$, keeping in mind that we have a surjection $R \rsurj \bT$ from Proposition \ref{prop: R to T}. As always, Assumption \ref{P1: assump: main} is in force. 

In addition to the notation, such as $J^\red$, $J^{\min}$, $B$ and $C$, set up in \S\ref{P1: sec: recollection of WWE5}, we use the following:
\begin{itemize}
    \item Let $\bar R := R/pR$, for convenience. 
    \item If $I \subset R$ is an ideal, let $\bar I \subset \bar R$ denotes its image in $\bar R$. We warn the reader that the natural surjection $I/pI \to \bar I$ may not be an isomorphism. 
    \item Let $\m = (\Jm,p) \subset R$ be the maximal ideal, which is consistent with $\bar \m \subset \bar R$ also being maximal. 
    \item For a Noetherian local $\Z_p$-algebra $(A, \mathfrak{n})$, let $\frt_A$, the (mod $p$) \emph{tangent space of $A$}, be the set of local ring homomorphisms $\Hom(A,\F_p[\epsilon]/(\epsilon^2))$, which is an $\F_p$-vector space. The dual vector space $\frt_A^*$ is identified with $\mathfrak{n}/(\mathfrak{n}^2,p)$, and called the (mod $p$) \emph{cotangent space} of $A$.
    It is naturally isomorphic to the cotangent space of $\bar A := A/pA$. 
\end{itemize}

\subsection{The tangent space of $R$} 

In this section, we compute the tangent space of $R$. In order to do this, we first recall Bella\"iche's computation of the tangent space of the unrestricted deformation ring $R_\Db$ \cite{bellaiche2012}. 

Let $J^\red_{\Db} \subset R_\Db$ denote the reducibility ideal and $R_\Db^\red = R_\Db/J^\red_\Db$, and let $E_\Db^u=\sm{R_\Db}{B_\Db}{C_\Db}{R_\Db}$ be the $R_\Db$-GMA structure on $E_\Db^u$. On the other hand, let $\frt_{R_\Db}^\irr$ be the cokernel of the natural map $\frt_{R_\Db^\red} \to \frt_{R_\Db}$; define $\frt_R^\irr$ analogously as the cokernel of $\frt_{R^\red} \to \frt_R$. We will address these tangent spaces mainly through their dual, which is the irreducible subspace of the cotangent space, 
\[
(\frt_R^\irr)^* \subset \frt_R^*, \qquad (\frt_{R_\Db}^\irr)^* \subset \frt_{R_\Db}^*. 
\]

We will access these irreducible subspaces as follows. According to Proposition \ref{prop: red ideal generators}, the GMA-multiplication map induces a surjective $R_\Db$-module homomorphism 
\[
B_\Db \otimes_{R_\Db} C_\Db \onto J_\Db^\red, \quad b \otimes c \mapsto b \cdot c.
\]
As a result, there is a composite surjection 
\begin{equation}
\label{eq:BC=J unobstructed}
    B_\Db/\m_\Db B_\Db \otimes_{\F_p} C_\Db/\m_\Db C_\Db \onto J^\red_\Db/\m_\Db J^\red_\Db \rsurj (\frt_{R_\Db}^\irr)^*
\end{equation}
of $\F_p$-vector spaces. Bella\"iche interprets this surjection in terms of cup products in Galois cohomology. 
\begin{prop}[{Bella\"iche \cite[Theorem A and \S4.1.1]{bellaiche2012}}]
\label{prop: Bellaiche sequence}
There is an exact sequence
\begin{equation}
\label{eq:Bellaiche sequence}
0 \to \frt_{R_\Db}^\irr \xrightarrow{\iota} H^1(\Z[1/Np], \F_p(1)) \otimes_{\F_p}  H^1(\Z[1/Np], \F_p(-1)) \buildrel{\cup}\over\lra 
H^2(\Z[1/Np], \F_p).
\end{equation}
where the final map is the cup product.
Moreover, under natural identifications 
\begin{equation}
    \label{eq: BC duals}
    B_\Db/\m_\Db B_\Db \cong (H^1(\Z[1/Np], \F_p(1)))^*, \quad C_\Db/\m_\Db C_\Db \cong (H^1(\Z[1/Np], \F_p(-1)))^*,
\end{equation}
the map $\iota$ is identified with the dual of \eqref{eq:BC=J unobstructed}. 
\end{prop}

Applying the proposition under our running assumptions $b_1 \cup c_0=0$ (see Lemma \ref{lem: log ell1 is zero}) and $b_0 \cup c_0 \ne 0$ (see Proposition \ref{prop: rk2 assumption}), we have the following

\begin{lem}
\label{lem:Bellaiche}
There is an element of $f \in B_\Db/\m_\Db B_\Db \otimes_{\F_p} C_\Db/\m_\Db C_\Db$ satisfying 
\begin{enumerate}[label=(\roman*)]
    \item under the dualities \eqref{eq: BC duals}, $f(b_0 \otimes c_0) \ne 0$ and $f(b_1 \otimes c_0)=0$ in $\F_p$, and
    \item $f$ maps to $0$ under \eqref{eq:BC=J unobstructed}.
\end{enumerate}
\end{lem}

\begin{proof}
Since $b_0 \cup c_0 \ne 0$, it is not in the image of $\iota$, so there is an element $\lambda$ of the dual of $H^1(\Z[1/Np], \F_p(1)) \otimes_{\F_p}  H^1(\Z[1/Np], \F_p(-1))$ such that $\lambda(b_0 \otimes c_0) \ne 0$ and such that $\lambda$ is zero on the image of $\iota$. In particular, since $b_1 \cup c_0$ is zero, $b_1 \otimes c_0$ is in the image of $\iota$ and $\lambda(b_1 \otimes c_0)=0$. Our identifications give an isomorphism between the dual of $H^1(\Z[1/Np], \F_p(1)) \otimes_{\F_p}  H^1(\Z[1/Np], \F_p(-1))$ and $B_\Db/\m_\Db B_\Db \otimes_{\F_p} C_\Db/\m_\Db C_\Db$, and we can take $f$ to be the image of $\lambda$ under this identification.
\end{proof}

Now we apply the computations in $R_\Db$ above, under the surjection $R_\Db \rsurj R$, to calculate the irreducible subspace of the mod $p$ cotangent space of $R$. Along the way, we specify minimal sets of generators for $J^\red \subset R$ and $\bar J^\red \subset \bar R$.

\begin{prop}
\label{prop: barJ^red is principal}
The ideal $\bar J^\red \subset \bar R$ is principal, generated by the non-zero image of $b_{\gamma_1} \cdot c_{\gamma_0}$ under $R \rsurj \bar R$. In particular, the image of $b_{\gamma_1} c_{\gamma_0}$ in $\frt_R^*$ generates the 1-dimensional subspace $(\frt_R^\irr)^*$. In contrast, the ideal $J^\red \subset R$ is not principal, and is generated by $\{b_{\gamma_0} c_{\gamma_0}, b_{\gamma_1}c_{\gamma_0}\}$. 
\end{prop}

    \begin{proof}
    First, we claim that $b_{\gamma_0} c_{\gamma_0}$ and $b_{\gamma_1}c_{\gamma_0}$ generate $J^\red$. Due to Proposition \ref{prop: red ideal generators} (which applies to any generalized matrix algebra and its scalar ring), this follows from the fact that $\{b_{\gamma_0},b_{\gamma_1}\}$ generate $B$ and $\{c_{\gamma_0}\}$ generates $C$, as recorded in Proposition \ref{prop: BC bases}. 
    
    Next, we claim that $b_{\gamma_1}c_{\gamma_0}$ generates $\bar J^\red$. Due to the previous claim, it suffices to prove that $b_{\gamma_0} c_{\gamma_0}$ is a multiple of $b_{\gamma_1} c_{\gamma_0}$ in $\bar J^\red$. Just as in \eqref{eq:BC=J unobstructed}, there is a similar map for $J^\red$ fitting into a commutative diagram
    \begin{equation}
    \label{eq: red ideal diagram}
    \xymatrix{
     B_\Db/\m_\Db B_\Db \otimes_{\F_p} C_\Db/\m_\Db C_\Db \ar@{->>}[r] \ar@{->>}[d] & J^\red_\Db/\m_\Db J^\red_\Db \ar@{->>}[d]  \\
    B/\m B \otimes_{\F_p} C/\m C \ar@{->>}[r] & J^\red/\m J^\red 
    }
    \end{equation}
    Under the interpretation of $B_\Db/\m_\Db B_\Db$ and $C_\Db/\m_\Db C_\Db$ as dual vector spaces found in \eqref{eq: BC duals}, the left vertical arrow is the dual of the inclusion of subspaces of the Galois cohomology groups. By \cite[Lem.\ 3.10.3]{WWE5}, we can identify these subspaces: the basis $\{b_{\gamma_0},b_{\gamma_1}\}$ of $B/\m B$ is dual to the basis $\{b_0,b_1\}$ of Galois cohomology; and $\{c_{\gamma_0}\}$ is a basis of $C/\m C$, dual to $\{c_0\}$. 
    
    Now consider the element $f \in  B_\Db/\m_\Db B_\Db \otimes_{\F_p} C_\Db/\m_\Db C_\Db$ from Lemma \ref{lem:Bellaiche}. The image of $f$ in $B/\m B \otimes_{\F_p} C/\m C$ is of the form $(x b_{\gamma_0}+y b_{\gamma_1}) \otimes c_{\gamma_0}$ for some $x, y \in \F_p$. Since $f(b_0 \otimes c_0) \ne 0$ and $f(b_1 \otimes c_0)=0$ it follows that $x \ne 0$ and $y=0$. Then diagram \eqref{eq: red ideal diagram} and the fact that $f$ maps to $0$ in $(\frt_\Db^\irr)^* \subset \frt_\Db^*$ imply that 
    \[
    b_{\gamma_0} \cdot c_{\gamma_0}= -x^{-1}y b_{\gamma_1} \cdot c_{\gamma_0} =0
    \]
    in $\frt_R^*$, completing the claim that $b_{\gamma_1} \cdot c_{\gamma_0}$ generates $\bar J^\red$. 
    
    Next we prove that $\frt_R^\irr \cong (\bar J^\red/\bar \m \bar J^\red)^*$ is 1-dimensional. It remains to show that $\frt_R^\irr$ is not zero. This follows from the existence of the irreducible first-order pseudodeformation $D_1$ of $\psi(\omega \oplus 1)$ of $\F_p[\ep_1]$ constructed in Lemma \ref{lem: D1 construction}, because the lemma showed that $D_1$ satisfies $\US_N$. 
    
    It remains to show that $J^\red$ is not principal. If it were principal, then because the image of $b_{\gamma_1}c_{\gamma_0}$ in $\bar J^\red$ is a generator, $b_{\gamma_1}c_{\gamma_0} \in J^\red$ would be a generator. But $b_{\gamma_1}c_{\gamma_0}$ vanishes under $R \rsurj R_{\ell_0}$ because the Galois pseudorepresentations parameterized by $R_{\ell_0}$ are unramified at $\ell_1$ (hence $b_{\gamma_1}$ maps to zero in the global level $\ell_0$ $R_{\ell_0}$-GMA). This would imply that the pseudorepresentation supported by $R_{\ell_0}$ is reducible. But this implication is known to be false: the Galois representation supported by the level $\Gamma_0(\ell_0)$ cusp form $f$ of Assumption \ref{P1: assump: main}(3) is irreducible. 
    \end{proof}
    
    Now we can calculate the whole tangent space of $R$. 

\begin{prop}
\label{prop: R-cotangent} 
The $\F_p$-dimension of $\frt_R$ is 2, with a basis given by the two maps $\varphi_{D^\red}, \varphi_{D_1}: R \to \F_p[\ep]/(\ep^2)$ specified in Lemma \ref{lem: reducible quotient} and Lemma \ref{lem: D1 construction}, respectively. More precisely:
\begin{enumerate}
    \item The subspace $\frt_{R^\red} \subset \frt_{R}$ is 1-dimensional and spanned by $D^\red$.
    \item The space $\frt_R^\irr$ is one-dimensional and the element $D_1$ of $\frt_R$ maps to a generator of it under the natural surjection $\frt_R \rsurj \frt_R^\irr$. 
\end{enumerate}
\end{prop}

\begin{proof} 
Since there is an exact sequence
\[
0 \to \frt_{R^\red} \to \frt_R \to \frt_{R}^\irr \to 0
\] 
it is enough to show (1) and (2). Part (1) follows from the isomorphisms
\[
R^\red /pR^\red \cong \F_p[y]/(y^2)
\]
of Lemma \ref{lem: reducible quotient}.

Part (2) follows from Proposition \ref{prop: barJ^red is principal} (see the end of its proof) along with the fact that $D_1$ is irreducible, which is inherent to its construction in Lemma \ref{lem: D1 construction}. 
\end{proof}

\subsection{The $R$-module $C$ is torsion}
Having characterized $\bar R/\bar \m^2$, we begin toward calculating $\bar R/\bar \m^3$, ultimately showing in Corollary \ref{cor: constraint on S} that $\bar \m^2/\bar \m^3$ is at most 1-dimensional. The first step is the following proposition, which will be used to show that lifts of certain cotangent vectors to $\bar R/\bar \m^3$ must have product zero. The idea is that $C$ is a factor of the irreducible cotangent vector under \eqref{eq: red ideal diagram}, so it will be useful to show that $C$ is killed by the reducible cotangent vector.

In Proposition \ref{prop: BC bases}, we saw that $C$ is a cyclic $R$-module, generated by the element $c_{\gamma_0} \in C$. An important consequence of our running assumption $a_0 \cup c_0 \ne 0$ (see Proposition \ref{prop: rk2 assumption}) is that $C$ is not a free $R$-module. 

\begin{prop}
\label{prop: C is torsion}
The $R$-module $C$ is cyclic and not free. In fact, the annihilator of $C \otimes_{R,D^\red} \F_p[\ep_1]$ is $\epsilon$ (here the tensor product is with respect to the ring map $\varphi_{D^\red}: R \to \F_p[\ep_1]$ defined in Lemma \ref{lem: reducible quotient}).
\end{prop}

\begin{proof} 
Because $BC=J^\red$ is non-zero, $C$ is also non-zero. 

Let $\bar C := C \otimes_{R,D^\red} \F_p[\ep_1]$. We will show that $\bar{C}$ is not a free $\F_p[\ep_1]$-module (in which case it must be isomorphic to $\F_p$), which implies that $C$ is not free as an $R$-module. To set up a contradiction, assume that $\bar{C}$ is a free $\F_p[\ep_1]$-module; we will show that this contradicts the assumption $a_0 \cup c_0 \ne 0$. 

    We know by Nakayama's lemma and Proposition \ref{prop: BC bases} that $\bar C$ is a cyclic $\F_p[\ep_1]$-module with generator $c_{\gamma_0}$. Because $D^\red$ is reducible, the $B$-coordinate $B_{E'}$ of the $\F_p[\ep_1]$-GMA $E'=E \otimes_{R, D^\red} \F_p[\ep_1]$ is a two-sided ideal; indeed, the reducibility implies that $B_{E'} \cdot C_{E'}$ is the reducibility ideal in $\F_p[\ep_1]$, which is the zero ideal. The quotient by $B_{E'}$ has the form 
    \[
    E'' := \ttmat{\F_p[\ep_1]}{}{\bar C}{\F_p[\ep_1]} \isoto \ttmat{\F_p[\ep_1]}{}{\F_p[\ep_1]}{\F_p[\ep_1]} \subseteq M_2(\F_p[\ep_1]) 
    \]
    (where we used $c_{\gamma_0}$ as a generator of $\bar C$ to draw the isomorphism) receiving a homomorphism from $\F_p[\ep_1][G_\Q]$ of the form 
    \[
    \ttmat{\omega(1 + \ep a_0)}{0}{\omega(c\up1 + \ep c\up2)}{1 - \ep a_0}. 
    \]
    
    In the coordinate expression, $c\up1$ appears because we have made a choice of GMA coordinates of $E$ compatible with $E \to E_1$ as in Proposition \ref{prop: GMA structure for E1}, and we use these coordinates under the surjection $E \rsurj E'$. 
    
    We have an equality of 2-cocycles valued in $\F_p(-1)$, 
    \[
    -dc\up2 = a_0 \smile c\up1 + c\up1 \smile (-a_0).
    \]
    The right hand side is in the cohomology class of $2 a_0 \cup c_0$. But our assumption $\rk_{\Z_p} \bT_{\ell_0} = 2$ implies that $a_0 \cup c_0 \neq 0$ in $H^2(\Z[1/Np], \F_p(-1))$ by Proposition \ref{prop: rk2 assumption}. Therefore such a $c\up2$ cannot exist. 
\end{proof}

\subsection{The ring $S:=\bar R/\bar{\mathfrak{m}}^3$} 
\label{subsec: S form}

Let $S := \bar R/\bar \m^3$. Because $S$ surjects onto $\bar R/\bar \m^2$ and Proposition \ref{prop: R-cotangent} describes the 2-dimensional cotangent space $\frt_R^* = \bar \m/\bar\m^2$, there are equivalences
\[
\dim_{\F_p} S = 3 \iff \dim_{\F_p} \bar R = 3 \iff \bar R \cong S \cong \frac{\F_p[x,y]}{(x^2,xy, y^2)},
\]
and Proposition \ref{prop: R-cotangent} characterizes $\bar R$ completely in this case. Otherwise, a priori we know that $\dim_{\F_p} S \leq 6$. Our goal is to refine this bound to $\dim_{\F_p} S \leq 4$ and to show that $(\bar J^\red)^2 = \bar \m^2$.

For an ideal $I$ in $R$ or $\bar R$, let $I_S \subset S$ denote its image in $S$. Note that $\dim_{\F_p} S = 3+ \dim_{\F_p}(\m_S^2)$.

\begin{prop}
\label{prop: y2 span}
The inclusion of ideals $(J^\red_S)^2 \subset \m_S J^\red_S$ is an equality. 
\end{prop}

\begin{proof}
Let $C_S := C \otimes_R S$, and likewise $B_S := B \otimes_R S$. 

We claim that $\m_S C_S \subset J^\red_S C_S$, which we will derive from Proposition \ref{prop: C is torsion}.  Proposition \ref{prop: C is torsion}, translated into our current notation using $S$, states that the maximal ideal of $S^\red$ kills $C_S^\red$. Lifting this result from $S^\red$-modules to $S$-modules, we find that $\m_S C_S \subset J^\red_S C_S$, which is the desired result. 

We derive from the equality $\m_S C_S = J^\red_S C_S$ that, for all $x \in \m_S$, there exists some $z \in J^\red_S$ such that
\begin{equation}
    \label{eq: mS to JredS on C}
    x c_{\gamma_0} = z c_{\gamma_0},
\end{equation}
and that every element of $\m_S C_S$ has this form because $c_{\gamma_0}$ generates $C_S$. We apply this to the surjection of $S$-modules
\[
C_S \otimes_S B_S \rsurj J^\red_S, 
\]
also using that $c_{\gamma_0}b_{\gamma_1}$ is a generator of the principal ideal $J^\red_S$ (Proposition \ref{prop: barJ^red is principal}). Namely, finding that every element of $\m_S J^\red_S = \m_S C_S B_S$ has the form 
\[
x (c_{\gamma_0} b_{\gamma_1} s) = (x c_{\gamma_0})(b_{\gamma_1}s) = z(c_{\gamma_0}b_{\gamma_1}) s \in (J^\red_S)^2,
\]
for some $s \in S$, and with $x$ and $z$ as in \eqref{eq: mS to JredS on C}. 
\end{proof}

\begin{cor}
\label{cor: constraint on S} 
Either $\dim_{\F_p} S = 3$ or $\dim_{\F_p} S = 4$. In general,
\[
\dim_{\F_p} S = \dim_{\F_p} J^\red_S + 2 = \dim_{\F_p} (J^\red_S)^2 + 3.
\]
Consequently, $(\bar J^\red)^2 = \bar \m^2$ in $\bar R$. 
\end{cor}

\begin{proof}
Because $S^\red = S/J^\red_S$ is $2$-dimensional, we have $\dim_{\F_p} S = 2 + \dim_{\F_p} J^\red_S$ in general. 

Because $J^\red_S \subset \m_S$ and $\m_S^3 = 0$, we have a filtration
\[
J^\red_S \supset \m_S J^\red_S \supset \m_S^2 J^\red_S = 0. 
\]
The principality of $J^\red_S$ (Proposition \ref{prop: barJ^red is principal}) implies that 
\[
\dim_{\F_p} \frac{J^\red_S}{\m_S J^\red_S} = 1.
\]
The equality $\m_S J^\red_S = (J^\red_S)^2$ of Proposition \ref{prop: y2 span} implies that $\dim_{\F_p} \m_S J^\red_S \leq 1$.  

The final claim follows from Nakayama's lemma. 
\end{proof}

\section{Galois-theoretic implications of $\dim_{\F_p}R/pR \ge 4$}
\label{sec: R and Galois}

Throughout this section, we assume that $\dim_{\F_p}R/pR \ge 4$ (or, equivalently, that $\dim_{\F_p} S =4$, where $S=R/(p,\m^3)$) and derive consequences for Galois cohomology. The main results are Propositions \ref{prop: a1 vanish at ell1 when S large} and \ref{prop: key alpha beta equation}, which together essentially prove one direction of Theorem \ref{P1: thm: main intro} from the introduction. Of note, Proposition \ref{prop: key alpha beta equation} gives an optimal presentation of $S$. 

\subsection{A GMA over $S$ when $\dim_{\mathbb{F}_p} S = 4$}
Henceforth, let $y$ be the image of $b_{\gamma_1}\cdot c_{\gamma_0}$ in $S$, which generates the principal ideal $J^\red_S \subset S$. According to Corollary \ref{cor: constraint on S}, the $\F_p$-dimension of $J^\red_S$ is $2$. Since $J^\red_S$ is principal, its annihilator $\Ann_S(J^\red_S)$ is also 2-dimensional. Consider the ring homomorphism
\begin{equation}
\label{eq: S mod Ann Jred}
S \rsurj \frac{S}{\Ann_S(J^\red_S)} \cong \F_p[\ep_1], 
\end{equation}
where the isomorphism $S/\Ann_S(J^\red_S) \isoto \F_p[\ep_1]$ is determined by $y \mapsto \ep$. This is possible because, using Corollary \ref{cor: constraint on S}, $y^2$ spans $(J^\red_S)^2 = \m_S^2$, which is non-zero in $S$ under the assumption that $\dim_{\F_p} S = 4$. 

\begin{defn}
We set up the following coordinates for objects within $S$. 
\begin{itemize}
    \item We define a $\F_p[\ep_1]$-valued pseudorepresentation $D_y : G_{\Q,Np} \to \F_p[\ep_1]$ by associating it to the local homomorphism 
\[
\varphi_{D_y} : R \rsurj S \rsurj \F_p[\ep_1]
\]
determined by the isomorphism $\frac{S}{\Ann_S(J^\red_S)} \isoto \F_p[\ep_1]$ above. 
\item We also allow ourselves to identify $J^\red_S$ with $\F_p[\ep_1]$, as $S$-modules where $\F_p[\ep_1]$ has structure map $\varphi_{D_y}$, under the isomorphism
\[
\F_p[\ep_1] \buildrel{\eqref{eq: S mod Ann Jred}}\over\isoto \frac{S}{\Ann_S(J^\red_S)} \isoto J^\red_S,
\]
where the rightmost isomorphism is determined by $s \mapsto ys$. 
\item Since the image of $b_{\gamma_0} \cdot c_{\gamma_0}$ in $S$ is in $\m_S J_S^\red=y^2S$ (Proposition \ref{prop: barJ^red is principal}), we see that there is a unique $\eta \in \F_p$ such that
\[
b_{\gamma_0} \cdot c_{\gamma_0} = \eta y^2.
\]
\end{itemize}
\end{defn}

We call the map \eqref{eq: S mod Ann Jred} and the following maps out of $B_S$ and $C_S$, collectively, \emph{coordinate maps}. 
\begin{lem}\label{lem:coord maps}
Assume that $\dim_{\F_p}S=4$. There are surjective $S$-module homomorphisms (which we will call \emph{coordinate maps}) 
\[
B_S \onto \F_p \oplus \F_p[\ep_1], \quad s_0 b_{\gamma_0} + s_1 b_{\gamma_1} \mapsto (\bar{s}_0, \varphi_{D_y}(s_1)),
\]
where $\bar{s}_0 \in \F_p$ is the image of $s_0$ under the augmentation $S \to \F_p$, and
\[
C_S \onto \F_p[\ep_1],\quad s c_{\gamma_0} \mapsto \varphi_{D_y}(s).
\]

Using these surjections $B_S \onto \F_p \oplus \F_p[\ep_1]$ and $C_S \onto \F_p[\ep_1]$ and the identification $J^\red_S=\F_p[\ep_1]$ sending $y$ to $1$, the GMA-multiplication map 
\[
B_S \otimes_S C_S \to J_S^\red
\]
induces the map
\begin{equation}
\label{eq:BC in E'_S}
(\F_p \oplus \F_p[\ep_1]) \otimes_{\F_p[\ep_1]} \F_p[\ep_1] \to  \F_p[\ep_1]
\end{equation}
given by 
\[
(u,v) \otimes z \mapsto \eta \ep uz + vz.
\]
\end{lem}

\begin{proof}
The only coordinate map that does not obviously exist as defined is that of $B_S$: $B_S$ is non-cyclic and generated by $\{b_{\gamma_0}, b_{\gamma_1}\}$, and we must show that any relation between the generators is sent by the coordinate map to $0$. First, observe that any relation $g b_{\gamma_0} + h b_{\gamma_1} = 0 \in B_S$ (for $g,h \in S$) must have $g,h \in \m_S$, since $B_S$ is not cyclic. Therefore, no relation $g b_{\gamma_0} + h b_{\gamma_1}$ can possibly map to something non-zero under the coordinate map for $B_S$, since this would imply that
\[
0 = g b_{\gamma_0}c_{\gamma_0} + h b_{\gamma_1}c_{\gamma_0} = g\eta y^2 + h y = hy \text{ in } J^\red_S,
\]
for some $h$ such that $\varphi_{D_y}(h) \neq 0$, contradicting $\Ann_S(y) = \ker \varphi_{D_y}$. Consequently, the coordinate map for $B_S$ is well defined. 

It remains to verify that the square of surjections
\[
\xymatrix{
B_S \otimes_S C_S \ar[r] \ar[d] & J_S^\red \ar[d] & \\
(\F_p \oplus \F_p[\ep_1]) \otimes_{\F_p[\ep_1]} \F_p[\ep_1] \ar[r] &  \F_p[\ep_1] & 
}
\]
commutes, which we can check on the generating set $\{b_{\gamma_0} \otimes c_{\gamma_0},b_{\gamma_1} \otimes c_{\gamma_0}\}$ of $B_S \otimes_S C_S$. 

The coordinates of $b_{\gamma_0} \otimes c_{\gamma_0}$ are $(1,0) \otimes 1$, which maps to $\eta\ep \in \F_p[\ep_1]$; on the other hand, $b_{\gamma_0}c_{\gamma_0} \in J^\red_S$ has the form $\eta y^2$ by definition of $\eta$, which also maps to $\eta \ep \in \F_p[\ep_1]$ under the coordinate map for $J^\red_S$. 

The coordinates of $b_{\gamma_1} \otimes c_{\gamma_0}$ are $(0,1) \otimes 1$, which maps to $1 \in \F_p[\ep_1]$; on the other hand, $b_{\gamma_1}c_{\gamma_0} \in J^\red_S$ equals $y$, which also maps to $1 \in \F_p[\ep_1]$ under the coordinate map $J^\red_S \isoto \F_p[\ep_1]$. 
\end{proof}

Let $E_S'$ denote the $S$-GMA
\begin{equation}
    \label{eq: ESprime def}
    E_S'=\ttmat{S}{\F_p \oplus \F_p[\ep_1]}{\F_p[\ep_1]}{S}
\end{equation}
where $\F_p[\ep_1]$ is a $S$-module via the map $\varphi_{D_y} :S \to \F_p[\ep_1]$, and where the GMA-multiplication map is given by \eqref{eq:BC in E'_S}. By Lemma \ref{lem:coord maps}, the coordinate maps comprise a surjective morphism of $S$-GMAs
\[
E \otimes_R S \onto E_S'.
\]

\subsection{The coordinates of a $S$-GMA valued Galois representation when $\dim_{\mathbb{F}_p} S = 4$} 

Now consider the Cayley--Hamilton representation $\rho_S':G_{\Q,Np} \to (E_S')^\times$ obtained as composition of the universal Cayley--Hamilton representation $\rho: G_{\Q,Np} \to E^\times$ with $E \rsurj E_S'$. We are interested in endowing it with coordinates and comparing these coordinates to the 1-reducible GMA representation $\rho_1 : G_{\Q,Np} \to E_1^\times$ of \eqref{eq: define d1}. 

To this end, the coordinates of $E_S'$ suffice, modulo the need for complete coordinates for $S$, which we now supply. To introduce this definition, we point out that $\{D^\red,D_y\}$ is a basis of $\frt_R$ according to Proposition \ref{prop: R-cotangent}, because $D^\red$ is reducible, $D_y$ is irreducible, and $\dim_{\F_p} \frt_R = 2$.
\begin{defn}
\label{defn: x}
Let $x \in S$ denote a generator for $\Ann_S(J^\red_S)$ whose image $\bar x \in \frt_R^* = \m_S/\m_S^2$ makes $\{\bar x, \bar y\} \subset \frt_R^*$ a dual basis to $\{D^\red, D_y\}$. 
\end{defn}

Here are the important properties of this choice of $x$; we also justify in this lemma that such a choice of $x$ exists. 
\begin{lem}
\label{lem: S coordinates}
Assume $\dim_{\F_p} S = 4$. A choice of $x \in S$ as in Definition \ref{defn: x} induces a presentation of $S$, 
\[
\frac{\F_p\lb X,Y\rb}{(X^2 - \mu Y^2,XY, Y^3)} \isoto S, \qquad X \mapsto x, Y \mapsto y,
\]
for some unique $\mu \in \F_p$. The possible choices of $x$ are a torsor under the 1-dimensional $\F_p$-vector space $(y^2) = (J^\red_S)^2$. 
\end{lem}

\begin{proof}
    The ideal $\Ann_S(J^\red_S) \subset S$ is contained in $\m_S$ because $J^\red_S \neq 0$. On the other hand, $\Ann_S(J^\red_S)$ is not contained in $\m_S^2$ because $\dim_{\F_p} \m_S^2 = 1$ while $\dim_{\F_p} \Ann_S(J^\red_S) = 2$. Therefore $\Ann_S(J^\red_S)$ has 1-dimensional image under the projection $\m_S \rsurj \m_S/\m_S^2 = \frt_R^*$. This image is complementary to $(\frt_R^\irr)^* = \langle \bar y\rangle$ because $y^2 \neq 0$, yet every element of $\Ann_S(J^\red_S)$ kills the generator $y$ of $J^\red_S)$. Similarly, $\Ann_S(J^\red) \subset S$ is the kernel of $\varphi_{D_y}$, so there exists a generator $x$ of $\Ann_S(J^\red_S)$ such that 
    \[
    \{\bar x, \bar y\} \text{ is a dual basis to } \{D^\red, D_y\}. 
    \]
    In particular, $x$ and $y$ generate $S$ as an $\F_p$-algebra, and we have a surjection $\phi: \F_p\lb X,Y\rb \rsurj S$ via $X \mapsto x, Y \mapsto y$. 
    
    The next goal is to show the existence of $\mu \in \F_p$ such that $(X^2 - \mu Y^2,XY, Y^3) \subset \ker \phi$. This will suffice to prove the presentation, because the quotient if $\F_p\lb X,Y\rb$ by this ideal is 4-dimensional over $\F_p$, like $S$. 
    
    Clearly $Y^3 \in \ker \phi$, since $\m_S^3 = 0$. Likewise, we know that $XY \in \ker \phi$ because $x \in S$ satisfies $x J^\red_S = 0$ by definition, and $y$ is a generator of $J^\red_S$. Finally, the existence of $\mu \in \F_p$ such that $X^2 - \mu Y^2 \in \ker \phi$ follows from the principality of $J^\red_S = (y)$ and the fact that $\m_S^2 = (J^\red_S)^2 = (y^2)$ is 1-dimensional over $\F_p$ (Propositions \ref{prop: barJ^red is principal} and \ref{prop: y2 span}), since we know from the start that $x^2 \in \m_S^2$. 
    
    The claim that the various choices of $x$ satisfying Definition \ref{defn: x} are a torsor under $(y^2)$ follows from the fact that $\{x,y^2\}$ is an $\F_p$-basis for $\Ann_S(J^\red)$, and that the projection of $gx + hy^2$ ($g,h\in \F_p$) to $\frt_R^*$ equals $g \bar x$. This also makes the uniqueness of $\mu$ clear, since $x^2$ only depends upon $\bar x$. 
\end{proof}

Combining the coordinates of $S$ from Lemma \ref{lem: S coordinates} with the coordinates for the off-diagonal parts of $E_S'$ from \eqref{eq: ESprime def}, we produce a coordinate-wise description of $\rho_S' : G_{\Q,Np} \to (E_S')^\times$.
\begin{equation}
\label{eq: E_S coordinates}
\rho_S' = \ttmat
{\omega(1 + y a\upp1 + y^2 a\upp2 + x a_0)}
{(b_0\up1, b\up1 + y b\upp2)}
{\omega(c\up1 + y c\upp2)}
{1 + yd\upp1 + y^2d\upp2 - x a_0'} 
\end{equation}
for some cochains 
\begin{itemize}
    \item $a\upp1, a\upp2, d\upp1, d\upp2: G_{\Q,Np} \to \F_p$,
    \item $b\upp2 : G_{\Q,Np} \to \F_p(1)$,
    \item $c\upp2 : G_{\Q,Np} \to \F_p(-1)$,
\end{itemize}
and cocycles $a_0$, $b\up1 = b_1\up1$, $b_0\up1$, and $c\up1$ defined in Definition \ref{P1: defn: pinned cocycles}. The reason that we find these previously defined cocycles among these coordinates is 
\begin{itemize}
    \item for $a_0$: the tangent vector $D^\red \in \frt_R$ is dual to $\bar x \in \frt_R^*$, and we observe that the pseudorepresentation induced by $\rho_S \otimes_{S, \varphi_{D^\red}} \F_p[\ep_1]$ is exactly $D^\red$ (keep in mind that $\varphi_{D^\red}(x) =\ep$, $\varphi_{D^\red}(y) = 0$)
    \item for the remaining cocycles: the presence of dual bases of the dual vector spaces of \eqref{eq: BC duals} (see the proof of Proposition \ref{prop: barJ^red is principal}), along with the normalization of both the generators of $B_S, C_S$ and the cocycles $b\up1$, $b_0\up1$, and $c\up1$ in terms of the elements $\gamma_0, \gamma_1$ of inertia groups. 
\end{itemize}

Next, we are interested in identifying $a\upp1$ with the $a\up1$ constructed in Lemma \ref{P1: lem: produce a1}, which implies the similar identification of $d\upp1$ with $d\up1 = b\up1 c\up1 - a\up1$. This will produce a surjection from $\rho_S'$ onto the $\rho_1 : G_{\Q,Np} \to E_1^\times$ constructed in Lemma \ref{lem: D1 construction} and implies that $D_y = D_1 := \psi(\rho_1)$. The key is the comparison of differential equations: the homomorphism property of $\rho_S'$ implies that $a\upp1$ satisfies the differential equation
\begin{equation}
    \label{eq: a1 of rhoS boundary}
    -da\upp1 = b\up1 \smile c\up1, 
\end{equation}
which $a\up1$ also satisfies (Lemma \ref{P1: lem: produce a1}). We note that the fact that $\rho'_S$ has constant determinant $\omega$ implies that $d\upp1 = b\up1c\up1 - a\up1$, just as in the discussion of $d\up1$ in \S\ref{subsec: construct rho1}. 

There are even more differential equations implied by the fact that $\rho_S'$ is a homomorphism, 
\begin{gather}
    \label{eq: c2}
    -dc\upp2 = c\up1 \smile a\upp1 + d\upp1 \smile c\up1  \\ 
    \label{eq: a2}
    -da\up2 = a^{(1)} \smile a^{(1)} + b\up1 \smile c^{(2)} + (b^{(2)} + \eta b_0\up1) \smile c^{(1)} + \mu a_0 \smile a_0  \\
    \label{eq: b2}
    -db\up2 = a\up1 \smile b\up1 + b\up1 \smile d\up1.
\end{gather}
In particular, the 2-cocycles on the right-hand-sides of these equations are coboundaries. 

\begin{lem}
The two 1-cochains $a\up1, a\upp1 : G_{\Q,Np} \to \F_p$ are equal. Consequently, $D_y = D_1 : G_{\Q,Np} \to \F_p[\ep_1]$. 
\end{lem}

\begin{proof}
Lemma \ref{P1: lem: produce a1} has listed characterizing properties (1)-(3) of $a\up1$. We will show that $a\upp1$ satisfies them as well. 

Property (1) is satisfied in \eqref{eq: a1 of rhoS boundary}. 

We will deduce property (2) from the finite-flat property that $\rho_S'\vert_p$, which it satisfies because it is a quotient GMA of the universal $\US_N$ GMA over $\Db$. By design, the 0-cochain $x_{c\up1} \in C^0(\Z[1/Np], \F_p(-1))$ conjugates $c\up1$ so that it vanishes on $G_p$, in the sense that the conjugation of $\rho_S'$ by $\ttmat{1}{0}{x}{1}$ is upper-triangular on $G_p$ modulo the ideal generated by the image of $y C_S$ in the $C$-coordinate. Then Proposition \ref{prop: finite-flat implies upper tri} implies the vanishing of $(a\upp1 + b\up1\smile x_{c\up1})\vert_{I_p}$, which is property (2). 

Because of the injection $H^2(\Z[1/Np], \F_p(-1)) \rinj H^2(\Q_{\ell_0}, \F_p(-1))$ of Lemma \ref{lem: Hasse} and the vanishing of $b\up1$ at $\ell_0$, equation \eqref{eq: c2} implies that $a\upp1\vert_{\ell_0}$ is a cocycle and $(2 c\up1 \smile a\upp1)\vert_{\ell_0}$ is a 2-coboundary. Since the cup product on $H^1(\Q_{\ell_0},\F_p) \times H^1(\Q_{\ell_0},\F_p)$ is alternating in the sense of Lemma \ref{lem: tate duality at ell0}, we conclude that $[a\upp1\vert_{\ell_0}]$ and $[c\up1\vert_{\ell_0}] \cup [\zeta]$ are colinear in $H^1(\Q_{\ell_0}, \F_p)$ for any choice of $\zeta \in H^0(\Q_{\ell_0}, \F_p(1))$, which is property (3). 

To deduce that $D_y = D_1$, observe that the equality $a\upp1 = a\up1$ implies that the pseudorepresentation $\phi(\rho_S') : G_{\Q,Np} \to S$ induces $D_1$ via
\[
S \to \F_p[\ep_1], \quad x \mapsto 0, y \mapsto \ep, 
\]
while, on the other hand, this map $S \to \F_p[\ep_1]$ is exactly the same as $\varphi_{D_y}$. 
\end{proof}

There are even more implications of the differential equations implied by the existence of $\rho_S'$. In particular, \eqref{eq: b2} has the following consequence about the restriction $a\up1|_{\ell_1}$ (note that $a\up1|_{\ell_1}$ is a cocycle since $da\up1 = b\up1 \smile c\up1$ and $c\up1|_{\ell_1}=0$).

\begin{prop}
\label{prop: a1 vanish at ell1 when S large}
There exists a cochain $b\up2$ satisfying \eqref{eq: b2} if and only if $a\up1|_{\ell_1}~=~0.$ In particular, if $\dim_{\F_p} S=4$, then $a\up1|_{\ell_1}=0.$
\end{prop}

\begin{proof}
Since $a\up1|_{\ell_1}$ is an element of $H^1(\Q_{\ell_1},\F_p)$, which is Tate-dual to $H^1(\Q_{\ell_1},\F_p(1))$, and since $b\up1|_{\ell_1}$ is a basis for $H^1(\Q_{\ell_1},\F_p(1))$, the cup product $a\up1|_{\ell_1}~\cup~b\up1|_{\ell_1}$ vanishes in $H^2(\Q_{\ell_1},\F_p(1))$ if and only if $a\up1|_{\ell_1}=0.$

The existence of a cochain $b\up2$ satisfying \eqref{eq: b2} is equivalent to
\begin{equation}
\label{eq:b2 in H2}
    a\up1 \cup b\up1 + b\up1 \cup d\up1
\end{equation}
vanishing in $H^2(G_{\Q,Np},\F_p(1))$. 
By Lemma \ref{lem: Hasse}, it is equivalent that the image of \eqref{eq:b2 in H2} vanishes in $H^2(\Q_{\ell_0},\F_p(1))$ and $H^2(\Q_{\ell_1},\F_p(1))$. Since $b\up1|_{\ell_0}=0$ in $H^1(\Q_{\ell_0},\F_p(1))$, it is enough to consider the restriction of \eqref{eq:b2 in H2} to $H^2(\Q_{\ell_1},\F_p(1))$.

Since $d\up1 = b\up1 c\up1 - a\up1$ and $c\up1|_{\ell_1}=0$, it follows that $d\up1|_{\ell_1}=-a\up1|_{\ell_1}$. Restricting \eqref{eq:b2 in H2} to $G_{\ell_1}$ then gives 
\[
a\up1|_{\ell_1} \cup b\up1|_{\ell_1} - b\up1|_{\ell_1} \cup a\up1|_{\ell_1},
\]
which vanishes if and only if $a\up1|_{\ell_1} \cup b\up1|_{\ell_1}=0$ by the skew-symmetry of cup product.
\end{proof}

\subsection{The invariant $\beta' \in \mathbb{F}_p(2)$} 
\label{subsec: beta-prime} 
The assumption $\dim_{\F_p} S=4$ implies the equation \eqref{eq: a2}. We use \eqref{eq: a2} to define an element $\beta' \in \F_p(2)$.

\begin{lem}
\label{lem: beta'}
Assume $\dim_{\F_p} S=4$.
There is a unique element $\beta' \in \F_p(2)$ such that
\begin{equation}
    \label{eq: beta'}
    (b^{(2)} + \eta b_0\up1)\vert_{\ell_0} = \beta' \cup c^{(1)}\vert_{\ell_0}.
\end{equation}
\end{lem}
\begin{proof}
By Lemma \ref{lem: tate duality at ell0}, it is enough to show that the cup product
$$(b^{(2)} + \eta b_0\up1)\vert_{\ell_0} \cup c\up1|_{\ell_0}$$ vanishes in $H^2(\Q_{\ell_0},\F_p)$.
This follows from \eqref{eq: a2} by restriction to $H^2(\Q_{\ell_0},\F_p)$. Indeed, recall from Lemma \ref{lem: log ell1 is zero} that $b\up1|_{\ell_0}=0$. Since $-da\up1=b\up1 \smile c\up1$, this implies that $a\up1|_{\ell_0}$ is a cocycle. By the skew-symmetry of cup product on cohomology, \eqref{eq: a2} then implies
\[
(b^{(2)} + \eta b_0\up1)\vert_{\ell_0} \cup c\up1|_{\ell_0}=0.
\]
in $H^2(\Q_{\ell_0},\F_p)$.
\end{proof}

\subsection{Implications of the $\mathrm{US}_N$ property of $\rho_S'$} 
The fact that $\rho_S'$ is unramified-or-Steinberg at $\ell_0$ implies a relationship between the invariants $\alpha \in \F_p(1)$, defined in Definition \ref{P1: defn: alpha}, and $\beta' \in \F_p(2)$, defined in Lemma \ref{lem: beta'}.

\begin{prop}
    \label{prop: key alpha beta equation}
    Assume $\dim_{\F_p} S = 4$. Then 
    \begin{enumerate}
        \item $\alpha^2 + \beta' =0$ in $\F_p(2)$. 
        \item the invariant $\mu \in \F_p$ set up in Lemma \ref{lem: S coordinates} is zero.
    \end{enumerate}
    In particular, the presentation of $S$ from Lemma \ref{lem: S coordinates} takes the form 
    \[
    \frac{\F_p\lb X,Y\rb}{(X^2,XY, Y^3)}.
    \]
\end{prop}

\begin{rem}
Since $\alpha$ depends only on the pinning data of Definition \ref{P1: defn: pinning}, part (1) implies that $\beta'$ depends only on this data as well.
\end{rem}

\begin{proof}
Since $\rho_S'$ is obtained as a quotient of the universal $\US_N$ Cayley--Hamilton representation $E$, it is also $\US_N$. Let $\sigma \in G_{\ell_0}$ and $\tau \in I_{\ell_0}$. By Definition \ref{defn: US local}, the fact that $\rho_S'$ is $\US_N$ implies that
\begin{equation}
\label{eq: rho_S' is US}
    (\rho_S'(\sigma)-\omega(\sigma))(\rho_S'(\tau)-1)
\end{equation}
vanishes in $E_S'$. Consider the top-left coordinate of \eqref{eq: rho_S' is US} in terms of the GMA decomposition \eqref{eq: E_S coordinates} of $\rho_S'$. Using the facts that $\omega|_{\ell_0}=1$ and $b\up1|_{\ell_0}=0$, and the formula for multiplication in $E_S'$ given in \eqref{eq:BC in E'_S}, the top-left coordinate in \eqref{eq: rho_S' is US} equals
\begin{equation}
\label{eq: a coord of rho_S' is US}
    \left(a\up1(\sigma)a\up1(\tau) + (b\up2(\sigma)+\eta b_0\up1(\sigma))c\up1(\tau)\right)y^2 +a_0(\sigma)a_0(\tau)x^2.
\end{equation}
Recall from the presentation given in Lemma \ref{lem: S coordinates} that $x^2=\mu y^2$ in $S$. Using the relations
\[
a\up1|_{\ell_0}=\alpha \smile c\up1|_{\ell_0}, \quad (b\up2 + \eta b_0\up1)|_{\ell_0} = \beta' \smile c\up1|_{\ell_0}
\]
that define $\alpha$ and $\beta'$, \eqref{eq: a coord of rho_S' is US} then simplifies to
\begin{equation}
\label{eq: a coord of rho_S' is US simplified}
    \left((\alpha^2+\beta')c\up1(\tau)c\up1(\sigma) +\mu a_0(\tau)a_0(\sigma)\right) y^2.
\end{equation}
Since \eqref{eq: rho_S' is US} vanishes in $E_S'$, this implies that \eqref{eq: a coord of rho_S' is US simplified} vanishes in $S$.

The vanishing of \eqref{eq: a coord of rho_S' is US simplified} in $S$ for arbitrary $\sigma \in G_{\ell_0}$ and $\tau \in I_{\ell_0}$ implies 
\[
(\alpha^2+\beta')c\up1(\tau)c\up1|_{\ell_0} +\mu a_0(\tau)a_0|_{\ell_0} =0,
\]
for all $\tau \in I_{\ell_0}$.
Since $a_0|_{\ell_0}$ and $c\up1|_{\ell_0}$ are linearly independent in $H^1(\Q_{\ell_0},\F_p)$ by Proposition \ref{prop: rk2 assumption}, this implies \[
(\alpha^2+\beta')c\up1(\tau)=0 \text{ and } \mu a_0(\tau) =0
\]
for all $\tau \in I_{\ell_0}$. Since $c\up1|_{I_{\ell_0}}$ and $a_0|_{I_{\ell_0}}$ are nonzero, this gives the result.
\end{proof}

Since $\mu=0$ in the presentation for $S$ of Lemma \ref{lem: S coordinates}, there is a ring homomorphism
\[
D_2: S \to \F_p[\ep_2], \quad x \mapsto 0, y \mapsto \ep
\]
whose composition with the quotient $\F_p[\ep_2] \to \F_p[\ep_1]$ is $D_1$. There is also a homomorphism of $\F_p[\ep_2]$-GMAs $E_S'\otimes_S \F_p[\ep_2] \to E_2$, where $E_2$ is the $1$-reducible GMA over $\F_p[\ep_2]$ of Definition \ref{defn: 1-reducible GMA}.

\begin{cor} 
\label{cor: rho_2 exists if dimS=4}
Assume $\dim_{\F_p}S=4.$
The map 
\[
\Upsilon_2: \F_p \oplus \F_p[\ep_1] \to \F_p[\ep_1]
\]
given by $\Upsilon_2(u,v)=\eta \epsilon u + v$ induces a map of $\F_p[\ep_2]$-GMAs
\[
E_S' \otimes_S \F_p[\ep_2] \xrightarrow{\sm{D_2}{\Upsilon_2}{\mathrm{Id}}{D_2}} E_2.
\]
In particular, there is an $\US_N$ Cayley--Hamilton representation $\rho_2: G_{\Q,Np} \to E_2^\times$ that deforms $\rho_1$ along the map $r_{2,1} : E_2 \rsurj E_1$ of \eqref{eq: r21}. 
\end{cor}
\begin{proof}
Given that $\mu=0$, the fact that $\sm{D_2}{\Upsilon_2}{\mathrm{Id}}{D_2}$ is ring homomorphism is a simple computation using the formula \eqref{eq:BC in E'_S} for multiplication in $E_S'$. The representation $\rho_2$ is obtained as the composition of $\rho_S'$ with $E_S' \otimes_S \F_p[\ep_2] \to E_2$.
\end{proof}

\section{Constructing a second-order $\mathrm{US}_N$ deformation $\rho_2$}
\label{sec: proof of main theorem}

In this section, we prove the remaining implication of Theorem \ref{P1: thm: main intro}. Throughout the section, we assume $a\up1\vert_{\ell_1} = 0$. Under this assumption, we construct an invariant $\beta \in \F_p(2)$, and show that if $\alpha^2+\beta=0$, then $\dim_{\F_p}R/pR>3$. In particular, if $\alpha^2+\beta=0$, we can apply the constructions of the previous section to obtain another invariant $\beta' \in \F_p(2)$, and we prove that $\beta'=\beta$.

The proof of $\dim_{\F_p}R/pR>3$ involves constructing an explicit $\US_N$ deformation using the 1-reducible GMAs of Definition \ref{defn: 1-reducible GMA}. We do this in steps, first constructing an arbitrary deformation, and then imposing the local conditions one at a time. We show that the assumption $a\up1|_{\ell_1}=0$ implies that a deformation exists. Next, we impose the finite-flat condition, which we show limits the set of deformations enough that there is a well-defined invariant $\beta \in \F_p(2)$. Finally, we show that the unramified-or-Steinberg condition is satisfied if $\alpha^2+\beta=0$.

\subsection{Construction of a second-order 1-reducible GMA representation without local conditions} 
Recall the Cayley--Hamilton representation
\[
 \rho_1 = \ttmat{\omega(1 + \ep a\up1)}{b\up1}{\omega c\up1}{1+\ep d\up1}
: G_{\Q,Np} \to E_1^\times
\]
of Lemma \ref{lem: D1 construction}. Let $\Pi_2$ denote the set of second-order 1-reducible Cayley--Hamilton deformations of $\rho_1$:
\[
\Pi_2 = \{\rho_2:~G_{\Q,Np}~\to~E_2^\times \mid r_{2,1} \circ \rho_2  = \rho_1\},
\]
where $r_{2,1}$ is the reduction map of 1-reducible GMAs $r_{2,1} : E_2 \rsurj E_1$ of \eqref{eq: r21}.

\begin{lem}
    \label{P1: lem: exists 2nd-order 1-reducible} 
    The set $\Pi_2$ is in bijection with the set quadruples of cochains $a\up2, d\up2: G_{\Q,Np}~\to~\F_p$, $b\up2: G_{\Q,Np} \to \F_p(1)$, and $c\up2:~G_{\Q,Np}~\to~\F_p(-1)$ that satisfy
\begin{enumerate}[label=(\roman*)]
\item $-da\up2 =  a\up1 \smile a\up1 + b\up1 \smile c\up2 + b\up2 \smile c\up1$
\item $-db\up2 = a\up1 \smile b\up1 + b\up1 \smile d\up1$
\item $-dc\up2 = c\up1 \smile a\up1 + d\up1 \smile c\up1$
\item $-dd\up2 = d\up1 \smile d\up1 + c\up1 \smile b\up2 + c\up2 \smile b\up1$. 
\end{enumerate}
    This set is non-empty if and only if $a\up1|_{\ell_1}=0$. Moreover, if it is non-empty, 
    \begin{enumerate}
        \item $\Pi_2$ admits the structure of a torsor under the group
    \[
    \mathfrak{Z}_2 := Z^1(\Z[1/Np], \F_p) \times Z^1(\Z[1/Np], \F_p) \times Z^1_b \times Z^1(\Z[1/Np], \F_p(-1)),
    \]
    where
    \[
    Z^1_b := \ker \left(Z^1(\Z[1/Np], \F_p(1)) \to \frac{H^1(\Q_{\ell_0}, \F_p(1))}{\langle \F_p(2)\vert_{\ell_0} \cup [c\up1] \vert_{\ell_0}\rangle} \right)
    \]
    and the action of $(a,d,b,c) \in \mathfrak{Z}_2$ on $(a\up2, d\up2, b\up2,c\up2) \in \Pi_2$ has the form
    \begin{align*}
        &(a,d,0,0) \cdot (a\up2,d\up2,b\up2,c\up2) = 
    (a\up2 + a, d\up2+d, b\up2,c\up2) \\
        &(0,0,b,c) \cdot (a\up2,d\up2,b\up2,c\up2) = \\
        & \qquad \qquad 
    (a\up2 + \sigma(b,c), d\up2-\sigma(b,c)+b\cdot c\up1 + b\up1 \cdot c, b\up2+b,c\up2+c) 
    \end{align*}    
    where $\sigma : Z^1_b \times Z^1(\Z[1/Np], \F_p(-1)) \to C^1(\Z[1/Np], \F_p)$ is a choice of linear map such that $-d\sigma(b,c) = b \smile c\up1 + b\up1 \smile c$. 
    \item For every $(a\up2, d\up2, b\up2,c\up2) \in \Pi_2$, the restriction $b\up2\vert_{\ell_0}$ is a cocycle whose cohomology class is a multiple of $c_0\vert_{\ell_0} = [c\up1]\vert_{\ell_0}$. 
    \end{enumerate}
\end{lem}

\begin{proof}
Every element $\rho_2$ of $\Pi_2$ can be written in the form
\begin{equation}
    \label{eq: purported rho_2}
    \rho_2 = \ttmat{\omega (1+a\up1\epsilon + a\up2\epsilon^2) }{b\up1 + b\up2 \epsilon}{\omega(c\up1 + c\up2 \epsilon)}{1+d\up1 \epsilon + d\up2\epsilon^2} : G_{\Q,Np} \to E_2^\times,
\end{equation}
for some cochains $a\up2$, $b\up2$, $c\up2$, $d\up2$. The fact that $\rho_2$ is a homomorphism implies the equations (i)-(iv). Conversely, given cochains satisfying (i)-(iv), the function $\rho_2$ defined by \eqref{eq: purported rho_2} is an element of $\Pi_2$. This gives the desired bijection. Now we show that there are cochains satisfying (i)-(iv) if and only if $a\up1|_{\ell_1}=0$.

    \noindent
    \textit{Coboundary condition} (ii). Note that (ii) is the same equation as \eqref{eq: b2}. By Proposition \ref{prop: a1 vanish at ell1 when S large}, there is a cochain $b\up2$ satisfying (ii) if and only if $a\up1|_{\ell_1}=0.$ The set of cochains $b\up2$ satisfying (ii) is a torsor for $Z^1(\Z[1/Np],\F_p(1))$; however, we will see that condition (i) can only be satisfied for a subset of the cochains $b\up2$ satisfying (ii).
    
    This shows that $a\up1|_{\ell_1}=0$ is necessary for $\Pi_2$ to be non-empty. Now assume $a\up1|_{\ell_1}=0$, and we will show this is sufficient.
    
    \noindent 
    \textit{Coboundary condition} (iii). 
    There is a cochain satisfying (iii) if 
    \[
    c\up1 \cup a\up1 + d\up1 \cup c\up1 =0
    \]
    in $H^2(G_{\Q,Np},\F_p(-1))$.
    By Lemma \ref{lem: Hasse}, we only have to check this vanishing after restriction to $H^2(\Q_{\ell_1},\F_p(-1))$ and $H^2(\Q_{\ell_0},\F_p(-1))$. The $\ell_1$-local restriction vanishes because $a\up1|_{\ell_1} = d\up1|_{\ell_1}=0$. Since $d\up1|_{\ell_0}=-a\up1|_{\ell_0}$, the $\ell_0$-local restriction is 
    \[
    c\up1|_{\ell_0} \cup a\up1|_{\ell_0} - a\up1|_{\ell_0} \cup c\up1|_{\ell_0},
    \]
    which vanishes because $a\up1|_{\ell_0} = \alpha \cup c\up1|_{\ell_0}$. The set of cochains satisfying (iii) is a torsor for $Z^1(\Z[1/Np],\F_p(-1))$.
    
    \noindent
    \textit{Coboundary condition} (i). Note that condition (i) is similar to \eqref{eq: a2}; this argument follows the same line as in the proof Lemma \ref{lem: beta'}. 
    
    Let $b\up2$ and $c\up2$ be arbitrary cochains satisfying (ii) and (iii), respectively. 
    There is a cochain $a\up2$ satisfying (i) if 
    \begin{equation}
    \label{eq: (i) in H2}
        a\up1 \smile a\up1 + b\up1 \smile c\up2 + b\up2 \smile c\up1
    \end{equation} 
    vanishes in $H^2(G_{\Q,Np},\F_p)$. By Lemma \ref{lem: Hasse} and since $H^2(\Q_{\ell_1},\F_p)=0$, it is enough to check this vanishing after restriction to $H^2(\Q_{\ell_0},\F_p)$. 
    
    Recall from Lemma \ref{lem: log ell1 is zero} that $b\up1|_{\ell_0}=0$. Since $-d a\up1 = b\up1 \smile c\up1$, this implies that $a\up1|_{\ell_0}$ is a cocycle. Likewise, differential equation (ii) implies that $b\up2\vert_{\ell_0}$ is a cocycle. By the skew-symmetry of cup product, \eqref{eq: (i) in H2} vanishes if and only if
    \[
    b\up2|_{\ell_0} \smile c\up1|_{\ell_0}
    \]
    vanishes in $H^2(\Q_{\ell_0},\F_p)$. This happens for some choices of $b\up2$, but not others: recall that the set of choices of $b\up2$ satisfying (ii) is a torsor for $Z^1(\Z[1/Np],\F_p(1))$. Indeed, since $H^2(\Q_{\ell_0},\F_p)$ has $\F_p$-dimension 1 and is spanned by $[b_0\up1]|_{\ell_0} \cup [c\up1]|_{\ell_0}$ by Proposition \ref{prop: rk2 assumption}, there is a constant $\gamma \in \F_p$ such that
    \[
    [b\up2|_{\ell_0}] \cup [c\up1]|_{\ell_0} = \gamma [b_0\up1]|_{\ell_0} \cup [c\up1|_{\ell_0}.
    \]
    This shows that \eqref{eq: (i) in H2} vanishes if  $b\up2$ is replaced by $b\up2-\gamma b_0\up1$. Moreover, the set of choices for $b\up2$ satisfying (ii) and such that \eqref{eq: (i) in H2} vanishes is a torsor for the set of $b \in Z^1(\Z[1/Np],\F_p(1))$ such that $[b]|_{\ell_0} \cup [c\up1]|_{\ell_0}=0$. By Tate duality (Lemma \ref{lem: tate duality at ell0}), this is same as $b$ belonging to the subgroup $Z^1_b$. 
    
    In summary, the set of $b\up2$ that satisfy (ii) and such that (i) has a solution is a torsor for $Z^1_b$; this is holds for any choice of $c\up2$. For any such $b\up2$, the set of $a\up2$ that satisfy (i) is a torsor for $Z^1(\Z[1/Np],\F_p).$
    
    \noindent
    \textit{Coboundary condition} (iv). The same analysis as for (i) applies to (iv). 
    
    Combining these analyses, we deduce that
    \begin{itemize}
        \item $\Pi_2$ is non-empty if and only if $a\up1\vert_{\ell_1} = 0$
        \item There is an action of $Z^1_b \times Z^1(\Z[1/Np], \F_p(-1))$ on $\Pi_2$ that acts by addition on the $b\up2$ and $c\up2$-coordinates 
        \item And there exists a linear choice of $\sigma$, namely 
        \[
        Z^1_b \times Z^1(\Z[1/Np], \F_p(-1)) \ni (b,c) \mapsto -d^{-1}(b \smile c\up1 + b\up1 \smile c)
        \]
        where $d^{-1}$ is an arbitrarily chosen linear section of the boundary map $d : C^1(\Z[1/Np], \F_p) \rsurj B^2(\Z[1/Np], \F_p)$. Under this definition of $\sigma$, one can compute that the differential equations (i) and (iv) are satisfied by $(0,0,b,c) \cdot (a\up2,b\up2,c\up2,d\up2)$. (See the origin of the formula for the $d\up2$-coordinate in the proof of Lemma \ref{lem: rho2 constant determinant}, below.) 
        \item There is an action of $Z^1(\Z[1/Np], \F_p)^{\oplus 2}$ that acts by addition on the $a\up2$ and $d\up2$-coordinates and fixes the $b\up2$ and $c\up2$-coordinates,
    \end{itemize}
    which amounts to claim (1). Claim (2) follows from the analysis of coboundary condition (i) above. 
\end{proof}

We will frequently use the bijection between $\Pi_2$ and the set of quadruples of cochains $(a\up2,d\up2,b\up2,c\up2)$  satisfying (i)-(iv) 
without comment. 
Let $\Pi_2^{\det}$ denote the subset of $\Pi_2$ consisting of elements with constant determinant $\omega$, 
\[
\Pi_2^{\det} := \{\rho_2 \in \Pi_2 \ | \ \det(\rho_2)=\omega\}.
\]

\begin{lem}
    \label{lem: rho2 constant determinant}
    Assume $a\up1|_{\ell_0}=0$. Then $\Pi_2^{\det}$ is non-empty, an element $\rho_2 \in \Pi_2$ is completely determined by its cochains $a\up2$, $b\up2$ and $c\up2$, and $\Pi_2^{\det}$ is a torsor for the subgroup $\mathfrak{Z}_2^{\det} \subset \mathfrak{Z}_2$ under the action of $\mathfrak{Z}_2$ on $\Pi_2$ of Lemma \ref{P1: lem: exists 2nd-order 1-reducible}, where
    \[
    \mathfrak{Z}_2^{\det} := \{(a,d,b,c) \in \mathfrak{Z}_2 \mid a + d = 0\} \subset \mathfrak{Z}_2. 
    \]
\end{lem}

\begin{proof}
Let $a\up2$, $b\up2$, and $c\up2$ be cochains satisfying equations (i), (ii) and (iii), respectively, of Lemma \ref{P1: lem: exists 2nd-order 1-reducible}. A straightforward calculation shows that the only choice of cochain $d\up2$ such that the resulting representation $\rho_2$ satisfies $\det(\rho_2)=\omega$ is  
    \[
    d\up2 = b\up1 c\up2 + b\up2 c\up1 - a\up1 d\up1 - a\up2. 
    \]
Moreover, a computation shows that this choice of $d\up2$ satisfies equation (iv). 
The action of an element $(a,b,c,d) \in \mathfrak{Z}_2$ fixes the determinant if and only if $a+d=0$. This follows from the equations for the action of $\mathfrak{Z}_2$ given in Lemma \ref{P1: lem: exists 2nd-order 1-reducible}.
\end{proof}

\subsection{The finite-flat at $p$ condition on $\rho_2$}
\label{P1: subsec: finite-flat on rho2}
We continue to assume that $a\up1|_{\ell_1}=0$; consequently, $\Pi_2^{\det}$ is non-empty by Lemma \ref{lem: rho2 constant determinant}.
Consider the subset of $\Pi_2$, 
\[
\Pi_2^{\det, p} = \{ \rho_2 \in \Pi_2^{\det} \mid \rho_2|_p\text{ is finite-flat}\} \subset \Pi_2.
\]

\begin{prop}
    \label{P1: prop: exists unique beta}
    Assume $a\up1|_{\ell_1}=0$. Then $\Pi_2^{\det,p}$ is non-empty, and the possibilities for $b\up2$-coordinates of $\rho_2 \in \Pi_2^{\det,p}$ is contained in a torsor under the subgroup of $Z^1(\Z[1/Np], \F_p(1))$ spanned by coboundaries and $b\up1$. In particular, there is a unique $\beta \in \F_p(2)$ such that, for every $\rho_2 \in \Pi_2^{\det,p}$, 
    \[
    [b\up2|_{\ell_0}]=\beta \cup [c\up1]|_{\ell_0} \quad \in H^1(\Q_{\ell_0}, \F_p(1)),
    \]
    where $b\up2$ is the cochain associated to $\rho_2$. 
\end{prop}

\begin{rem}
In fact, the phrase ``contained in'' in the proposition can be replaced by ``equal to,'' but we do not have a use for that result. 
\end{rem}

\begin{rem}
Lemma \ref{P1: lem: exists 2nd-order 1-reducible}(2) already implies that some such $\beta$ exists for any \emph{single} $\rho_2 \in \Pi_2$; our supplemental work will be to show that there is only one $\beta$ that appears among $\rho_2 \in \Pi_2^{\det,p}$. 
\end{rem}

We will prove the first claim of Proposition \ref{P1: prop: exists unique beta}, that $\Pi_2^{\det,p}$ is non-empty, using a series of lemmas to produce an element of $\Pi_2^{\det}$ that is finite-flat at $p$. 

Just as in the proof that $\rho_1$ is finite-flat in Lemma \ref{lem: D1 construction}, it will be convenient to change the basis of $\rho_2 \in \Pi_2^{\det}$ in order to test the finite-flat condition of $\rho_2\vert_p$. Recall the element $x_{c\up1} \in \F_p(1)$ of Definition \ref{P1: defn: pinned cocycles} satisfying $dx_{c\up1}\vert_p = c\up1|_p$. Define $\rho_2' := \ad(\sm{1}{0}{-x_{c\up1}+y\epsilon}{1})\rho_2$ for $y \in \F_p$ to be chosen later, and write $\rho_2'$ as
\begin{equation}
    \label{eq: rho2 prime}
\rho_2'=\ttmat{\omega (1+a\upp1\epsilon + a\upp2\epsilon^2) }{b\upp1 + b\upp2 \epsilon}{\omega(c\upp1 + c\upp2 \epsilon)}{1+d\upp1 \epsilon + d\upp2\epsilon^2}.
\end{equation}
Explicitly:
\begin{itemize}
\item $a\upp1 =  a\up1 + b\up1 \smile x_{c\up1}$
\item $b\upp1= b\up1$
\item $c\upp1= c\up1 - dx_{c\up1}$
\item $d\upp1=  d\up1 - x_{c\up1} \smile b\up1$
\end{itemize}
and
\begin{itemize}
\item $a\upp2 =  a\up2 + b\up2 \smile x_{c\up1} + b\up1 \smile y$
\item $b\upp2= b\up2$
\item $c\upp2= c\up2 - x_{c\up1}\smile a\up1 +d\up1 \smile x_{c\up1} - x_{c\up1} \smile b\up1 \smile x_{c\up1} - dy$
\item $d\upp2=  d\up2 - x_{c\up1} \smile b\up2 - y \smile b\up1$.
\end{itemize}
Just as in the proof of Lemma \ref{lem: D1 construction}, we have
\begin{itemize}
\item $a\upp1|_{p}$ and $d\upp1|_p$ are unramified homomorphisms
\item $c\upp1|_{p} = 0$.
\end{itemize}
Because $\rho_2'$ is also a homomorphism, the primed cochains also satisfy equations (i)-(iv) of Lemma \ref{P1: lem: exists 2nd-order 1-reducible}. 

\begin{lem}
\label{lem: flat rho2 step 1}
Assume $a\up1|_{\ell_1}=0$. There exists $\rho_2 \in \Pi_2^{\det}$ such that $\rho_2|_p$ is upper-triangular (in the sense that $c\upp2\vert_p~=~0$). 
\end{lem}

\begin{proof}
Let $\rho_2 \in \Pi_2^{\det}$ be arbitrary. We will find an element $(a,-a,b,c) \in \mathfrak{Z}_2^{\det}$ such that $\rho_{2,\mathrm{new}}:=(a,-a,b,c)\cdot \rho_2$ 
has the desired property.

By equation (iii) of Lemma \ref{P1: lem: exists 2nd-order 1-reducible} applied to $c\upp2$, 
\[
-dc\upp2|_p = c\upp1|_p \smile a\upp1|_p + d\upp1|_p \smile c\upp1|_p  = 0
\]
since $c\upp1|_p = 0$. Hence $c\upp2|_p$ is a cocycle. 

\begin{sublem}
The $\F_p$-dimension of $H^1(\Q_p,\F_p(-1))$ is 1. The localization map $H^1(\Z[1/Np],\F_p(-1)) \to H^1(\Q_p,\F_p(-1))$ is surjective.
\end{sublem}
\begin{proof}
The first claim is a standard consequence of Tate local duality and local Euler characteristics at $p$; in particular, the Euler characteristic of $H^\bullet(\Q_p, \F_p(-1))$ is $-1$.  For the second claim, consider the exact sequence
\[
0 \to H^1_{(p)}(\Z[1/Np],\F_p(-1)) \to H^1(\Z[1/Np],\F_p(-1)) \to H^1(\Q_p,\F_p(-1)) 
\]
coming from the definition of $H^\bullet_{(p)}$ as a cone. The Euler characteristic of global cohomology $H^\bullet(\Z[1/Np], \F_p(-1))$ is $-1$ by the global Euler characteristic formula. We also know from the proof of Lemma \ref{lem: Hasse} that $\dim_{\F_p} H^2(\Z[1/Np],\F_p(-1)) = 1$. Therefore $\dim_{\F_p} H^1(\Z[1/Np], \F_p(-1)) = 2$. The desired surjectivity follow from the fact that $H^1_{(p)}(\Z[1/Np], \F_p(-1))$ has dimension 1. Indeed, $c_0$ is a basis for it, as discussed in Definition \ref{P1: defn: pinned cocycles}. 
\end{proof}

By the sublemma, there exists $z \in Z^1(\Z[1/Np],\F_p(-1))$ such that $z|_p=-c\upp2|_p$. Let $\rho_{2,\mathrm{new}}=(0,0,0,z)\cdot \rho_2$.
It has $c\up2_\mathrm{new}= c\up2 + z$. 
By the formula for $c\upp2$ in terms of $c\up2$ (given after \eqref{eq: rho2 prime}), we also have $c_\mathrm{new}\upp2 = c\upp2+z$. Therefore $c_\mathrm{new}\upp2\vert_p = 0$, as desired. 
\end{proof}

Let $\rho_2 \in \Pi_2^{\det}$ be as in Lemma \ref{lem: flat rho2 step 1}. Then
\begin{equation}
    \label{eq: rho2 prime p}
\rho_2'|_p=\ttmat{\omega \chi_2 }{(b\upp1 + b\upp2 \epsilon)|_p}{0}{\chi_2^{-1}} : G_p \to E_2^\times,
\end{equation}
where
\[
\chi_2= (1+a\upp1\epsilon+a\upp2\epsilon^2)|_p: G_p \to \F_p[\ep_2]^\times.
\]
Indeed, since $\rho_2'|_p$ is upper-triangular, $\chi_2$ is a homomorphism, and since $\det(\rho_2')=\omega$, the lower-right coordinate of $\rho_2'|_p$ must be $\chi_2^{-1}$. Let $\chi_1 : G_p \to \F_p[\ep_1]^\times$ denote the character $\chi_1 := \chi_2 \otimes_{\F_p[\ep_2]} \F_p[\ep_1]$, which equals $1 + a\upp1\ep$.  
We want to characterize the finite-flat at $p$ property of $\rho_2$, bootstrapping from the fact that its reduction $\rho_1$ is finite-flat at $p$.
To this end, we induce two representations $\eta_2$ and $\eta_1$ associated to an element of $\Pi_2^{\det}$

\begin{defn}
\label{defn: etas}
Assume that $\rho_2 \in \Pi_2^{\det}$ has the property that $\rho_2'\vert_p$ is upper-triangular. Then there are two associated representations
  \[
    \eta_2 =\ttmat{\omega\chi_2}{\ep b\upp1 + \ep^2 b\upp2}{0}{\chi_2^{-1}}: G_p \to \GL_2(\F_p[\ep_2])
    \]
given by $\rho_2'\vert_p$ composed with the embedding of Lemma \ref{lem: UT embedding}, and
\begin{equation}
        \label{eq: eta1 coordinates}
     \eta_1=\ttmat{\omega(1 + \ep a\upp1)}{b\upp1 + \ep b\upp2}{\omega \ep c\upp1}{1 + \ep d\upp1}: G_{\Q,Np} \to \GL_2(\F_p[\ep_1])
\end{equation}
given by $\rho_2'$ composed with the map $E_2 \to M_2(\F_p[\ep_1])$ of Lemma \ref{lem: matrix reduction}.
\end{defn}

\begin{rem}
Note that
\[
\eta_1\vert_p = \ttmat{\omega \chi_1}{b\upp1\vert_p + \ep b\upp2\vert_p}{0}{\chi_1^{-1}}. 
\]
Also, be aware that $\eta_1$ does not equal the reduction of $\eta_2$ modulo $\ep^2$. Rather, one obtains $\eta_1$ from $\eta_2$ by ``dividing the extension class $\ep b\upp1 + \ep^2 b\upp2 \in \Ext^1_{\F_p[\ep_2][G_p]}(\chi_2^{-1}, \omega\chi_2)^\mathrm{flat}$ by $\epsilon$,'' which will be made rigorous in the proof of Lemma \ref{lem: flat rho2 step 2}.
\end{rem}

\begin{lem}
\label{lem: flat rho2 step 2}
Assume that $\rho_2 \in \Pi_2^{\det}$ has the property that $\rho_2'\vert_p$ is upper-triangular, and let $\eta_2$ and $\eta_1$ be the associated representations of Definition \ref{defn: etas}. The following are equivalent:
\begin{enumerate}
    \item The Cayley--Hamilton representation $\rho_2\vert_p$ is finite-flat.
    \item The homomorphism $\eta_2$
    is finite-flat.
    \item The homomorphism $\eta_1|_p: G_p \to \GL_2(\F_p[\ep_1])$ is finite-flat and $\chi_2$ is unramified.
\end{enumerate}
\end{lem}

\begin{proof}
The equivalence of (1) and (2) follows from the embedding of Lemma \ref{lem: UT embedding} along with Lemma \ref{lem: GMA property}.

Now we assume (2) and prove (3). By Proposition \ref{prop: finite-flat implies upper tri}, $\chi_2$ is unramified. We will show that $\eta_1|_p$ is isomorphic to a subquotient representation of $\eta_2$. This implies that $\eta_1|_p$ is finite-flat, since the finite-flat property is \emph{stable}, as discussed in \S\ref{P1: subsubsec: finite-flat}.

From the exact sequences 
\[
0 \to \ep \omega \chi_2 \to \omega \chi_2 \to \omega \to 0, \quad 0 \to \ep^2 \chi_2^{-1} \to \chi_2^{-1} \to \chi_1^{-1} \to 0
\]
there is a commutative diagram of $\Ext$ groups over $\F_p[\ep_2][G_p]$ with exact rows and columns 
\[
\xymatrix{
 & \Ext^1(\chi_1^{-1},\ep \omega \chi_2) \ar[d] & \\
0 \ar[r] & \Ext^1(\chi_2^{-1},\ep \omega\chi_2) \ar[r] \ar[d] & \Ext^1(\chi_2^{-1}, \omega \chi_2) \ar[r] \ar[d] & \Ext^1(\chi_2^{-1},\omega) \\
0 \ar[r] & \Ext^1(\ep^2 \chi_2^{-1},\ep \omega \chi_2) \ar[r] & \Ext^1(\ep^2 \chi_2^{-1},\omega\chi_2)
}
\]
The representation $\eta_2$ defines a class in $\Ext^1(\chi_2^{-1}, \omega\chi_2)$, written as $\ep b\upp1 + \ep^2b\upp2$. The fact that this class is a multiple of $\ep$ implies that $\eta_2$ maps to zero under both the horizontal and the vertical map out of $\Ext^1(\chi_2^{-1}, \omega\chi_2)$ in the diagram. By a diagram chase, there is a class $W \in \Ext^1(\chi_1^{-1},\ep \omega \chi_2)$ mapping to $\eta_2$. This $W$ is a subquotient of $\eta_2$, so it is finite-flat. Moreover, a computation of the maps in the diagram in coordinates, as in \cite[Appendix C]{WWE3}, shows that there is an $\F_p[\ep_1]$-basis for $W$ such that the action of $G_p$ on $W$ is given by $\eta_1|_p$. In particular, $\eta_1|_p$ is isomorphic to $W$ as a $\F_p[G_p]$-module, so it is finite-flat.

Finally, we assume (3) and prove (2). By Proposition \ref{prop: smoothness} (the formal smoothness of the finite-flat deformation functor), there is a finite-flat representation $\eta_{2,\mathrm{lift}} : G_p \to \GL_2(\F_p[\ep_2])$ of the form
\begin{align*}
\eta_{2,\mathrm{lift}} &= \ttmat{\omega(1 + \ep a\upp1 +\ep^2 a_\mathrm{lift}\upp2)}
{b\upp1 + \ep b\upp2 + \ep^2 b_\mathrm{lift}\upp3}{0}
{1 + \ep d\upp1 + \ep^2 d_\mathrm{lift}\upp2} \\ 
&= \ttmat{\omega\chi_{2,1,\mathrm{lift}}}{b\upp1 + \ep b\upp2 + \ep^2 b\upp3_\mathrm{lift}}{0}{\chi_{2,2,\mathrm{lift}}}.
\end{align*}
deforming $\eta_1|_p$. Let $\epsilon\cdot\eta_{2,\mathrm{lift}}$ denote the homomorphism
\[
\epsilon\cdot\eta_{2,\mathrm{lift}}=\ttmat{\omega\chi_{2,1,\mathrm{lift}}}{\ep b\upp1 + \ep^2 b\upp2}{0}{\chi_{2,2,\mathrm{lift}}} : G_p \to \GL_2(\F_p[\ep_2]), 
\]
which represents the class in $\Ext^1_{\F_p[\ep_2][G_p]}(\chi_{2,2,\mathrm{lift}}^{-1},\omega\chi_{2,1,\mathrm{lift}})$ that is the $\ep$-multiple of the class of $b\upp1 + \ep b\upp2 + \ep^2 b\upp3$. By \cite[Rem.\ C.3.2]{WWE3}, since $\eta_{2,\mathrm{lift}}$ is finite-flat, $\ep \cdot \eta_{2,\mathrm{lift}}$ is too.

Finally, since $\chi_2$, $\chi_{2,1,\mathrm{lift}}$, and $\chi_{2,2,\mathrm{lift}}$ are all unramified characters, it follows that
\[
a=a\upp2-a\upp2_\mathrm{lift}, \text{ and } d=d\upp2-d\upp2_\mathrm{lift}
\]
are unramified cocycles. Then $\eta_2$ is obtained from the finite-flat representation $\ep \cdot \eta_{2,\mathrm{lift}}$ by adding the cocycle $\sm{a}{0}{0}{d} \in Z^1(G_p, \End_{\F_p}(\omega \oplus 1))$, which is in the finite-flat subspace. By Proposition \ref{prop: smoothness}, this implies that $\eta_2$ is also finite-flat.
\end{proof}

\begin{lem}
\label{lem: flat rho2 step 3}
Assume $a\up1|_{\ell_1}=0$. There exists $\rho_2 \in \Pi_2^{\det}$ such that 
\begin{itemize}
    \item $\rho_2'\vert_p$ is upper-triangular (equivalently, $c\upp2\vert_p = 0$), and 
    \item the associated homomorphism $\eta_1$ as in Definition \ref{defn: etas} is finite-flat.
\end{itemize}
\end{lem}

\begin{proof}
Let $\rho_2 \in \Pi_2^{\det}$ be such that $\rho_2\vert_p$ is upper-triangular (which exists by Lemma \ref{lem: flat rho2 step 1}). We will find an element $(0,0,b,0) \in \mathfrak{Z}_2^{\det}$ such that $\rho_{2,\mathrm{new}}:=(0,0,b,0)\cdot \rho_2$ has $\eta_{1,\mathrm{new}}$ being finite-flat. This $\rho_{2,\mathrm{new}}|_p$ is still upper-triangular because $c\upp2_{\mathrm{new}}=c\upp2$.

Let $\eta=\sm{\omega}{b\upp1}{0}{1}=(\eta_1 \bmod{\ep})$; it is finite-flat at $p$ by Lemma \ref{P1: lem: local Kummer finite-flat}. The lift $\eta_1$ of $\eta$ over $\F_p[\ep_1] \rsurj \F_p$ can and will be considered to be an element of $Z^1(\Z[1/Np], \Ad^0(\eta))$ by Lemma \ref{lem: basic def theory}. We want to examine its coordinates so we set up the following notions.

The filtration of $\F_p[G_{\Q,Np}]$-modules 
\[
0 \to \F_p(1) \buildrel\iota\over\to \eta \buildrel\pi\over\to \F_p \to 0 
\]
induces a filtration of $\Ad^0(\eta)$, 
\[
0 \subset \Hom_{\F_p}(\F_p, \F_p(1)) \subset U \subset \Ad^0(\eta),
\]
where $U := \{f \in \Ad^0(\eta) \mid \pi \circ f \circ \iota = 0\}$ and
\[
\frac{U}{\Hom_{\F_p}(\F_p, \F_p(1))} \simeq \F_p, \qquad \frac{\Ad^0(\eta)}{U} \cong \Hom_{\F_p}(\F_p(1), \F_p) \cong \F_p(-1). 
\]
For any subgroup $G \subset G_{\Q,Np}$, it is exactly the cochains in $C^1(G, \Ad^0(\eta))$ that lie in $C^1(G, U)$ that are upper-triangular. Therefore we are interested in $Z^1(\Q_p, U)$, and its finite-flat subspace $Z^1(\Q_p, U)^\fl$. We also want to use the subspace of global lifts that are upper-triangular upon restriction to $G_p$, 
\[
Z^1(\Z[1/Np], \Ad^0(\eta))^{p\text{-UT}} = \ker\left(Z^1(\Z[1/Np], \Ad^0(\eta)) \to 
\frac{Z^1(\Q_p, \Ad^0(\eta\vert_p))}{Z^1(\Q_p, U)}
\right).
\]

\begin{sublem}
\label{sublem: eta adjoint} 
There is a commutative diagram induced by the filtrations above with exact rows 
\[
\resizebox{\columnwidth}{!}{
\xymatrix{
0 \ar[r] &  Z^1(\Z[1/Np], \F_p(1))  \ar[r]^(.45){\iota_\Q} \ar[d] & Z^1(\Z[1/Np], \Ad^0(\eta))^{p\text{-UT}}  \ar[d] &  & \\
0 \ar[r] & Z^1(\Q_p, \F_p(1))  \ar[r]^{\iota_p} & Z^1(\Q_p, U) \ar[r]^{\kappa_p}  & Z^1(\Q_p, \F_p)  & \\
 0 \ar[r] & Z^1(\Q_p, \F_p(1))^\fl \ar[r] \ar[u] & Z^1(\Q_p, U)^\fl \ar[r] \ar[u] & Z^1(\Q_p, \F_p)^\mathrm{flat} \ar[r] \ar[u] & 0
}
}
\]
where 
\[
\iota_* : b \mapsto \ttmat{0}{b}{0}{0}, \quad \kappa_* : \ttmat{a_p\upp1}{b_p\upp2}{0}{-a_p\upp1} \mapsto a_p\upp1.
\]
\end{sublem}

\begin{proof}
The commutativity follows directly from the filtrations. The exactness of the top two rows follows from standard long exact sequences in Galois cohomology, for $G = G_{\Q,Np}, G_p$, 
\[
0 \to H^0(G, \F_p) \to H^1(G, \F_p(1)) \to H^1(G, U) \to H^1(G, \F_p),
\]
and the observation that the kernel of $H^1(G, \F_p(1)) \to H^1(G,U)$ arises from cocycles being sent to coboundaries that are non-zero. The exactness of the third row follows from direct calculation of $\iota_p$ and $\kappa_p$. The final claim of the lemma follows from Proposition \ref{prop: smoothness} and the exactness of the rows of the diagram. 
\end{proof}

Since $\eta_1|_p$ is upper-triangular, 
$\eta_1 \in Z^1(\Z[1/Np], \Ad^0(\eta))^{p\text{-UT}}$. 
Moreover $\kappa_p(\eta_1|_p)=a\upp1|_p$
is in $Z^1(\Q_p,\F_p)^\fl=Z^1_{\mathrm{unr}}(\Q_p,\F_p)$ because $a\upp1|_p$ is unramified by construction (see Lemma \ref{P1: lem: produce a1}).

Consider the diagram in the sublemma. Since $\kappa_p(\eta_1|_p) \in Z^1(\Q_p,\F_p)^\fl$, the snake lemma implies that the class of $\eta_1|_p$ is in the image of
\[
\frac{Z^1(\Q_p,\F_p(1))}{Z^1(\Q_p,\F_p(1))^\fl} \xrightarrow{\iota_p} \frac{Z^1(\Q_p,U)}{Z^1(\Q_p,U)^\fl}.
\]
By Lemma \ref{lem: global Kummer theory}, the image of $Z^1(\Z[1/Np],\F_p(1))$ generates $\frac{Z^1(\Q_p,\F_p(1))}{Z^1(\Q_p,\F_p(1))^\fl}$. This implies that there is $b \in Z^1(\Z[1/Np],\F_p(1))$ such that $\iota_p(b|_p)+\eta_1|_p$ is in $Z^1(\Q_p,U)^\fl$. The commutativity of the diagram implies that $\iota_p(b|_p)=\iota_\Q(b)|_p$.

Let $b \in Z^1(\Z[1/Np],\F_p(1))$ be such that $\iota_\Q(b)|_p+\eta_1|_p$ is in $Z^1(\Q_p,U)^\fl$ and let $\rho_{2, \mathrm{new}} := (0,0,b,0) \cdot \rho_2 \in \Pi_2^{\det}$. By construction, $\eta_{1,\mathrm{new}}=\eta_1+\iota_\Q(b)$, and $\eta_{1,\mathrm{new}}|_p=\eta_1|_p+\iota_\Q(b)|_p$ is finite-flat.
\end{proof}

\begin{lem}
\label{lem: flat rho2 step 4}
Assume $a\up1|_{\ell_1}=0$. Then $\Pi_2^{\det,p}$ is non-empty. 
\end{lem}
\begin{proof}
Let $\rho_2 \in \Pi_2^{\det}$ be as in Lemma \ref{lem: flat rho2 step 3}. We claim that there is a cocycle $a \in \Z^1(\Z[1/Np],\F_p)$ such that $(a\upp2+a)|_p$ is unramified. Assume this claim, and let $\rho_{2,\mathrm{new}}=(a,-a,0,0) \cdot \rho_2$. As this action only changes $a\up2$ and $d\up2$, the representation $\eta_{1,\mathrm{new}}$ for $\rho_{2,\mathrm{new}}$ is identically equal to $\eta_{1}$, so it is finite-flat. The character $\chi_{2,\mathrm{new}}$ for $\rho_{2,\mathrm{new}}$ is given by
\[
\chi_{2,\mathrm{new}} = \left(1+ \ep a\upp1 + \ep^2(a\upp2+a)\right)|_p
\]
and it is unramified because $(a\upp2+a)|_p$ is unramified. By Lemma \ref{lem: flat rho2 step 2}, $\rho_{2,\mathrm{new}}|_p$ is finite-flat.

It remains to prove the claim that  there is a cocycle $a \in \Z^1(\Z[1/Np],\F_p)$ such that $(a\upp2+a)|_p$ is unramified. The fact that $\chi_2$ is a character implies
\[
-d a\upp2|_p = a\upp1 \smile a \upp1.
\]
This is the same as the coboundary of $\frac{1}{2}(a\upp1|_p)^2$, so the difference $a':=\frac{1}{2}(a\upp1|_p)^2-a\upp2|_p$
is a cocycle. Since the map
\[
Z^1(\Z[1/Np],\F_p) \to \frac{Z^1(\Q_p,\F_p)}{Z^1_\mathrm{un}(\Q_p,\F_p)}
\]
is surjective, there is $a \in Z^1(\Z[1/Np],\F_p)$ such that $a'-a|_p$ is unramified. Thus, since $a\upp1|_p$ is unramified, it follows that
\[
(a\upp2+a)|_p = \frac{1}{2}(a\upp1|_p)^2 +(a' -a|_p)
\]
is unramified.
\end{proof}

Let
\[
(Z^1_b)^\mathrm{flat} := Z^1_b \cap Z^1(\Z[1/Np], \F_p(1))^\mathrm{flat} \subset Z^1(\Z[1/Np], \F_p(1)). 
\]

\begin{prop}
\label{prop: b2 coordinates of flat rho2}
For every pair of elements of $\Pi_2^{\det,p}$, the difference between their $b\up2$-entries is contained in $(Z^1_b)^\mathrm{flat}$. Moreover, $(Z^1_b)^\mathrm{flat}$ is the span of $b\up1$ and $B^1(\Z[1/Np], \F_p(1))$. In particular, $b|_{\ell_0}=0$ for all $b \in (Z^1_b)^\mathrm{flat}$.
\end{prop}

\begin{rem}
In fact, the set of differences of the Proposition is equal to $(Z^1_b)^\mathrm{flat}$, but we have no need of this result.
\end{rem}

\begin{proof}
Let $\rho_2,\rho_{2,\mathrm{alt}} \in \Pi_2^{\det,p}$. By Lemma \ref{lem: flat rho2 step 2}, both $\eta_1\vert_p$ and $\eta_{1,\mathrm{alt}}\vert_p$ are finite-flat. 

As in the proof of Lemma \ref{lem: flat rho2 step 3}, and with the notation used there, $\eta_1$ and $\eta_{1,\mathrm{alt}}$ are identified with elements of $Z^1(\Z[1/Np],\Ad^0(\eta))^{p-UT}$. Since they have equal coordinates other than $b\upp2$, the difference $b\upp2-b_\mathrm{alt}\upp2$ is in $Z^1(\Z[1/Np], \F_p(1))$.
As $\eta_1\vert_p$ and $\eta_{1,\mathrm{alt}}\vert_p$ are finite-flat, the commutativity of the diagram in Sublemma \ref{sublem: eta adjoint} implies that $b\upp2-b_\mathrm{alt}\upp2 \in Z^1(\Z[1/Np], \F_p(1))^\fl$. On the other hand, Lemma \ref{P1: lem: exists 2nd-order 1-reducible}(2) implies that $b\upp2-b_\mathrm{alt}\upp2 \in Z^1_b$. This proves that $b\up2-b_\mathrm{alt}\up2 \in (Z^1_b)^\mathrm{flat}$, as desired.

When $x \in \F_p(1)$ is a basis, then $\{dx, b_p, b\up1, b_0\up1\}$ is a basis for $Z^1(\Z[1/Np], \F_p(1))$. By Lemma \ref{lem: global Kummer theory}, $\{dx,b\up1,b_0\up1\}$ is a basis for the subspace $Z^1(\Z[1/Np], \F_p(1))^\mathrm{flat}$. 
Since $b\up1\vert_{\ell_0} = 0$ and $dx|_{\ell_0}=0$, both $b\up1$ and $dx$ are in $(Z^1_b)^\fl$. But, by Lemma \ref{lem: log ell1 is zero} and Proposition \ref{prop: rk2 assumption}, $b_0\up1$ is not in $Z^1_b$, so $\{dx,b\up1\}$ is a basis for $(Z_b^1)^\fl$.
\end{proof}

The main result of this section, Proposition \ref{P1: prop: exists unique beta}, follows immediately from Lemma \ref{lem: flat rho2 step 4} and Proposition \ref{prop: b2 coordinates of flat rho2}.

\subsection{The $\US_N$ condition on $\rho_2$}
Finally, consider the subset $\Pi_2^{\US_N}$ of $\Pi_2^{\det,p}$ consisting of those $\rho_2$ which satisfying the $\US_N$ condition. This subset is cut out by local conditions as 
\[
\Pi_2^{\US_N} = \{ \rho_2 \in \Pi_2^{\det,p} \ | \ \rho_2\vert_{\ell_i} \text{ is }\US_{\ell_i} \text{ for } i = 0,1\}. 
\]
Indeed, the $\US_N$ condition is simply the combination of the finite-flat condition at $p$ along with the two $\US_{\ell_i}$ conditions at $\ell_i$, and the constant determinant condition actually follows from the $\US_N$ condition according to \cite[Prop.\ 3.8.3]{WWE5}.

Recall $\alpha \in \F_p(1)$ from Definition \ref{P1: defn: alpha}. If $a\up1|_{\ell_1}=0$, then $\beta \in \F_p(2)$ as in Proposition \ref{P1: prop: exists unique beta} is defined.
\begin{prop}
    \label{prop: exists rho_2}
    If $a\up1|_{\ell_1}=0$ and $\alpha^2+\beta=0$, then $\Pi_2^{\US_N} =  \Pi_2^{\det,p}$, and, in particular, $\Pi_2^{\US_N}$ is non-empty.
\end{prop}

\begin{proof}
By Proposition \ref{P1: prop: exists unique beta}, $\Pi_2^{\det,p}$ is non-empty. 
Let $\rho_2 \in \Pi_2^{\det,p}$. We will first show that, if $\alpha^2+\beta=0$, then $\rho_2\vert_{\ell_0}$ is $\US_{\ell_0}$. Let $\sigma,\tau \in G_{\ell_0}$. It suffices to show that
\[
(\rho_2(\sigma)-\omega(\sigma))(\rho_2(\tau)-1)
\]
is zero in $E_2$. Using the facts that $\omega|_{\ell_0}=1$ and $b\up1|_{\ell_0}=0$, this product, written in coordinates, is 
\begin{equation*}
\begin{split}
\resizebox{\columnwidth}{!}{
 $\displaystyle \ttmat{a\up1(\sigma)\ep + a\up2(\sigma)\ep^2}{b\up2(\sigma) \ep}{c\up1(\sigma) + c\up2(\sigma) \ep}{d\up1(\sigma) \ep + d\up2(\sigma)\ep^2} \cdot \ttmat{a\up1(\tau)\ep + a\up2(\tau)\ep^2}{b\up2(\tau) \ep}{c\up1(\tau) + c\up2(\tau) \ep}{d\up1(\tau) \ep + d\up2(\tau)\ep^2}$
 } 
 \\
=  \ttmat{(a\up1(\sigma)a\up1(\tau) + b\up2(\sigma)c\up1(\tau))\ep^2}
 {0}
 {(c\up1(\sigma)a\up1(\tau) +d\up1(\sigma)c\up1(\tau))\ep}
 {(c\up1(\sigma)b\up2(\tau)+ d\up1(\sigma)d\up1(\tau))\ep^2}.
 \end{split}
\end{equation*}

Using the equations that
\[
a\up1|_{\ell_0}= \alpha \smile c\up1|_{\ell_0}, \ d\up1|_{\ell_0} = - a\up1|_{\ell_0}, \ b\up2|_{\ell_0} = \beta \smile c\up1|_{\ell_0},
\]
all instances of $a\up1$, $d\up1$, and $b\up2$ can be replaced by appropriate multiples of $c\up1$, and the formula simplifies to
\[
(\rho_2(\sigma)-\omega(\sigma))(\rho_2(\tau)-1)= \ttmat{(\alpha^2+\beta)c\up1(\sigma)c\up1(\tau)\ep^2}
 {0}
 {0}
 {(\alpha^2+\beta)c\up1(\sigma)c\up1(\tau)\ep^2} 
\]
which vanishes by the assumption $\alpha^2+\beta=0$. This implies that $\rho_2\vert_{\ell_0}$ is $\US_{\ell_0}$.

It remains to show that $\rho_2\vert_{\ell_1}$ is $\US_{\ell_1}$. To do so, it will be convenient to change $c\up2$ by adding an element of $B^1(\Z[1/Np],\F_p(-1))$ to it. This amounts to conjugating $\rho_2$ by an element of $E_2^\times$, which does not affect whether the $\US_{\ell_1}$-condition holds.

Note that $\omega|_{I_{\ell_1}}=1$ and that  $a\up1|_{{\ell_1}}$, $d\up1|_{{\ell_1}}$, and  $c\up1|_{\ell_1}$ are zero. By equation (iii) in Lemma \ref{P1: lem: exists 2nd-order 1-reducible}, $c\up2|_{\ell_1}$ is a cocycle. Since $H^1(\Q_{\ell_1},\F_p(-1))$ vanishes, $c\up2\vert_{\ell_1}$ is a coboundary. Therefore, by adding an element of $B^1(\Z[1/Np],\F_p(-1))$ to $c\up2$ if necessary, we may and do assume $c\up2\vert_{\ell_1} = 0$. 

With this assumption, $\rho_2|_{\ell_1}$ can be written in coordinates as
\[
\rho_2|_{\ell_1} =  \ttmat{\omega \chi}
  {b\up1|_{\ell_1}+b\up2|_{\ell_1} \epsilon}
  {0}
  {\chi^{-1}} \\
\]
where $\chi = 1+a\up2|_{\ell_1} \ep^2 : G_{\ell_1} \to \F_p[\ep_2]^\times$ is a homomorphism. Since $\chi$ has order dividing $p$ and the pro-$p$-abelian quotient of $G_{\ell_1}$ is generated by Frobenius, $\chi$ is unramified.

Let $\sigma \in G_{\ell_1}$, $\tau \in I_{\ell_1}$. Then $\chi(\tau)=1$ and $\omega(\tau)=1$, so
\begin{align*}
(\rho_2(\sigma)-\omega&(\sigma))(\rho_2(\tau)-1)= \\
& \ttmat{\omega(\sigma)(\chi(\sigma)-1)}
{b\up1(\sigma)+b\up2(\sigma) \epsilon}
{0}
{\chi^{-1}(\sigma)-\omega(\sigma)}
\cdot \ttmat{0}
{b\up1(\tau)+b\up2(\tau) \epsilon}
{0}
{0}    
\end{align*}
This equals zero because $\omega(\sigma)(\chi(\sigma)-1) \in \ep^2\F_p[\ep_2]$, which annihilates the $B$-coordinate $\F_p[\ep_1]$. On the other hand
\begin{align*}
    (\rho_2(\tau)-\omega(\tau))&(\rho_2(\sigma)-1) = \\
& \ttmat{0}
{b\up1(\tau)+b\up2(\tau) \epsilon}
{0}
{0}
\cdot
\ttmat{\omega(\sigma)\chi(\sigma)-1}
{b\up1(\sigma)+b\up2(\sigma) \epsilon}
{0}
{\chi^{-1}(\sigma)-1}.
\end{align*}
This equals zero similarly, because $\chi^{-1}(\sigma)-1 \in \ep^2\F_p[\ep_2]$. Therefore, for all $\sigma, \tau \in G_{\ell_1} \times I_{\ell_1} \cup I_{\ell_1} \times G_{\ell_1}$,
\[
(\rho_2(\sigma)-\omega(\sigma))(\rho_2(\tau)-1)=0
\]
and so $\rho_2\vert_{\ell_1}$ is $\US_{\ell_1}$.
\end{proof}

\begin{cor}
    \label{cor: implying dim S is 4}
    Assume $a\up1|_{\ell_1}=0$ and $\alpha^2+\beta=0$. Then $\dim_{\F_p}R/pR>3.$
\end{cor}

\begin{proof}
By Proposition \ref{prop: exists rho_2}, $\Pi_2^{\US_N}$ is non-empty. Let $\rho_2 \in \Pi_2^{\US_N}$, let $D_2=\psi(\rho_2) : G_{\Q,Np} \to \F_p[\ep_2]$ be the associated pseudorepresentation, and let $\phi_2:R \to \F_p[\ep_2]$ be the local homomorphism induced by $D_2$ using the universal property of $R$. By construction, the composition of $\phi_2$ with the
quotient $\F_p[\ep_2] \onto \F_p[\ep_1]$ is the map $\phi_1$ of Lemma \ref{lem: D1 construction}. Since $\phi_1$ is surjective, this implies that $\phi_2$ is surjective.

Recall the notation $S=R/(p,\m^2)$ of Section \ref{subsec: S form}; since $S$ is a quotient of $R/pR$, it is enough to show that $\dim_{\F_p} S>3$. The map $\phi_2$ induces a surjective local-ring homomorphism $\phi_2: S \onto \F_p[\ep_2]$. Any element $y \in \m_S$ such that $\phi_2(y)=\ep$ has $\phi_2(y^2)=\ep^2 \ne 0$, so $y^2 \in \m_S^2$ is non-zero. Since $\dim_{\F_p}(R/(p,\m^2))=3$ by Proposition \ref{prop: R-cotangent}, this implies
\[
\dim_{\F_p} S = 3+\dim_{\F_p}(\m_S^2) > 3. \qedhere
\]
\end{proof}

Combining this corollary with Proposition \ref{prop: key alpha beta equation}, we can prove the main theorem.

\begin{thm}
    \label{P1: thm: main with a1}
    The $\F_p$-dimension of $R/pR$ is greater than $3$ if and only if
    \begin{enumerate}
        \item $a\up1 \vert_{\ell_1} = 0$ in $H^1(\Q_{\ell_1}, \F_p)$ and
        \item $\alpha^2 + \beta = 0$ in $\F_p(2)\vert_{\ell_0}$. 
    \end{enumerate} 
    Moreover, if $\dim_{\F_p} R/pR = 3$, the surjection $R \to \bT$ of Proposition \ref{prop: R to T} is an isomorphism of reduced finite flat $\Z_p$-algebras of rank 3. 
\end{thm}

\begin{proof}
We prove the final statement first. Suppose $\dim_{\F_p} R/pR = 3$. Since $R$ is complete and separated as a $\Z_p$-module, Nakayama's lemma implies that there is a surjection of $\Z_p$-modules $\Z_p^3 \onto R$. On the other hand, $\bT$ is a free $\Z_p$-module, and Ribet's Theorem \ref{thm: ribet} implies that $\rk_{\Z_p} \bT \geq 3$. The surjectivity of the composition
\[
\Z_p^3 \onto R \onto \bT
\]
then implies that $\rk_{\Z_p} \bT =3$, so the composition is an isomorphism. This implies that $R \onto \bT$ is an isomorphism.

Now we prove the first statement. One implication is immediate from Corollary \ref{cor: implying dim S is 4}. Conversely, assume $\dim_{\F_p}R/pR>3$. 
Since $\dim_{\F_p}R/(p,\m^2)=3$ by Proposition \ref{prop: R-cotangent}, this implies $\dim_{\F_p} S >3$ where $S=R/(p,\m^3)$. By Corollary \ref{cor: constraint on S}, this implies $\dim_{\F_p}S=4$. By Proposition \ref{prop: a1 vanish at ell1 when S large}, this implies (1).

It remains to show that (2) holds when $\dim_{\F_p}S=4$. Let $\rho_2 \in \Pi_2^{\US_N}$ be the element constructed in Corollary \ref{cor: rho_2 exists if dimS=4}. By the construction and by Proposition \ref{P1: prop: exists unique beta}, the element $\beta' \in \F_p(2)$ defined in Lemma \ref{lem: beta'} is equal to $\beta$. Since $\alpha^2+\beta'=0$ by Proposition \ref{prop: key alpha beta equation}, this implies (2).
\end{proof}

\section{The invariant $\alpha^2 + \beta$ is canonical}
\label{sec: invariant is canonical}

In this section, we prove that $\alpha^2 + \beta$ is a canonical element of $\mu_p^{\otimes 2}$, when it exists. That is, we will show that it does not depend on the pinning data of Definition \ref{P1: defn: pinning}. This improves upon Theorem \ref{P1: thm: main with a1}, which only implies that the \emph{vanishing} of $\alpha^2 + \beta$ is independent of the pinning data of Definition \ref{P1: defn: pinning}. 

\subsection{Formulation and outline of the proof}
First, we must make precise what we mean by ``$\alpha^2 + \beta$ is a canonical element of $\mu_p^{\otimes 2}$". Up until this point, we have defined $\alpha^2+\beta$ as an element of $\F_p(2)$, not of $\mu_p^{\otimes 2}$. Note the difference: $\mu_p \subset \barQ$ is the group of $p$th roots of unity, and a choice of primitive $p$th root of unity $\zeta \in \barQ$ defines an isomorphism
\begin{equation}
\label{eq:phi zeta}
    \phi_\zeta: \F_p(2) \to \mu_p^{\otimes 2}, \quad 1 \mapsto \zeta \otimes \zeta
\end{equation}
of $\F_p[G_\Q]$-modules. Since the pinning data includes a choice primitive $p$th root of unity $\zeta \in \barQ$, for any choice of pinning data, we have an element $\phi_\zeta(\alpha^2+\beta) \in \mu_p^{\otimes 2}$. When we say that $\alpha^2 + \beta$ is a canonical, we mean that the $\phi_\zeta(\alpha^2+\beta)$ is independent of the choice of pinning data.

\begin{thm}
    \label{thm: main invariant}
    Make Assumption \ref{P1: assump: main}.
    \begin{enumerate}
        \item The condition $a\up1\vert_{\ell_1} = 0$ does not depend on the pinning data. 
        \item 
        If $a\up1\vert_{\ell_1} = 0$, then there is an element $\delta \in \mu_p^{\otimes 2}$ such that, for each choice of pinning data, $\alpha^2+\beta = \phi_{\zeta}^{-1}(\delta)$, where $\zeta \in \barQ$ is the primitive $p$th root of unity given in the pinning data.
    \end{enumerate}
\end{thm}

We explain the main ideas of the proof.
The pinning data is mainly used in the paper in two ways: first, as a normalization factor to choose a particular Galois cohomology class that is only canonical up to scalar, and, second, to select cocycles within those normalized cohomology classes. Since $\alpha$ and $\beta$ and defined in terms of these cocycles, their values could depend on pinning data. However, in a sense, we can think of $\alpha$ and $\beta$ as being ``ratios" of pairs of cocycles, and we show that, for most changes to the pinning data that affect the normalization, both elements in the pair are scalar by the same factor, and, as a result, the ratios $\alpha$ and $\beta$ are unchanged. The only change of the pinning data that affects $\alpha^2+\beta$ is changing the primitive $p$th root of unity $\zeta$, and we show that change behaves as expected for 
an element of $\mu_p^{\otimes 2}$: changing $\zeta$ to $\zeta^a$ multiplies $\alpha^2+\beta$ by $a^{-2}$.

The second kind of effect of the pinning data is to change the choice of cocycle within a cohomology class. This kind of change amounts to conjugating the representations $\rho_1$ and $\rho_2$. By Lemma \ref{P1: lem: exists 2nd-order 1-reducible}, the condition $a\up1|_{\ell_1}=0$ can be interpreted in terms of the existence of a deformation $\rho_2$ of $\rho_1$, and this is unaffected by conjugation. In general, conjugation changes the value of $\alpha$ and $\beta$, but we show that the quantity $\alpha^2+\beta$ is left unchanged.

To prove the theorem, we analyze the effect of changing each datum independently. In Section \ref{subsec: conjugation computation}, we prove a general lemma about how $\alpha$ and $\beta$ change under conjugation. In each of the remaining parts, we focus on a single change to the pinning data, and compute its effect. We will use the following notation scheme:
\begin{itemize}
    \item We maintain the same notation $\alpha$, $\beta$, $\rho_1$, $\rho_2$, $a\up1$, $b\up1$, \dots, as in the earlier parts of the paper, computed with respect the pinning data fixed in Definition \ref{P1: defn: pinning}. Here $\rho_2$ denotes an arbitrary element of $\Pi_2^{\det,p}$, assuming it exists.
    \item We use primed notation $\alpha'$, $\beta'$, $\rho_1'$, $\rho_2'$,  $a\upp1$, $b\upp1$, \dots,  for the same objects computed with respect to the altered pinning data under consideration at the time.
\end{itemize}
In particular, be warned that the meaning of the primed objects is variable (and they also differ from the primed objects considered in Section \ref{P1: subsec: finite-flat on rho2}).

\subsection{Coordinate-wise calculation of conjugation of $\rho_1$ and $\rho_2$}
\label{subsec: conjugation computation} 
In this section, we compute the effect of conjugation on the representations $\rho_1$ and $\rho_2$ and their constituent cochains.

Let $M \in E_2^\times$ be an element of the form
\[
M= \ttmat{A_0 + A_1 \ep + A_2 \ep^2}{B_0 + B_1 \ep}{C_0 + C_1\ep}{1 + D_1\ep + D_2\ep^2}
\]
with $A_i,B_i,C_i,D_i \in \F_p$,
such that $\det(M)=A_0$, and also write $M$ as $M=\sm{A}{B}{C}{D}$. The usual formula for inverting a $2\times2$-matrix is valid in $E_2$:
\[
M^{-1} = A_0^{-1} \ttmat{D}{-B}{-C}{A}
\]
and for $N=\sm{a}{b}{c}{d} \in E_2$, the conjugation $M^{-1}NM$ is given by
\[
M^{-1}NM = A_0^{-1}\ttmat{ADa-\ep ABc + \ep CDb - \ep BCd}
{D^2b+BD(a-d)-\ep B^2c}
{A^2c + AC(d-a)-\ep C^2 b}
{ADd - \ep CDb -\ep BCa + \ep ABc}
\]

Let $M_1 \in E_1^\times$ be the image of $M$ under the map $E_2 \to E_1$, and let $\rho_{1,M}$ and $\rho_{2,M}$ denote the conjugates of $\rho_1$ and $\rho_2$ (if it exists)
\[
\rho_{1,M}(\sigma) = M_1^{-1}\rho_1(\sigma)M_1, \ \rho_{2,M}(\sigma) = M^{-1}\rho_2(\sigma) M.
\]
Write these in coordinates as
\[
\rho_{2,M} =\ttmat{\omega (1+a_M\up1\epsilon + a_M\up2\epsilon^2) }{b_M\up1 + b_M\up2 \epsilon}{\omega(c_M\up1 + c_M\up2 \epsilon)}{1+d_M\up1 \epsilon + d_M\up2\epsilon^2}
\]
and similarly for $\rho_{1,M}$. Using the explicit formula for conjugation above, we can express these new cochains in terms of the original ones.

\begin{lem}
\label{lem: conj formulas}
Let $M \in E_2^\times$ and $\rho_{1,M}$, $\rho_{2,M}$ be as above. Then
\begin{align}
    \label{eq: conj b1}
    b_M\up1 &= A_0^{-1}\left(b\up1 + B_0 (\omega-1)\right) \\
    \label{eq: conj c1}
    c_M\up1 &= A_0(c\up1 + A_0^{-1}C_0(\omega^{-1}-1) \\
    \label{eq: conj a1}
    a_M\up1 &= a\up1 - B_0 c\up1 + A_0^{-1}C_0 \omega^{-1}b\up1 - A_0^{-1}B_0C_0(\omega^{-1}-1) \\
         \label{eq: conj b2} 
     \begin{split}
b_M\up2 &= A_0^{-1}\Big(b\up2 + (D_1 + A_0^{-1}A_1 + B_0C_0)b\up1 \\
    &+ (B_1 - A_0^{-1}A_1B_0 + B_0^2 C_0)(\omega-1) + B_0(\omega a\up1 - d\up1) - B_0^2 \omega c\up1\Big).
    \end{split}
\end{align}
\end{lem}

Using this lemma, we can see how changes to the pinning data that cause $\rho_1$ and $\rho_2$ to be replaced by conjugates affect the values of $a\up1|_{\ell_1}$ and $\alpha^2+\beta$.

\begin{lem}
\label{lem:conjugation}
Consider a change to the pinning data that does not alter the decomposition group at $\ell_0$ or the primitive $p$th root of unity $\zeta$, and let $G_{\ell_1}'$ denote the decomposition group at $\ell_1$ that is part of this new data. Suppose that the representation $\rho_1'$ computed with this respect to this new data is of the form $\rho_1'=\rho_{1,M}$ for some $M \in E_2^\times$ as above. Then
\begin{enumerate}
    \item $a\upp1|_{G'_{\ell_1}}=0$ if and only if $a\up1|_{\ell_1}=0$.
    \item if $a\up1|_{\ell_1}=0$, then $\alpha'^2 +\beta' = A_0^{-2}(\alpha^2+\beta).$ 
\end{enumerate}
\end{lem}
\begin{proof}
For part (1), note that, by Lemma \ref{P1: lem: exists 2nd-order 1-reducible}, $a\up1|_{\ell_1}=0$ if and only if $\rho_1$ has a deformation $\rho_2$. If $a\up1|_{\ell_1}=0$, then $\rho_2$ exists, and $\rho_{2,M}$ gives a deformation of $\rho_1'=\rho_{1,M}$, so $a\upp1|_{G_{\ell_1}}=0$. This argument is symmetric, so the other implication follows.

Now suppose $a\up1|_{\ell_1}=0$, so $\rho_2$ exists and we can define $\rho_2'=\rho_{2,M}$, and $\beta$ and $\beta'$ are defined. Then $\alpha$ and $\beta$ are defined by the formulas
\[
a\up1|_{\ell_0} = \alpha \cup c\up1|_{\ell_0}, \ b\up2|_{\ell_0} = \beta \cup c\up1|_{\ell_0}
\]
and similarly for $\alpha'$ and $\beta'$. Since both pinning data have the same primitive $p$th root of unity $\zeta$, we can and do use $\zeta$ to identify twists of $\F_p$ with $\F_p$. In this way, we can think of $\alpha$ and $\beta$ as elements of $\F_p$ and think of this cup product as scalar multiplication.

Noting that $\omega|_{\ell_0}=1$, $d\up1|_{\ell_0}=-a\up1|_{\ell_0}$, and $b\up1|_{\ell_0}=0$, 
the formulas of Lemma \ref{lem: conj formulas} give
\begin{align*}
    a\upp1|_{\ell_0} &= a\up1|_{\ell_0} -B_0 c\up1|_{\ell_0} \\
    c\upp1|_{\ell_0} &= A_0 c\up1|_{\ell_0} \\
    b\upp2|_{\ell_0} & =A_0^{-1}( b\up2|_{\ell_0} +2B_0 a\up1|_{\ell_0} - B_0^2 c\up1|_{\ell_0}).
\end{align*}
Rearranging to write everything in terms of $c\upp1|_{\ell_0}$ gives
\begin{align*}
    a\upp1|_{\ell_0} & = A_0^{-1}(\alpha-B_0) c\upp1|_{\ell_0} \\
    b\upp2|_{\ell_0} & = A_0^{-2}(\beta +2 B_0\alpha -B_0^2) c\upp1.
\end{align*}
In other words,
\[
\alpha' = A_0^{-1}(\alpha-B_0),  \quad \beta' =A_0^{-2}(\beta +2 B_0\alpha -B_0^2)
\]
so $\alpha'^2 +\beta' = A_0^{-2}(\alpha^2+\beta)$, as desired.
\end{proof}

In the next section, we will see that, for any such change to the pinning data, we have $\rho_1'=\rho_{1,M}$ for an element $M$ with $A_0=\det(M)=1$. Then the lemma implies Theorem \ref{thm: main invariant} for these types of changes. Finally, we will deal with changes to the decomposition group at $\ell_0$ and changes to $\zeta$ by separate arguments.

\subsection{Changes that affect $\rho_1$ by conjugation}
In this section, we consider the changes to the pinning data of the types allowed in Lemma \ref{lem:conjugation}. We maintain the same notation $\rho_{1,M}$ as in the previous section. We will often rely on Definition \ref{P1: defn: pinned cocycles}, which describes how the cocycles $b_0\up1$, $b\up1$, $c\up1$, and $a_0$ as well as the elements $\gamma_0 \in I_{\ell_0}$ and $\gamma_1 \in I_{\ell_1}$ and the 0-cochain $x_{c\up1}$ are determined by the pinning data. We also frequently use Lemma \ref{P1: lem: produce a1}, which describes the cochain $a\up1$.

\begin{lem}[Change of decomposition group at $\ell_1$]
Let $G_{\ell_1}' \subset G_{\Q,Np}$ be another choice of decomposition group at $\ell_1$, and let $\rho_1'$ be the representation obtained by this change to the pinning data. Then $\rho_1' = \rho_{1,M}$ where
\[
M = \ttmat{1}{0}{C_0}{1}
\]
for some $C_0 \in \F_p$.

In particular, this change does not affect the condition $a\up1|_{\ell_1}=0$ and does not change the value of $\alpha^2+\beta$.
\end{lem}
\begin{proof}
The cocycle $b\up1$ and the class $c_0$ of the cocycle $c\up1$ do not depend on the choice of decomposition group at $\ell_1$, so $b\upp1=b\up1$ and $c\upp1-c\up1$ is a coboundary. Therefore 
\[
c\upp1 = c\up1+C_0(\omega^{-1}-1)
\]
for some $C_0 \in \F_p$. This implies
\[
x_{c\upp1}= x_{c\up1} +C_0.
\]
It remains to show that
\[
a\upp1 = a\up1 - C_0\omega^{-1}b\up1.
\]
This follows from Lemma \ref{P1: lem: produce a1}, as the defining properties (1)-(3) are easily checked with these values of $b\upp1$, $c\upp1$, and $x_{c\upp1}$. (Alternatively, properties (1) and (2) correspond to properties of the resulting map $\rho_1'$ (that it be a homomorphism and be finite-flat at $p$, respectively) that are unchanged by conjugation.) The last statement is clear from Lemma \ref{lem:conjugation}.
\end{proof}

Next we consider the choice of decomposition group at $p$ and the choice of root $\ell_1^{1/p}$ of $\ell_1$ (or equivalently, the choice of cocycle $b\up1$ in the class $b_1$). These cannot be considered completely independently because we insist that $b\up1|_p=0$ when $b_1|_p=0$ (note that the condition $b_1|_p=0$ is independent of the choice of decomposition group).

\begin{lem}[Change of decomposition group at $p$ and change of root $\ell_1^{1/p}$ of $\ell_1$]\label{lem: change pth root of ell1}
Let $G_{p}' \subset G_{\Q,Np}$ be a choice of decomposition group at $p$ and let $b\upp1$ be a choice of cocycle in the class $b_1$ of $b\up1$ that satisfies
\[
b\upp1|_{G_p'}=0
\]
if $b_1|_p=0$. Let $\rho_1'$ be the representation obtained by this change to the pinning data. Then $\rho_1' = \rho_{1,M}$ where
\[
M = \ttmat{1}{B_0}{0}{1}
\]
for some $B_0 \in \F_p$.

In particular, this change does not affect the condition $a\up1|_{\ell_1}=0$ and does not change the value of $\alpha^2+\beta$.
\end{lem}
\begin{proof}
The cocycles $b\upp1$ and $b\up1$ have the same class, so $b\upp1 = b\up1+B_0(\omega-1)$ for some $B_0 \in \F_p$. The cocycle $c\up1$ does not depend on the choice of decomposition group at $p$ or on the choice of root $\ell_1^{1/p}$ of $\ell_1$, so $c\upp1=c\up1$. It remains to show that
\[
a\upp1 = a\up1-B_0c\up1.
\]
This follows from Lemma \ref{P1: lem: produce a1} just as in the last lemma. The last statement is clear from Lemma \ref{lem:conjugation}.
\end{proof}

Changing the root $\ell_0^{1/p}$ of $\ell_0$ only changes the cocycle $b_0\up1$ and does not affect $b\up1$ or $c\up1$, and consequently does not change $a\up1$, $\alpha$, or $\beta$.

\subsection{Change of decomposition group at $\ell_0$}
Changing the decomposition group at $\ell_0$ changes the element $\gamma_0 \in I_{\ell_0}$ that is used to normalize $c\up1$. Hence it will scale $c\up1$ by a factor. However, the following lemma shows that it changes $a\up1$ and $b\up2$ by the same factor.

\begin{lem}\label{lem: change ell0}
Let $G_{\ell_0}' \subset G_{\Q,Np}$ be another choice of decomposition group at $\ell_0$, and let $\rho_1'$ be the representation obtained by this change to the pinning data. Then there is an element $A \in \F_p^\times$ such that
\begin{align*}
    & b\upp1 = b\up1  \\
    & c\upp1 = A c\up1 \\
    & a\upp1 = A a\up1.
\end{align*}
In particular, this change does not affect the condition $a\up1|_{\ell_1}=0$. 

If, moreover, $a\up1|_{\ell_1}=0$, then there is 
deformation $\rho_2' \in {\Pi'}_2^{\det,p}$ such that
\[
b\upp2 = A b\up2.
\]
In particular, this change does not alter the value of $\alpha$ or $\beta$.
\end{lem}
\begin{proof}
The cocycle $b\up1$ does not depend on the choice 
of decomposition group at $\ell_0$, so $b\upp1=b\up1$. Let $\sigma \in G_\Q$ be such that $G_{\ell_0}' = \sigma^{-1} G_{\ell_0} \sigma$.

A computation with cocycles\footnote{Note that this is the the same as the conjugation formula \eqref{eq: conj b1}, and can also be proven in the same way.}
shows that, for all $\tau \in G_{\ell_0}$:
\[
b_0\up1(\sigma^{-1} \tau \sigma) = \omega(\sigma)^{-1}\left(b_0\up1(\tau) + b_0\up1(\sigma)(\omega(\tau)-1)\right)
\]
In particular, letting $A_0=\omega(\sigma)$ and $\gamma_0' =\sigma^{-1} \gamma_0^{A_0} \sigma$, it follows that $b\up1(\gamma_0')=1$.

The cocycle $c\upp1$ is normalized so that $c\upp1(\gamma_0')=1$. Formula \eqref{eq: conj c1} applied with $M_1=\rho_1(\sigma)$ gives
\[
c\up1(\gamma_0')=A_0^2
\]
so $c\upp1=A_0^{-2}c\up1$. Letting $A=A_0^{-2}$, this gives $c\upp1=A c\up1$.

Now we claim that 
\[
a\upp1 = A a\up1.
\]
This follows from Lemma \ref{P1: lem: produce a1}. Finally, if $a\up1|_{\ell_1}=0$, then we claim that
\[
\rho_2' =  \ttmat{\omega (1+Aa\up1\epsilon + A^2a\up2\epsilon^2) }{b\up1 +A b\up2 \epsilon}{\omega(Ac\up1 + Ac\up2 \epsilon)}{1+Ad\up1 \epsilon + A^2 d\up2\epsilon^2}
\]
is in ${\Pi'}_2^{\det,p}$. In order to prove the claim, we apply the implication $(3) \Rightarrow (1)$ of Lemma \ref{lem: flat rho2 step 2}. The $\eta_1'$ produced from $\rho_2'$ via \eqref{eq: eta1 coordinates}, considered as an element of $Z^1(\Z[1/Np], \Ad^0(\eta'))$ via Lemma \ref{lem: basic def theory} where $\eta' = \eta_1 \pmod{\ep}$, has coordinates $A \cdot \sm{a\up1}{b\up2}{c\up1}{d\up1}$. Since the subset of finite-flat at $p$ lifts of $\eta'$ is a subspace containing $\sm{a\up1}{b\up2}{c\up1}{d\up1}$, it contains $\eta_1'$ as well. 

Finally, since $a\up1$, $c\up1$, and $b\up2$ are all scaled by the same factor, the values of $\alpha$ and $\beta$ are left unchanged.
\end{proof}

\subsection{Changing the root of unity}
Finally, we check that changing the root of unity alters $\alpha^2+\beta$ in the expected way.
\begin{lem}
\label{lem: change zeta}
Let $\zeta' \in \barQ$ denote another choice of primitive root of unity and let $A \in \F_p^\times$ be such that $\zeta = \zeta'^A$. Let $\rho'_1$ be the representation obtained by this change to the pinning data. Then 
\begin{align*}
    & b\upp1 = A b\up1 \\
    & c\upp1 = A c\up1 \\
    & a\upp1 = A^2 a\up1. 
\end{align*}
In particular, this change does not affect the condition $a\up1\vert_{\ell_1} = 0$.

If, moreover, $a\up1|_{\ell_1}=0$, then there is 
deformation $\rho_2' \in {\Pi'}_2^{\det,p}$ such that
\[
b\upp2 = A^3 b\up2.
\]
In particular, $\alpha'=A\alpha$ and $\beta'=A^2\beta$, and
\[
\phi_{\zeta'}(\alpha'^2 + \beta') = \phi_\zeta(\alpha^2+\beta).
\]
where $\phi_\zeta$ is as in \eqref{eq:phi zeta}. 
\end{lem}
\begin{proof}
Recall that $b\up1$ is defined by the equation
\[
\frac{\sigma \ell_1^{1/p}}{\ell_1^{1/p}} = \zeta^{b\up1(\sigma)}
\]
for all $\sigma \in G_{\Q,Np}$.
Replacing $\zeta$ by $\zeta'^A$, it follows that $b\upp1 = Ab\up1$. Similarly $b_0\upp1 = A b_0\upp1$.

The cocycle $c\upp1$ is a scalar multiple of $c\up1$, normalized such that $c\upp1(\gamma_0')=1$ where $\gamma_0'
 \in I_{\ell_0}$ satisfies $b_0\upp1(\gamma_0')=1$.
Since $b_0\up1(\gamma_0)=1$ and $b_0\upp1=A b_0\up1$, we can choose $\gamma_0'=\gamma_0^{A^{-1}}$. Given that $c\up1(\gamma_0)=1$, this shows that $c\upp1=Ac\up1$.

The fact that $a\upp1=A^2a\up1$ follows immediately from Lemma \ref{P1: lem: produce a1}. Similarly, it is easy to see that $b\upp2 = A^3 b\up2$ satisfies differential equation (ii) in Proposition \ref{prop: exists rho_2}, and the fact that the resulting $\rho_2$ is finite-flat is clear.

The equations $\alpha'=A\alpha$ and $\beta'=A^2\beta$ follow immediately from the definitions and, since $\phi_{\zeta}=A^2\phi_{\zeta'}$, this shows that
\[
\phi_{\zeta'}(\alpha'^2+\beta')=\phi_{\zeta'}(A^2(\alpha^2+\beta))=\phi_{\zeta}(\alpha^2+\beta). \qedhere
\]
\end{proof}

\bibliographystyle{alpha}
\bibliography{CWEbib-2023-RR3}

\end{document}